\documentclass{amsart}
\usepackage{amssymb, amsmath, amsthm}
\usepackage[dvipsnames]{xcolor}
\usepackage{hyperref}
\usepackage{stmaryrd}
\usepackage{mathrsfs}
\usepackage[all]{xy}
\usepackage{tikz-cd}
\usepackage{cleveref}
\usepackage{leftidx}
\usepackage{colonequals}

\newtheorem{prop}{Proposition}[section]

\newtheorem{lem}[prop]{Lemma}
\newtheorem{cor}[prop]{Corollary}
\newtheorem{thm}[prop]{Theorem}
\theoremstyle{definition}
\newtheorem{defi}[prop]{Definition}
\theoremstyle{remark}
\newtheorem{examp}[prop]{Example}

\newtheorem{remar}[prop]{Remark}

\DeclareMathAlphabet{\mathpzc}{OT1}{pzc}{m}{it}

\DeclareMathOperator{\End}{End}
\DeclareMathOperator{\Hom}{Hom}

\DeclareMathOperator{\Res}{Res}

\DeclareMathOperator{\GL}{GL}

\DeclareMathOperator{\Ker}{Ker}

\DeclareMathOperator{\Gal}{Gal}

\DeclareMathOperator{\Irr}{Irr}

\DeclareMathOperator{\Spec}{Spec}

\DeclareMathOperator{\codim}{codim}

\DeclareMathOperator{\Sp}{Sp}
\DeclareMathOperator{\Rep}{Rep}

\DeclareMathOperator{\Def}{Def}
\DeclareMathOperator{\Set}{Set}

\newcommand{\Qp}{\mathbb {Q}_p}
\newcommand{\Zp}{\ZZ_p}
\newcommand{\Qpbar}{\overline{\mathbb{Q}}_p}

\newcommand{\MM}{\mathscr M}

\newcommand{\ZZ}{\mathbb Z}

\newcommand{\QQ}{\mathbb Q}
\newcommand{\Aa}{\mathfrak A}

\newcommand{\Fp}{\mathbb F_p}

\newcommand{\RR}{\mathbb R}
 
\newcommand{\mm}{\mathfrak m}

\newcommand{\OO}{\mathcal O}

\DeclareMathOperator{\wtimes}{\widehat{\otimes}}

\newcommand{\nn}{\mathfrak n}

\newcommand{\br}[1]{\llbracket #1\rrbracket}

\newcommand{\ana}{\mathrm{an}}

\newcommand{\adm}{\mathrm{adm}}
\newcommand{\alg}{\mathrm{alg}}

\newcommand{\Htilde}{\widetilde{H}^0}

\newcommand{\rhobar}{\overline{\rho}}
\newcommand{\Frob}{\mathrm{Frob}}

\newcommand{\Gm}{\mathbb G_m}

\newcommand{\CH}{\mathrm{CH}}
\newcommand{\gen}{\mathrm{gen}}

\DeclareMathOperator{\Spf}{Spf}
\newcommand{\rig}{\mathrm{rig}}
\newcommand{\red}{\mathrm{red}}
\newcommand{\univ}{\mathrm{univ}}

\newcommand{\irr}{\mathrm{irr}}

\newcommand{\tw}{\mathrm{tw}}

\newcommand{\cyc}{\mathrm{cyc}}

\newcommand{\kappabar}{\overline{\kappa}}
\DeclareMathOperator{\Lie}{Lie}

\newcommand{\lin}{\mathrm{lin}}

\newcommand{\kbar}{\bar{k}}

\newcommand{\PC}{\mathrm{PC}}
\newcommand{\cPC}{\mathrm{cPC}}

\newcommand{\hyphen}{\text{-}}

\newcommand{\disc}{\operatorname{disc}}

\newcommand{\Thetabar}{\overline{\Theta}}

\newcommand{\Cont}{\mathcal C}

\newcommand{\Psibar}{\overline{\Psi}}

\DeclareMathOperator{\Group}{Group}

\newcommand{\Cp}{\mathbb C_p}

\newcommand{\ps}{\mathrm{ps}}
\newcommand{\Ghat}{\widehat{G}}
\newcommand{\LG}{\leftidx{^L}{G}{}}

\newcommand{\CG}{\leftidx{^C}G{}}

\newcommand{\ghat}{\widehat{\mathfrak g}}
\newcommand{\that}{\widehat{\mathfrak t}}

\newcommand{\Sen}{\mathrm{Sen}}

\newcommand{\grho}{\leftidx{^g}{\rho}{}}

\newcommand{\Rig}{\mathrm{Rig}}
\newcommand{\op}{\mathrm{op}}
\newcommand{\Ber}{\mathrm{Ber}}
\newcommand{\Ohat}{\widehat{\OO}}
\newcommand{\sing}{\mathrm{sing}}
\newcommand{\ver}{\mathrm{ver}}
\newcommand{\Sigmahat}{\widehat{\Sigma}}
\newcommand{\Spl}{\mathrm{Spl}}
\newcommand{\gl}{\mathfrak{gl}}
\newcommand{\la}{\mathrm{la}}

\title{Infinitesimal characters and Lafforgue's pseudocharacters}
\author{Vytautas Pa\v{s}k\={u}nas and  Julian Quast }
\date{\today.}
\begin{document}

\begin{abstract} 
We associate infinitesimal characters to $p$-adic families of Lafforgue's pseudocharacters of the absolute Galois group of a $p$-adic local field by extending a construction of Dospinescu, Schraen, and the first author. We use this construction to study the action of the centre of the universal enveloping algebra on the locally analytic vectors in the Hecke eigenspaces of the completed cohomology.
\end{abstract}
\maketitle

\tableofcontents

\section{Introduction}
Let $\Gamma_F$ be the absolute Galois group of a $p$-adic local field 
$F$. Let $G$ be a connected reductive group defined over $F$ and
let $\rho: \Gamma_F\rightarrow \LG(\Qpbar)$ be a continuous admissible
representation, where $\LG$ is the $L$-group of $G$. Dospinescu, Schraen and the first 
author in \cite{DPS} have attached to $\rho$ a $\Qp$-algebra homomorphism  $\zeta_{\rho}: Z(\Res_{F/\Qp} \mathfrak g)\rightarrow \Qpbar$, 
where $\Res_{F/\Qp} \mathfrak g$ is the Lie algebra of the restriction of scalars $\Res_{F/\Qp}G$ and 
$Z(\Res_{F/\Qp} \mathfrak g)$ is the centre of
its enveloping algebra. The homomorphism $\zeta_{\rho}$ is defined using Sen theory and encodes  information about the Hodge--Tate weights of the Galois representation $\rho$. The point of view taken in \cite{DPS} is 
that the construction $\rho \mapsto \zeta_{\rho}$ 
is an infinitesimal shadow of a hypothetical $p$-adic 
Langlands correspondence. Namely to $\rho$ there should
correspond a packet of admissible unitary $p$-adic Banach space 
representations of $G(F)$. For any representation $\Pi$ in the 
packet corresponding to $\rho$ the expectation is that the 
action of $Z(\Res_{F/\Qp} \mathfrak g)$ on the locally analytic vectors in $\Pi$ is given by $\zeta_{\rho}$. Although the
$p$-adic Langlands correspondence has not yet been defined at this level of generality, in \cite{DPS}  the authors have verified that the expectation is true when $\Pi$ 
is a Hecke eigenspace in the completed cohomology in many cases. 

A key result proved in \cite{DPS} and used in the application to the Hecke eigenspaces is that the construction $\rho\mapsto \zeta_{\rho}$ behaves well in $p$-adic families. 
Namely, if $\rho: \Gamma_F\rightarrow \LG(A)$
is an admissible continuous representation, where  $A$ is an affinoid algebra
then there is a $\Qp$-algebra homomorphism 
$\zeta_{\rho}: Z(\Res_{F/\Qp} \mathfrak g)\rightarrow A$, compatible with specialisations at closed points $x:A \rightarrow \Qpbar$, so that 
$ (\zeta_{\rho})_x = \zeta_{\rho_x}$. Moreover, $\zeta_{\rho}$ is functorial in $A$, which allows instead of $\Sp A$ to consider more general rigid analytic spaces. 

In this paper we refine the results of \cite{DPS} by considering 
$p$-adic families of Lafforgue's pseudocharacters  instead 
of Galois representations. 
If $\rho: \Gamma_F\rightarrow \LG(\Qpbar)$ 
is a continuous representation then $\zeta_{\rho}$ depends 
only on the $\LG$-semisimplification of $\rho$, which in turn 
depends only on the associated $\LG$-pseudocharacter
$\Theta_{\rho}$. 
Conversely, every continuous admissible $\Qpbar$-valued  $\LG$-pseudocharacter $\Theta$ corresponds to a $\Ghat(\Qpbar)$-conjugacy class of continuous admissible $\LG$-semisimple representations $\rho: \Gamma_F\rightarrow \LG(\Qpbar)$ and hence 
we may define $\zeta_{\Theta}:= \zeta_{\rho}$. 

The main result 
of this paper is that the construction $\Theta\mapsto \zeta_{\Theta}$ varies well in $p$-adic families.
To formulate it we let  $\Rig_{\Qp}$ be the category of rigid analytic spaces over $\Qp$ and let 
 $\tilde{X}^{\adm}_{\LG}: \Rig_{\Qp}^{\op}\rightarrow \Set$ be the functor, which 
maps a rigid analytic space $Y$ to 
the set of continuous admissible $\LG$-pseudocharacters 
$\Theta: \Gamma_F\rightarrow \LG(\OO(Y))$. The admissibility for 
$\LG$-pseudocharacters of $\Gamma_F$, defined in \Cref{adm_pseudo}, is a condition
analogous to admissibility for representations of
$\Gamma_F$ into $L$-groups. In particular, if $G$ is 
split over $F$ then according to our conventions $\LG$ is equal to the dual group $\Ghat$ and every $\LG$-pseudocharacter is admissible. 

\begin{thm}[\Cref{main_L}]\label{intro_main_L} There exists a unique  family of $\Qp$-algebra 
homomorphisms $$\zeta_{Y,\Theta}: Z(\Res_{F/\Qp} \mathfrak g)\rightarrow \OO(Y)$$ indexed by pairs $Y\in \Rig_{\Qp}$ and 
$\Theta\in \tilde{X}^{\adm}_{\LG}(Y)$ with the following properties:
\begin{enumerate} 
\item (functoriality) for every morphism $\varphi: Y'\rightarrow Y$ in $\Rig_{\Qp}$ we have $$\varphi^*(\zeta_{Y, \Theta})= \zeta_{Y', \varphi^* \Theta};$$
\item (specialisation) for every $Y\in \Rig_{\Qp}$ and every 
$y\in Y(\Qpbar)$ the specialisation 
$$\zeta_{y, \Theta}:  Z(\Res_{F/\Qp} \mathfrak g)
\overset{\zeta_{Y,\Theta}}{\longrightarrow} \OO(Y)\overset{y}{\longrightarrow} \Qpbar$$
is equal to $\zeta_{\rho}$, where $\rho: \Gamma_F \rightarrow 
\LG(\Qpbar)$ is any continuous representation with 
$\Theta_{\rho}= \Theta\otimes_{\OO(Y), y} \Qpbar$. 
\end{enumerate}
\end{thm}

In \Cref{main_C} we prove a version of \Cref{intro_main_L} for $C$-admissible $\CG$-pseudo\-characters of $\Gamma_F$, where
$\CG$ is the $C$-group of $G$.

\subsection{Strategy of the proof}
Let us explain the strategy of the proof of \Cref{intro_main_L}. Using the results of J.Q. \cite{quast}, which generalise the results of 
Chenevier \cite{che_durham} when $G=\GL_n$, we show that the functor $\tilde{X}^{\adm}_{\LG}$ is representable by 
a quasi-Stein rigid analytic space $X^{\adm}_{\LG}$. 
Moreover, $X^{\adm}_{\LG}$ is 
disjoint union of rigid analytic spaces of 
the form $(\Spf R^{\ps}_{\Thetabar})^{\rig}$, 
where $\Thetabar$ is a $k$-valued continuous 
admissible $\LG$-pseudocharacter, where $k$ is a 
finite extension of $\Fp$, and $R^{\ps}_{\Thetabar}$
is the universal deformation ring of $\Thetabar$. 
Let $(R^{\ps}_{\Thetabar})^{\ana}$ be the ring of 
global sections of $(\Spf R^{\ps}_{\Thetabar})^{\rig}$. To prove the existence part of \Cref{intro_main_L} it is enough to construct $\zeta_{Y, \Theta}$ for $Y= (\Spf R^{\ps}_{\Thetabar})^{\rig}$ and $\Theta=\Theta^{u}\otimes _{R^{\ps}_{\Thetabar}} (R^{\ps}_{\Thetabar})^{\ana}$, where $\Theta^{u}$  is the universal deformation
of $\Thetabar$, such that $\zeta_{Y, \Theta}$ satisfies the specialisation property. To prove the uniqueness part of \Cref{intro_main_L} it is 
enough to prove that $(\Spf R^{\ps}_{\Thetabar})^{\rig}$ is reduced, which is equivalent to the ring $R^{\ps}_{\Thetabar}[1/p]$ being reduced. 

Let us explain how we prove that $R^{\ps}_{\Thetabar}[1/p]$ is normal, which implies that it is reduced.  Let us replace 
$\Gamma_F$ by an arbitrary finitely generated profinite group $\Gamma$ and $\LG$ by a generalised reductive $\OO$-group scheme $H$. If $A$ is an $R^{\ps}_{\Thetabar}$-algebra we let $X^{\gen}_{\Thetabar}(A)$ be the 
set of $R^{\ps}_{\Thetabar}$-condensed representations $\rho: \Gamma \rightarrow H(A)$ such that its $H$-pseudocharacter $\Theta_{\rho}$ is equal to the specialisation of $\Theta^u$ at $A$. 
The condition \textit{$R^{\ps}_{\Thetabar}$-condensed} is a  continuity condition on $\rho$ expressed in an algebraic way. We have shown in \cite{defG} 
that $X^{\gen}_{\Thetabar}$ is representable by a finitely generated 
$R^{\ps}_{\Thetabar}$-algebra $A^{\gen}_{\Thetabar}$. We let $X^{\ps}_{\Thetabar}:=\Spec R^{\ps}_{\Thetabar}$.
\begin{thm}[\Cref{adequate}]\label{intro_adequate} For any finitely generated\footnote{In \cite{MRL} we prove the assertion for profinite groups $\Gamma$ satisfying Mazur's finiteness condition at $p$ by reducing to the case where $\Gamma$ is finitely generated.} 
profinite group $\Gamma$ and any generalised reductive 
$\OO$-group scheme $H$ the natural map $R^{\ps}_{\Thetabar}\rightarrow A^{\gen}_{\Thetabar}$ induces an isomorphism $R^{\ps}_{\Thetabar}[1/p]\rightarrow (A^{\gen}_{\Thetabar}[1/p])^{H^0}$ and hence a finite adequate homeomorphism 
$X^{\gen}_{\Thetabar} \sslash H^0 \rightarrow X^{\ps}_{\Thetabar}$. 
\end{thm}

In \cite{defG} we have already shown that $R^{\ps}_{\Thetabar}\rightarrow (A^{\gen}_{\Thetabar})^{H^0}$
is finite and induces a universal homeomorphism on spectra. Thus the new result is the isomorphism after inverting $p$. This is proved by building on the results of Emerson--Morel \cite{emerson2023comparison}, 
who proved an analogous result, when $\Gamma$ is an abstract group and $X^{\gen}_{\Thetabar}$ is 
replaced by the scheme of representations 
of $\Gamma$ into $H$. If $\Gamma=\Gamma_F$ then we have shown in \cite{defG} that $A^{\gen}_{\Thetabar}[1/p]$ is normal. The normality of $R^{\ps}_{\Thetabar}[1/p]$ follows from this result and \Cref{intro_adequate}.

Let us now explain how we attach an infinitesimal character $\zeta$ to $\Theta^u$. The most naive idea would be to try and construct a representation 
$\rho: \Gamma_F \rightarrow \LG(R^{\ps}_{\Thetabar})$ such that $\Theta_{\rho}= \Theta^u$ and then 
take $\zeta= \zeta_{\rho^{\ana}}$ the infinitesimal character constructed in \cite{DPS}, where $\rho^{\ana}= \rho\otimes_{R^{\ps}_{\Thetabar}} (R^{\ps}_{\Thetabar})^{\ana}$. This naive idea does not work as there are non-trivial families of reducible representations, which have constant pseudocharacter.  
However, we are in much better shape if we restrict to the absolutely
irreducible locus, denoted by the superscript $\irr$ below.  We are motivated by results of Chenevier in 
\cite{che_durham} who constructed a representation with values 
in the units of an Azumaya algebra over the irreducible 
locus of $X^{\ps}_{\Thetabar}$, when  $H=\GL_n$. 
We show for an arbitrary generalised reductive $\OO$-group scheme $H$ and $\Gamma=\Gamma_F$:  
\begin{prop}[\Cref{nice_locus}, \Cref{much_ado}]\label{intro_nice_locus}  There is an open  subscheme $U$ of $X^{\ps}_{\Thetabar}[1/p]$
and an open subscheme $V$ of $X^{\gen}_{\Thetabar}[1/p]$ such that the  following hold:
\begin{enumerate}
\item $U\subset (X^{\ps}_{\Thetabar}[1/p])^{\irr}$ and $V\subset (X^{\gen}_{\Thetabar}[1/p])^{\irr}$;
\item $U$ and $V$ are both regular;
\item $\pi|_V:V\rightarrow U$ is smooth and surjective;
\item $\dim X^{\ps}_{\Thetabar}[1/p] - \dim (X^{\ps}_{\Thetabar}[1/p]\setminus U) \ge \min(2, [F:\Qp])$;
\item 
$\dim X^{\gen}_{\Thetabar}[1/p] - \dim (X^{\gen}_{\Thetabar}[1/p]\setminus V)\ge \min(2, [F:\Qp])$.
\end{enumerate} 
 \end{prop} 
Above $\pi: X^{\gen}_{\Thetabar} \rightarrow X^{\ps}_{\Thetabar}$ is the map that sends a representation to its $H$-pseudocharacter. It follows from part (3) that $\pi|_V$ has sections \'etale locally on $U$. We obtain an \'etale covering $\{U_i\rightarrow U\}$ by affine schemes $U_i$ and representations $\rho_i\in X^{\gen}_{\Thetabar}(U_i)$ for all
$i\in I$. We let $U_i^{\ana}\in \Rig_{\Qp}$ be the analytification of $U_i$ and let $\rho_i^{\ana}= 
\rho_i\otimes_{\OO(U_i)}\OO(U_i^{\ana})$. If $H=\LG$ then 
we show that the infinitesimal characters $\zeta_{\rho_i^{\ana}}$ constructed in \cite{DPS}
glue to an infinitesimal character 
$\zeta: Z(\Res_{F/\Qp} \mathfrak g) \rightarrow \OO(U^{\ana})$. 

Normality of 
$R^{\ps}_{\Thetabar}[1/p]$ implies  that 
$(\Spf R^{\ps}_{\Thetabar})^\rig$ is normal.
If $F\neq \Qp$ then 
it follows from part (4) that the complement of $U^{\ana}$ in $(\Spf R^{\ps}_{\Thetabar})^{\rig}$ has codimension  at least $2$, thus $\OO(U^{\ana})= (R^{\ps}_{\Thetabar})^{\ana}$ by Riemann's Hebbarkeitssatz. We thus constructed $\zeta:Z(\Res_{F/\Qp} \mathfrak g) \rightarrow (R^{\ps}_{\Thetabar})^{\ana}$. The construction of $\zeta$, when $F=\Qp$, is reduced to this case by 
using the fact that the Sen operator does not 
change if we restrict to the Galois group of a finite extension. 

Let $\rho^u: \Gamma_F \rightarrow \LG(A^{\gen}_{\Thetabar})$ be the universal representation over 
$X^{\gen}_{\Thetabar}$. Let $(A^{\gen}_{\Thetabar})^{\ana}$ be the 
ring of global functions of the analytification $(X^{\gen}_{\Thetabar})^{\ana}$ and let $\rho^{u,\ana}$ be the representation obtained from $\rho^u$  by 
extending scalars to $(A^{\gen}_{\Thetabar})^{\ana}$. We show in \Cref{main}, 
using normality of $X^{\gen}_{\Thetabar}[1/p]$, that $\zeta_{\rho^{u,\ana}}$
takes values in $(R^{\ps}_{\Thetabar})^{\ana}$ and is equal to $\zeta$.
This implies the specialisation property in \Cref{intro_main_L}.

If $G=\GL_n$ then $\LG=\GL_n$ and all the 
ingredients to prove  \Cref{intro_main_L} are available in \cite{che_durham} and 
\cite{BIP_new} - this has also been  observed independently by Benjamin Schraen.  One difficulty we have to 
deal with when defining $U$ in \Cref{intro_nice_locus}  in a more general setting is that 
the centralisers of absolutely irreducible representations need not be connected, and this yields an obstruction to smoothness of $\pi$.

\subsection{Application}
The result allows us to drop the hypothesis 
that the residual representation is irreducible in \cite[Section 9.9]{DPS}, where the authors show  that the subspace of locally analytic vectors of a Hecke eigenspace
in the completed cohomology of a definite unitary group 
affords an infinitesimal character and compute this infinitesimal character 
in terms of the associated Galois representation. The key issue, that we resolve in \Cref{sec_Hecke}, is that the automorphic theory leads to a determinant law valued in a big Hecke algebra and not directly to a Galois representation itself, 
which means that the results of \cite{DPS} cannot be  applied.
Although the application uses 
only the results, when $\LG=\GL_n$, we expect that \Cref{intro_main_L} in the more general setting will 
be useful in the emerging categorical local Langlands correspondence in the locally analytic setting \cite[Section 6.2]{padic_LL}. 

\subsection{Overview by section} In \Cref{laf} we recall the definition of Lafforgue's pseudocharacters and prove \Cref{intro_adequate}. In \Cref{sec_abs_irr}
we prove \Cref{intro_nice_locus}. In \Cref{sec_rel_an} we discuss the analytification of schemes locally of finite type over $\Spec R[1/p]$, where $R$ is a complete local noetherian $\Zp$-algebra. We 
prove that for such a scheme $X$ the functor 
$U\mapsto \OO(U^{\ana})$ is a sheaf for the \'etale topology on $X$. We expect that the results are known to the experts. In \Cref{sec_constr} we carry out the construction of the infinitesimal character and prove \Cref{intro_main_L}. In \Cref{sec_Hecke} we apply our results to study 
locally analytic vectors in Hecke eigenspaces in the setting of definite unitary groups. Don't read Sections \ref{sec_dim_codim} and \ref{section_todd} unless you really want to. 

\subsection{Notation} We let $F$ be a finite extension of $\Qp$ and fix an algebraic closure $\overline{F}$ of $F$. We let $\Gamma_F=\Gal(\overline{F}/F)$ be 
the absolute Galois group of $F$. We let $G$ be a connected reductive group over $F$ split over a finite Galois extension $E$ of $F$ contained in $\overline{F}$. 
We fix a further finite extension $L$ of $\Qp$ - our field of coefficients, and 
let $\OO$ be its ring of integers and $k$ the residue field. Typically in our arguments we have the freedom 
to replace $L$ by a finite extension. We denote by $\LG$ (resp. $\CG$) the $L$-group of 
$G$ (resp. the $C$-group of $G$), the definition is recalled
in \Cref{adm_pseudo} (resp. \Cref{sec_mod_C}). 
These are smooth affine group schemes of finite type over $\ZZ$, and according to our conventions depend on the choice of $E$. (The theorems we prove do not.) In our main theorem we work with $\LG$- and $\CG$-pseudocharacters of $\Gamma_F$. In Sections \ref{laf} and \ref{sec_abs_irr} we prove intermediate results for $\Gamma$ any finitely generated profinite group (instead of $\Gamma_F$) and $H$ any generalised $\OO$-group scheme 
(instead of $\LG$ or $\CG$). In \Cref{sec_Hecke} we are in a global setting and the notation changes completely, for example $F$ is a totally real number field. If $R$ is a ring then 
$R\text{-}\alg$ denotes the category of $R$-algebras. If $X$ is a scheme (or a rigid analytic space) then  both $\OO(X)$ and $\OO[X]$ denote the ring of its global sections.

\subsection{Acknowledgements} We would like to thank 
Ulrich G\"ortz for several stimulating discussions and Brian Conrad, Antoine Ducros  and Klaus K\"unnemann for correspondence regarding Berkovich spaces. We would like to thank Toby Gee and Carl Wang-Erickson for their comments. The paper builds in an essential way on 
the work \cite{DPS} by Gabriel Dospinescu, Benjamin Schraen and V.P.. We thank Gabriel Dospinescu for a detailed list of comments on an earlier draft.  We are grateful to 
Jelena Ivan\v{c}i\'c for raising the natural
question of whether our construction of infinitesimal characters is compatible with
the construction in \cite{DPS} for families of
Galois representations. We answer this question
positively in \Cref{compatibility_L} for
$L$-groups and in \Cref{compatibility_C} for
$C$-groups. We would like thank the anonymous
referee for their careful reading of the paper. 

The research of J.Q. was funded by the Deutsche Forschungsgemeinschaft (DFG, German Research Foundation) – project number 517234220. 

\subsection{AI usage note}
The authors used ChatGPT (GPT-6 Astra) to assist with manuscript proofreading. The tool identified minor  technical errors, which were subsequently reviewed, verified, and manually corrected by the authors. The authors maintain full responsibility for the mathematical content of this paper.
\section{Dimension and codimension}\label{sec_dim_codim}
If $X$ is a topological space its \textit{dimension} $\dim X$ is defined as the supremum of lengths of chains of 
closed irreducible subspaces, \cite[\href{https://stacks.math.columbia.edu/tag/0055}{Tag 0055}]{stacks-project}. If $Z$ is a closed irreducible subset of $X$ then the \textit{codimension  of $Z$ in $X$}, denoted by $\codim(Z, X)$,  is defined as the supremum of lengths 
of closed irreducible subsets in $X$ starting with $Z$, \cite[\href{https://stacks.math.columbia.edu/tag/02I3}{Tag 02I3}]{stacks-project}. It follows from the definition that 
\begin{equation}\label{dim_codim}
\dim Z+ \codim(Z, X)\le \dim X.
\end{equation} 
We say that \textit{the dimension formula holds for $X$} if \eqref{dim_codim} is an equality for all  closed irreducible subsets $Z$ of $X$.
If $X$ is an irreducible scheme  of finite type over a field 
then the dimension formula holds by \cite[Proposition 5.30]{GW1}. However, it may fail for general schemes as the following example shows.
\begin{examp} Let $X=\Spec \Zp[x]$. The ideal $(px-1)$ 
is maximal in $\Zp[x]$, as its residue field is $\Qp$. 
We let $Z$ be the closed point in $X$ defined by this ideal. Then $\dim Z=0$, $\dim X=2$ and  $\codim(Z, X)=1$, as the localisation of $\Zp[x]$ at $(px-1)$ 
coincides with the localisation of $\Qp[x]$ at $(x-p^{-1})$. 
\end{examp}
The purpose of this note is to show that \eqref{dim_codim} is an equality for all irreducible $Z$ if $X$ is equidimensional locally of finite 
type over $S:=(\Spec R)\setminus \{\mm\}$, where 
$R$ is a complete local noetherian ring with a maximal 
ideal $\mm$. Sometimes $S$ is called the \textit{punctured spectrum} of $R$. The result will be used in \Cref{sec_abs_irr} with $R=R^{\ps}_{\Thetabar}$ and $X=X^{\gen}_{\Thetabar}[1/p]$ with the notation of \Cref{laf}.

\begin{lem}\label{main_dim_codim} Let $X$ be an integral scheme, locally of finite type over $S$. Then $\dim \OO_{X,x}=\dim X$ for 
all closed points $x$ of $X$. 
\end{lem}
\begin{proof} 
We first note that $S$ is Jacobson by  \cite[\href{https://stacks.math.columbia.edu/tag/02IM}{Tag 02IM}]{stacks-project}. It follows from the
Cohen structure theorem for complete local noetherian rings that $R$ is a quotient of a regular ring. Thus $S$ is universally catenary by part (2) of \cite[\href{https://stacks.math.columbia.edu/tag/02JB}{Tag 02JB}]{stacks-project}.

Since $R$ is a local ring, the only Zariski open subset of $\Spec R$ containing the closed point is $\Spec R$ 
itself. It follows from 
\cite[\href{https://stacks.math.columbia.edu/tag/02IC}{Tag 02IC}]{stacks-project}
and its proof 
that $$\delta: \Spec R \rightarrow \ZZ, 
\quad y \mapsto \codim( \{\mm\}, Z_i)-\codim(\overline{\{y\}}, Z_i),$$ 
where $Z_i$ is 
any irreducible component of $\Spec R$ containing $y$, 
is a dimension function in the sense of \cite[\href{https://stacks.math.columbia.edu/tag/02I9}{Tag 02I9}]{stacks-project}. A point $y\in S$ is closed 
in $S$ if and only if its closure in $\Spec R$ is 
equal to $\{y, \mm\}$. It follows from the 
definition of dimension function that this is 
equivalent to $\delta(y)=\delta(\mm)+1=1$.
Thus $\delta': S \rightarrow \ZZ$
defined by $\delta'(y)=\delta(y)-1$ is a dimension 
function on $S$ satisfying $\delta'(y)=0$ for all closed points $y$. The assertion follows from \cite[\href{https://stacks.math.columbia.edu/tag/02QO}{Tag 02QO}]{stacks-project} applied with $Z=X$. 
\end{proof} 

\begin{lem} If $f: X\rightarrow S$ is locally of finite type
then $X$ is Jacobson and universally catenary. Moreover, if $x$ is a closed point of $X$ then $f(x)$ is a closed point of $S$ and 
the extension $\kappa(x)/\kappa(f(x))$ is finite. 
\end{lem}
\begin{proof} We have shown in the proof of \Cref{main_dim_codim} that $S$ is universally catenary and Jacobson. The assertion of the lemma then follows from
\cite[\href{https://stacks.math.columbia.edu/tag/02J5}{Tag 02J5}]{stacks-project} and
\cite[\href{https://stacks.math.columbia.edu/tag/01TB}{Tag 01TB}]{stacks-project}.
\end{proof} 

\begin{lem}\label{dim_form} Let $X$ be an equidimensional scheme, 
locally of finite type over $S$. Then the dimension
formula holds for $X$. 
\end{lem} 
\begin{proof} If $x$ is a closed point of $X$ and $Y$ is an irreducible component of $X$ containing $x$ then it follows from \Cref{main_dim_codim} that $\dim \OO_{Y,x}= \dim Y =\dim X$, as $X$ is equidimensional by assumption. Thus $\OO_{X,x}$ is equidimensional 
and $\dim \OO_{X,x} = \dim X$. Since $S$ is universally catenary, $X$ and the local ring $\OO_{X,x}$ are catenary. It follows from \cite[Lemma 2.4]{heinrich}
that $X$ is biequidimensional and \cite[Proposition 2.3]{heinrich} implies that the dimension formula holds for $X$.
\end{proof}

\begin{remar}\label{useful_codim_dim} Let $X$ be an equidimensional scheme, of finite type over $S$. 
Let $Z$ be a closed subset of $X$. Then $X$, and hence $Z$, are noetherian. Thus $Z$ has only finitely many 
irreducible components $Z_1, \ldots, Z_r$. We have 
$\dim Z= \max_i \dim Z_i$. Thus if we define 
$\codim(Z, X)= \min_i \codim( Z_i, X)$ then \Cref{dim_form} implies that 
$\dim Z + \codim(Z, X)= \dim X$. 
\end{remar}

\section{Chevalley--Shephard--Todd theorem}\label{section_todd}
Let $V$ be a finite dimensional vector space over a field $K$ and let $\Delta$ be a finite subgroup of $\GL(V)$. 
We assume that the order of $\Delta$ is invertible in $K$. 
Recall that $g\in \GL(V)$ is a \textit{pseudoreflection} 
if $g-1$ has rank $1$. The action of $\Delta$ on $V$ induces 
an action on the symmetric algebra $S(V)$. The subring 
of $\Delta$-invariants $S(V)^{\Delta}$ is a graded $K$-subalgebra
of $S(V)$. 
\begin{thm}[Chevalley--Shephard--Todd]\label{CST} The following are equivalent: 
\begin{enumerate}
\item $\Delta$ is generated by pseudoreflections;
\item $S(V)^{\Delta}$ is a polynomial algebra over $K$;
\item $S(V)$ is a free $S(V)^{\Delta}$-module.
\end{enumerate}
\end{thm}
\begin{proof} \cite[chap.V, \S 5, no.5, Theorem 4]{Bourbaki_Lie46}.
\end{proof} 

The main result of this section 
\Cref{main_todd}  adds a further equivalent condition, which is used in \Cref{sec_luna}.
Let $P(V)$ be the completion of $S(V)$ at the maximal ideal
$\bigoplus_{n\ge 1} S^n(V)$. Then $P(V)$ is isomorphic
to a formal power series ring in $d:=\dim_K V$ variables. We denote 
its maximal ideal by $\mm$.
The action of $\Delta$ on $V$ induces an action on $P(V)$ 
by continuous ring automorphisms. The subring of $\Delta$-invariants is a closed local subring of $P(V)$ with maximal ideal $\nn:=\mm \cap P(V)^{\Delta}= \mm^{\Delta}$. Since $\Delta$ is a finite group 
the map $P(V)^{\Delta}\rightarrow P(V)$ is finite. This implies that $\nn P(V)$ contains $\mm^k$ for some $k\ge 1$ and hence the $\nn$-adic topology on $P(V)^{\Delta}$ coincides with the subspace topology. Moreover, $\nn^2P(V)$ 
contains $\mm^{2k}$ and hence $\nn^2$ contains $\nn\cap \mm^{2k}$, which implies
that $\nn/\nn^2$ is a finite dimensional 
$K$-vector space, and hence $P(V)^{\Delta}$ is a complete  local  noetherian 
$K$-algebra with residue field $K$. 

\begin{lem}\label{grading} The ideal $\nn$ is generated by 
 homogeneous elements $f_1, \ldots, f_m \in S(V)^{\Delta}$, where $m=\dim_K \nn/\nn^2$.  Moreover, any such family generates
$S(V)^{\Delta}$ as a $K$-algebra.
\end{lem}
\begin{proof} Let $g_1, \ldots, g_m \in \nn$ 
such that their images form a $K$-basis for $\nn/\nn^2$. Let $t_1, \ldots, t_m\in \oplus_{i\le 2k-1} S^i(V)$ 
be such that   $t_i \equiv g_i \pmod{\mm^{2k}}$ for $1\le i \le m$. 
Since the action of $\Delta$ preserves the  degree, $t_i$ are $\Delta$-invariant, 
and since $\nn^2$ contains $\nn\cap \mm^{2k}$, 
$t_i\equiv g_i \pmod{\nn^2}$. Thus $\nn=(t_1, \ldots, t_m)$ by Nakayama's lemma. If we write $t_i=\sum_{j=1}^{2k-1}
t_{ij}$ with $t_{ij}\in S^j(V)$ then each $t_{ij}$ is fixed by $\Delta$, and hence
$\nn$ is generated by $t_{ij}$ for $1\le i\le m$ 
and $1\le j \le 2k-1$. The image of this set  
in $\nn/\nn^2$ spans it as a $K$-vector space
and thus contains a basis. A subset of $\{t_{ij}\}$ that maps bijectively onto the basis 
generates $\nn$ by Nakayama's lemma.

Let us assume that $f_1, \ldots, f_m \in S(V)^{\Delta}$ with $f_i\in S^{d_i}(V)$ generate $\nn$ as an ideal. 
Since $\dim_K \nn/\nn^2=m$ and $P(V)^{\Delta}$ is a complete local $K$-algebra the map 
\begin{equation}\label{mappy}
K\br{x_1,\ldots, x_m}\twoheadrightarrow P(V)^{\Delta}, \quad F\mapsto F(f_1,\ldots, f_m)
\end{equation}
is surjective. Let  $f\in S(V)^{\Delta}$ be homogeneous of degree $n$. Then 
$f\in P(V)^{\Delta}$ and hence there is $F\in K\br{x_1,\ldots, x_m}$ such that $f= F(f_1,\ldots, f_m)$ in $P(V)$. We define a grading on $K[x_1,\ldots, x_m]$ by letting 
$\deg(x_i)=d_i$ for $1\le i \le m$. 
We may write $F=\sum_{i=0}^{\infty} F_i$ with $F_i\in K[x_1,\ldots, x_m]$ homogeneous of degree $i$. Since $F_i(f_1,\ldots, f_m)\in S^i(V)$ for all $i\ge 0$, we deduce that $f=F_n(f_1,\ldots, f_m)$. Hence, $f_1, \ldots, f_m$ generate
$S(V)^{\Delta}$ as a $K$-algebra.
\end{proof}

\begin{lem}\label{main_todd} The conditions in \Cref{CST} are equivalent to: 
\begin{itemize}
\item[(4)] $P(V)^{\Delta}$ is formally smooth over $K$. 
\end{itemize}
\end{lem}
\begin{proof} We observe that since 
$P(V)^{\Delta}\rightarrow P(V)$ is finite, the Krull 
dimension of $P(V)^{\Delta}$ is 
equal to $d$ by \cite[Corollary A.8]{BH}. Hence, 
 (4) is equivalent to the existence
of an isomorphism $P(V)^{\Delta}\cong K\br{x_1, \ldots, x_d}$ as local $K$-algebras, which is equivalent to $\dim_K \nn/\nn^2=d$. 

If (4) holds then by \Cref{grading} there
exist $f_1,\ldots, f_d\in S(V)^{\Delta}$
which generate $\nn$ as an ideal and $S(V)^{\Delta}$ as a $K$-algebra. Moreover, $d=\dim_K \nn/\nn^2$.
The map \eqref{mappy} is an isomorphism 
for  dimension reasons. It maps $K[x_1,\ldots, x_d]$ onto the subalgebra 
of $S(V)^{\Delta}$ generated by $f_1,\ldots, f_d$ and hence induces an 
isomorphism between $K[x_1,\ldots, x_d]$
and $S(V)^{\Delta}$. Hence, (2) holds. 

Conversely, if $S(V)^{\Delta}$ is a polynomial algebra over $K$, then its completion with respect to the ideal $\bigoplus_{n\ge 1} S^n (V)^{\Delta}$ is isomorphic to 
$K\br{x_1,\ldots, x_d}$, as this holds for any maximal ideal 
of $K[x_1,\ldots, x_d]$ with residue field $K$. The completion coincides with the 
closure of $S(V)^{\Delta}$ inside $P(V)$ and hence is equal to $P(V)^{\Delta}$. Thus (4) holds. 
\end{proof}

\section{Lafforgue's pseudocharacters}\label{laf}

In this section we recall the definition of pseudocharacters and the moduli space of condensed representations. The main result of the section is \Cref{adequate} and its \Cref{normal}. Until \Cref{sec_pseudo_def}  we let $\OO$ be any commutative ring.

A \emph{generalised reductive group} $H$ over $\OO$ is a 
smooth affine group scheme, that has reductive geometric fibres and such that $H/H^0$ is finite, where $H^0$ is the neutral component of $H$. For each $n\ge 1$ the group scheme $H$ acts on the $n$-fold product $H^n$ by $g \cdot (g_1, \dots, g_n) = (gg_1g^{-1}, \dots, gg_ng^{-1})$. This induces an algebraic action of $H$ on the affine coordinate ring $\OO[H^n]$ of $H^n$. The submodule $\OO[H^n]^{H^0} \subseteq \OO[H^n]$ is defined as the algebraic invariant module of the $H^0$-representation $\OO[H^n]$. It is an $\OO$-subalgebra, since $H$ acts by $\OO$-linear automorphisms.
In subsequent sections, we will take $H$ to be an $L$-group or a $C$-group of a $p$-adic group.

\begin{defi}\label{LafPC} Let $\Gamma$ be an abstract group and let $A$ be a commutative $\OO$-algebra. An \emph{$H$-pseudocharacter} $\Theta$ of $\Gamma$ over $A$ is a sequence $(\Theta_n)_{n \geq 1}$ of $\OO$-algebra maps
$$\Theta_n : \OO[H^n]^{H^0} \to \mathrm{Map}(\Gamma^n,A)$$ for $n \geq 1$, satisfying the following conditions:
\begin{enumerate}
    \item For each $n,m \geq 1$, each map $\zeta : \{1, \dots, m\} \to \{1, \dots,n\}$, $f \in \OO[H^m]^{H^0}$ and $\gamma_1, \dots, \gamma_n \in \Gamma$, we have
    $$ \Theta_n(f^{\zeta})(\gamma_1, \dots, \gamma_n) = \Theta_m(f)(\gamma_{\zeta(1)}, \dots, \gamma_{\zeta(m)}) $$
    where $f^{\zeta}(g_1, \dots, g_n) = f(g_{\zeta(1)}, \dots, g_{\zeta(m)})$.
    \item For each $n \geq 1$, for each $\gamma_1, \dots, \gamma_{n+1} \in \Gamma$ and each $f \in \OO[H^n]^{H^0}$, we have
    $$ \Theta_{n+1}(\hat f)(\gamma_1, \dots, \gamma_{n+1}) = \Theta_n(f)(\gamma_1, \gamma_2, \dots, \gamma_{n-1}, \gamma_n\gamma_{n+1}) $$
    where $\hat f(g_1, \dots, g_{n+1}) = f(g_1, g_2, \dots, g_{n-1}, g_ng_{n+1})$.
\end{enumerate}
\end{defi}
We denote the set of $H$-pseudocharacters of $\Gamma$ over $A$ by $\PC_H^{\Gamma}(A)$.
If $f : A \to B$ is a homomorphism of $\OO$-algebras, then there is an induced map $f_* : \mathrm{PC}_{H}^{\Gamma}(A) \to \mathrm{PC}_{H}^{\Gamma}(B)$.
For $\Theta \in \mathrm{PC}_{H}^{\Gamma}(A)$, the image $f_*(\Theta)$ is called the \emph{specialisation} of $\Theta$ along $f$ and is denoted by $\Theta \otimes_A B$.
It is easy to verify that specialisation along $f : A \to B$ commutes with composition with a homomorphism $\varphi : H \to H'$, i.e. $(\varphi \circ \Theta) \otimes_A B = \varphi \circ (\Theta \otimes_A B)$.

When $\Gamma$ is a topological group, $A$ is a topological ring and for all $n \geq 1$, the map $\Theta_n$ has image in the set $\Cont(\Gamma^n, A)$ of continuous maps $\Gamma^n \to A$, we say that $\Theta$ is \emph{continuous}. We denote the set of continuous $H$-pseudocharacters of $\Gamma$ with values in $A$ by $\cPC_H^{\Gamma}(A)$.

The functor $A \mapsto \PC^{\Gamma}_H(A)$ is representable by an $\OO$-algebra \cite[Theorem 3.19]{quast}.
Let us write $\Rep^{\Gamma}_H(A) := \Hom(\Gamma, H(A))$.
The functor $A \mapsto \Rep^{\Gamma}_H(A)$ is also representable by an $\OO$-algebra.
To a representation $\rho : \Gamma \to H(A)$, we associate an $H$-pseudocharacter $\Theta_{\rho} \in \PC^{\Gamma}_H(A)$ by the formula
$$ \Theta_{\rho, m}(f)(\gamma_1, \dots, \gamma_m) := f(\rho(\gamma_1), \dots, \rho(\gamma_m))$$ 
for all $m \geq 1$, all $f \in \OO[H^m]^{H^0}$ and all $\gamma_1, \dots, \gamma_m \in \Gamma$.
This defines a morphism of schemes $\Rep^{\Gamma}_H \to \PC^{\Gamma}_H$ by
\begin{equation}\label{rep_to_pc}
    \Rep^{\Gamma}_H(A) \to \PC^{\Gamma}_H(A), \quad \rho \mapsto \Theta_{\rho} 
\end{equation}

\begin{lem}\label{lit_defT}
    If $H^0 \subseteq Z(H)$, then for all $\OO$-algebras $A$ the map \eqref{rep_to_pc} is bijective.
\end{lem}

\begin{proof}
    This is \cite[Lemma 6.3 (1)]{defT}.
    Our assumption implies that $H^0$ is a connected commutative reductive group scheme and that the action of $H^0$ on $\Rep^{\Gamma}_H$ by conjugation is trivial. By \cite[Proposition 6.2]{defT} the map $\Rep^{\Gamma}_H \to \PC^{\Gamma}_H$ induces an isomorphism $\Rep^{\Gamma}_H \sslash H^0 \cong \PC^{\Gamma}_H$, but $\Rep^{\Gamma}_H \sslash H^0 = \Rep^{\Gamma}_H$ and the claim follows.
\end{proof}

\subsection{Comparison to the GIT quotient} 
Let $\Gamma$ be an abstract group. The map \eqref{rep_to_pc} is $H^0$-equivariant for the conjugation action on the source and the 
trivial action on the target. Thus it induces a morphism $\Rep^{\Gamma}_H \sslash H^0\rightarrow \PC^{\Gamma}_H$.

\begin{thm}[Emerson--Morel]\label{EM_dis}
    If $\OO$ is a field of characteristic $0$, then the natural map $\Rep^{\Gamma}_H \sslash H^0 \to \PC^{\Gamma}_H$ is an isomorphism.
\end{thm}

\begin{proof}
If $H$ is connected then the statement is proved in \cite[Proposition 2.11 (i)]{emerson2023comparison}.  We have adapted their argument to the 
case, when $H$ is possibly disconnected, $H^0$ is a torus and $\OO$ is any commutative ring in \cite[Proposition 6.2]{defT}. The key point is the exactness of  $H^0$-invariants. The argument in \cite[Proposition 6.2]{defT} applies verbatim, when $\OO$ is a field of characteristic zero and $H$ is possibly  disconnected, as then $H^0$ is linearly reductive. 
\end{proof}

\subsection{\texorpdfstring{Deformations of $H$-pseudocharacters}{Deformations of H-pseudocharacters}}\label{sec_pseudo_def} From now on $\OO$ is 
the ring of integers in a finite extension $L$ of $\Qp$ with residue field $k$.
Let $\Gamma$ be a profinite group. Consider a continuous $H$-pseudocharacter $\Thetabar \in \cPC^{\Gamma}_H(k)$.
In \cite[Section 5]{quast} we have introduced a deformation problem, that we will recall here. Let $\Aa_{\OO}$ be the category of artinian local $\OO$-algebras with residue field $k$.
We define the \emph{deformation functor} of $\Thetabar$
$$ \Def_{\Thetabar} : \Aa_{\OO} \to \Set, ~A \mapsto \{\Theta \in \cPC_H^{\Gamma}(A) \mid \Theta \otimes_A k = \Thetabar\} $$
that sends an object $A \in \Aa_{\OO}$ to the set of continuous $H$-pseudocharacters $\Theta$ of $\Gamma$ over $A$ with $\Theta \otimes_A k = \Thetabar$.
It is pro-representable by a complete local $\OO$-algebra $R^{\ps}_{\Thetabar}$ with residue field $k$.
If $\Gamma$ is topologically finitely generated, then $R^{\ps}_{\Thetabar}$ is noetherian \cite[Theorem 5.7]{quast}.
We will add a superscript $\Gamma$ and write $\Def^{\Gamma}_{\Thetabar}$, $R^{\ps, \Gamma}_{\Thetabar}$ to emphasise that  we are working with pseudocharacters of $\Gamma$, and will drop the superscript if the context is clear. 
Denote by $\Theta^u \in \cPC^{\Gamma}_H(R^{\ps}_{\Thetabar})$ the universal deformation of $\Thetabar$. If 
$A$ is an $R^{\ps}_{\Thetabar}$-algebra we will write $\Theta^u_{|A}$ for 
the specialisation of $\Theta^u$ along $R^{\ps}_{\Thetabar} \rightarrow A$. We let $X^{\ps, \Gamma}_{\Thetabar}:=\Spec R^{\ps}_{\Thetabar}$.

\subsection{Chenevier's determinant laws and finitely generated groups}

We first consider the case, that $\Gamma = \widehat \Sigma$ is the profinite completion of a finitely generated group $\Sigma$ and that $H = \GL_d$.
Recall, that $\Theta^u$ corresponds to a \emph{continuous determinant law} 
$D : R^{\ps}_{\Thetabar}\br{\Sigmahat} \to R^{\ps}_{\Thetabar}$ by
\cite[Theorem 4.1 (ii)]{emerson2023comparison},
see \cite[Sections 1.5,  2.30]{che_durham} for a definition of continuous determinant laws.

We denote by $\CH(D) \subseteq R^{\ps}_{\Thetabar}\br{\Sigmahat}$ the \emph{Cayley--Hamilton ideal} of $D$ introduced in \cite[Section 1.17]{che_durham}.
The Cayley--Hamilton ideal $\CH(D)$ is contained in the kernel of $D$, hence $D$ descends to a determinant law on the quotient $R^{\ps}_{\Thetabar}\br{\Sigmahat}/\CH(D)$ \cite[Lemma 1.21]{che_durham} which we will also denote by $D$, and we will call the quotient together with this determinant law the \emph{Cayley--Hamilton quotient}. We say, that a determinant law is \emph{Cayley--Hamilton}, if its Cayley--Hamilton ideal is zero. The Cayley--Hamilton quotient enjoys the following universal property: for all $R^{\ps}_{\Thetabar}$-algebra homomorphisms $f : R^{\ps}_{\Thetabar}\br{\Sigmahat} \to S$ and all Cayley--Hamilton determinant laws $D' : S \to R^{\ps}_{\Thetabar}$, such that $D'|_{R^{\ps}_{\Thetabar}\br{\Sigmahat}} = D$, we have $\CH(D) \subseteq \ker(f) \subseteq \ker(D)$ and $f$ descends to a unique map $R^{\ps}_{\Thetabar}\br{\Sigmahat}/\CH(D) \to S$, such that the restriction of $D'$ to the Cayley--Hamilton quotient gives back $D$.

Via restriction along the map 
$R^{\ps}_{\Thetabar}[\Sigma] \to R^{\ps}_{\Thetabar}\br{\Sigmahat}$ we obtain a determinant law $D|_{R^{\ps}_{\Thetabar}[\Sigma]} : R^{\ps}_{\Thetabar}[\Sigma] \to R^{\ps}_{\Thetabar}$.
Similarly, we can consider the Cayley--Hamilton ideal $\CH(D|_{R^{\ps}_{\Thetabar}[\Sigma]})$.
We observe, that the Cayley--Hamilton quotients coincide:

\begin{lem}\label{ident_CH_quotients}
    The natural map $R^{\ps}_{\Thetabar}[\Sigma]/\CH(D|_{R^{\ps}_{\Thetabar}[\Sigma]}) \to R^{\ps}_{\Thetabar}\br{\Sigmahat}/\CH(D)$ induced by the universal property of the Cayley--Hamilton quotient is an isomorphism.
\end{lem}

\begin{proof}
    By \cite[Proposition 2.13]{WE_alg} $E \colonequals R^{\ps}_{\Thetabar}[\Sigma]/\CH(D|_{R^{\ps}_{\Thetabar}[\Sigma]})$ is a finitely generated $R^{\ps}_{\Thetabar}$-module. Thus $E/\mm^n E$ is finite as a set, where $\mm$ is the maximal ideal of $R^{\ps}_{\Thetabar}$, and hence $\Sigma\rightarrow (E/\mm^n E)^{\times}$ factors through a finite quotient. We deduce that $R^{\ps}_{\Thetabar}[\Sigma]\twoheadrightarrow E$ extends  to a surjection $R^{\ps}_{\Thetabar}\br{\Sigmahat} \twoheadrightarrow E$.
    Now $E$ is equipped with the determinant law $D_E$ induced by $D|_{R^{\ps}_{\Thetabar}[\Sigma]}$.
    Both $D_E|_{R^{\ps}_{\Thetabar}\br{\Sigmahat}}$ and $D$ restrict to $D|_{R^{\ps}_{\Thetabar}[\Sigma]}$.
    Since the image of $R^{\ps}_{\Thetabar}[\Sigma]$ is dense in $R^{\ps}_{\Thetabar}\br{\Sigmahat}$, this implies $D_E|_{R^{\ps}_{\Thetabar}\br{\Sigmahat}} = D$.
    The universal property of the Cayley--Hamilton quotient yields a surjection $R^{\ps}_{\Thetabar}\br{\Sigmahat}/\CH(D) \twoheadrightarrow E$, whose composition with the natural map $E \to R^{\ps}_{\Thetabar}\br{\Sigmahat}/\CH(D)$ is the identity.
\end{proof}

\subsection{\texorpdfstring{The moduli space $X^{\gen}_{\Thetabar}$ of condensed representations}{The moduli space X gen Thetabar of condensed representations}}
Let $\Gamma$ be an arbitrary topologically finitely generated profinite group and let $R$ be a profinite noetherian $\OO$-algebra. 
For an $R$-algebra $A$ a representation $\rho : \Gamma \to H(A)$ is \emph{$R$-condensed}, if there exists a homomorphism of condensed groups $$\underline{\Gamma} \to H(A_{\disc} \otimes_{R_{\disc}} \underline{R})$$ which recovers $\rho$ after evaluation at a point.
Here $\underline{(-)}$ denotes the functor from topological spaces to condensed sets and $(-)_{\disc}$ takes a set to a discrete condensed set.
See \cite[Section 4]{defG} for a detailed discussion of $R$-condensed representations. One should think 
of $R$-condensed as a continuity condition on the representation. 

\begin{defi}\label{defi_paper}
    Let $A$ be an $R^{\ps}_{\Thetabar}$-algebra.
    Let $X^{\gen, \Gamma}_{\Thetabar}(A)$ be the set of $R^{\ps}_{\Thetabar}$-condensed representations $\rho : \Gamma \to H(A)$, such that $\Theta_{\rho}$ is equal to the specialisation of $\Theta^u$ at $A$.
    This defines a functor $X^{\gen, \Gamma}_{\Thetabar} : R^{\ps}_{\Thetabar}\hyphen\alg \to \Set$.
\end{defi}

In \cite[Proposition 8.3]{defG}, we have shown, that $X^{\gen, \Gamma}_{\Thetabar}$ is representable by a finitely generated $R^{\ps}_{\Thetabar}$-algebra $A^{\gen, \Gamma}_{\Thetabar}$. We will drop the superscript $\Gamma$ and will write $X^{\gen}_{\Thetabar}$, $A^{\gen}_{\Thetabar}$, when the context is clear. 

\subsection{Profinite completions}
We now assume $\Gamma = \Sigmahat$, where $\Sigma$ is a finitely generated discrete group.
In this case we have an alternative description of $X^{\gen, \Sigmahat}_{\Thetabar}$ in terms of abstract representations of $\Sigma$.

\begin{lem}\label{Agen_versus_abstract}
    Let $A$ be an $R^{\ps, \Sigmahat}_{\Thetabar}$-algebra.
    Every representation $\rho : \Sigma \to H(A)$ with $\Theta_{\rho} = (\Theta^u_{|A})|_{\Sigma}$ extends to a unique $R^{\ps, \Sigmahat}_{\Thetabar}$-condensed representation $\widetilde \rho : \Sigmahat \to H(A)$ with $\Theta_{\widetilde \rho} = \Theta^u_{|A}$. In particular, we have a canonical isomorphism
    \begin{equation}\label{iso_Qhat}
    R^{\ps, \Sigmahat}_{\Thetabar} \otimes_{\OO[\PC^\Sigma_H]} \OO[\Rep^\Sigma_H] \cong A^{\gen, \Sigmahat}_{\Thetabar} 
    \end{equation}
    of $R^{\ps, \Sigmahat}_{\Thetabar}$-algebras.
\end{lem}

\begin{proof}  We choose a closed embedding $\tau : H \to \GL_d$ for some $d \geq 1$. We have $\Theta_{\rho} = (\Theta^u_{|A})|_\Sigma$, so $\det(\tau \circ \rho) = D_{\tau \circ \Theta^u|_\Sigma} \otimes_{R^{\ps}} A$, where $R^{\ps}:= R^{\ps, \Sigmahat}_{\Thetabar}$. 
    Hence the linearisation $(\tau \circ \rho)^{\lin} :
    R^{\ps}[\Sigma] \to M_d(A)$ factors over the Cayley--Hamilton quotient
    $E = R^{\ps}[\Sigma]/\CH(D|_{R^{\ps}[\Sigma]}) \to M_d(A)$.
    The isomorphism of \Cref{ident_CH_quotients} induces a natural
    surjection of $R^{\ps}$-algebras $R^{\ps}\br{\Sigmahat} \twoheadrightarrow E$. The equivalence of parts (1) and (4) of \cite[Lemma 5.2]{defG} 
    implies that $\tau\circ \rho: 
    \Sigmahat \rightarrow \GL_d(A)$ is $R^{\ps}$-condensed. Part (2) of \cite[Lemma 4.16]{defG} implies that $\rho: \Sigmahat\rightarrow H(A)$ is $R^{\ps}$-condensed.  The isomorphism of $R^{\ps}$-algebras is immediate from the definition of $A^{\gen, \Sigmahat}_{\Thetabar}$ in \Cref{defi_paper}. 
\end{proof}

\begin{lem}\label{tensor_inv2} Let  $M$ and $N$ be modules over  an $L$-algebra $A$. 
Suppose that $H^0_L$ acts $A$-linearly on $M$ and trivially on $N$. 
Then the natural map $ M^{H^{0}}\otimes_A N\rightarrow (M\otimes_A N)^{H^0}$ is an isomorphism of $A$-modules. 
\end{lem} 
\begin{proof} The assertion is clear if $N$ is a free $A$-module. 
The general case follows by choosing a presentation 
$\oplus_{j\in J} A\rightarrow \oplus_{i\in I} A\rightarrow N\rightarrow 0$ and using that $H^0_L$ is linearly reductive.
\end{proof} 

\begin{cor}\label{cor_iso_Qhat} The isomorphism \eqref{iso_Qhat} induces an isomorphism 
\begin{equation} 
R^{\ps, \Sigmahat}_{\Thetabar}[1/p] \cong (A^{\gen, \Sigmahat}_{\Thetabar}[1/p])^{H^0}.
\end{equation}
\end{cor} 
\begin{proof}  By \Cref{EM_dis}
the natural map $L[\PC^\Sigma_H] \rightarrow L[\Rep^\Sigma_H]^{H^0}$ is an isomorphism. The assertion follows from \Cref{Agen_versus_abstract} and \Cref{tensor_inv2}.
\end{proof}

\subsection{An adequate homeomorphism} Let $\Gamma$ be a 
finitely generated profinite group. We may choose a surjection 
$\Sigmahat\twoheadrightarrow \Gamma$, where $\Sigma$ is a finitely generated discrete group. We now assume that $\Thetabar \in \cPC_H^{\Gamma}(k)$ and let $R^{\ps, \Gamma}_{\Thetabar}$
be  the universal deformation ring of $\Thetabar$ 
pro-representing $\Def^{\Gamma}_{\Thetabar}$ 
and let $R^{\ps, \Sigmahat}_{\Thetabar}$
be  the universal deformation ring of $\Thetabar$
pro-representing $\Def^{\Sigmahat}_{\Thetabar}$.
The main result of this section is \Cref{adequate}, which improves \cite[Corollary 7.6]{defG}.
Its \Cref{normal} 
is  a key technical ingredient in this paper.

\begin{lem}\label{pushout_wins} 
We have a cocartesian square of $\OO[\PC^{\widehat \Sigma}_H]$-algebras
    \begin{equation}
        \begin{tikzcd}
            \OO[\Rep^{\widehat \Sigma}_H] \ar[r] \ar[d] & A^{\gen, \widehat \Sigma}_{\Thetabar} \ar[d] \\
            \OO[\Rep^{\Gamma}_H] \ar[r] \ar[r] & A^{\gen, \Gamma}_{\Thetabar}
        \end{tikzcd}
    \end{equation}
\end{lem}
\begin{proof}
    Since $\widehat \Sigma \to \Gamma$ is surjective the assertion about cocartesian square follows directly by looking at the moduli functors the rings represent.
\end{proof}

\begin{lem}\label{pullup_wins} 
We have a cocartesian square of $R^{\ps, \Sigmahat}_{\Thetabar}$-algebras
    \begin{equation}\label{eq_pull}
        \begin{tikzcd}
            R^{\ps, \Sigmahat}_{\Thetabar}\otimes_{\OO[\PC^{\Sigmahat}_H]}\OO[\Rep^{\widehat \Sigma}_H] \ar[r] \ar[d] & A^{\gen, \widehat \Sigma}_{\Thetabar} \ar[d] \\
            R^{\ps, \Sigmahat}_{\Thetabar}\otimes_{\OO[\PC^{\Sigmahat}_H]}\OO[\Rep^{\Gamma}_H] \ar[r] \ar[r] & A^{\gen, \Gamma}_{\Thetabar}
        \end{tikzcd}
    \end{equation}
\end{lem}
\begin{proof} The assertion follows from \Cref{pushout_wins} 
and an elementary manipulation with tensor products: if 
$B$ is an $A$-algebra, and $A$ and $Y$ are $X$-algebras and 
$C$ an $A\otimes_X Y$-algebra then 
$$ (B\otimes_X Y)\otimes_{A\otimes_X Y} C \cong 
B\otimes_A(A\otimes_X Y)\otimes_{A\otimes_X Y} C \cong B\otimes_A C.$$
\end{proof}

\begin{lem}\label{tensor_inv1} Let $A$, $B$ and $C$ be $L$-algebras on which $H^0_L$ acts
by homomorphisms of $L$-algebras. Let $A\rightarrow B$ and $A\rightarrow C$ be surjective $H^0$-equivariant homomorphisms of $L$-algebras. Then 
$(B\otimes_A C)^{H^0}\cong B^{H^0}\otimes_{A^{H^0}} C^{H^0}.$
\end{lem}
\begin{proof} Let $I:= \Ker( A\twoheadrightarrow B)$ and let $J:=\Ker(A\twoheadrightarrow C)$. Since $L$ has characteristic zero, $H^0_L$ is linearly reductive, and thus $A^{H^0}\rightarrow B^{H^0}$
is surjective with kernel $I^{H^0}$. Similarly, 
$A^{H^0}\rightarrow C^{H^0}$ is surjective with kernel $J^{H^0}$. 
We conclude that $B^{H^0}\otimes_{A^{H^0}} C^{H^0} \cong A^{H^0}/(I^{H^0}+ J^{H^0})$. By considering the exact sequence 
$0\rightarrow I\cap J \rightarrow I \oplus J \rightarrow I+J \rightarrow 0$ we deduce that $(I+J)^{H^0}= I^{H^0}+J^{H^0}$. The assertion follows
from the exact sequence $0\rightarrow I+J \rightarrow A\rightarrow B\otimes_A C\rightarrow 0$.
\end{proof} 
\begin{thm}\label{adequate}
    The natural map $R^{\ps, \Gamma}_{\Thetabar}[1/p] \to (A^{\gen, \Gamma}_{\Thetabar}[1/p])^{H^0}$ is an isomorphism.
    In particular, the map 
    $X^{\gen, \Gamma}_{\Thetabar}\sslash H^0 \rightarrow X^{\ps, \Gamma}_{\Thetabar}$  is a finite adequate homeomorphism.
\end{thm}

\begin{proof}  We will deduce the first assertion 
from \Cref{tensor_inv1} by applying it to the cocartesian 
square in \Cref{pullup_wins} after inverting $p$. 

\Cref{Agen_versus_abstract} implies that the top horizontal arrow in \eqref{eq_pull} is surjective. The first vertical arrow in \eqref{eq_pull} is also surjective as $\Gamma$ is a quotient of $\Sigmahat$. Hence, after inverting $p$ we are in 
a setting, where \Cref{tensor_inv1} applies. 

By \Cref{EM_dis}, the natural map $L[\PC^{\Sigmahat}_H] \to L[\Rep^{\Sigmahat}_H]^{H^0}$ is an isomorphism. Thus \Cref{tensor_inv2}
implies that the natural map 
\begin{equation}R^{\ps, \Sigmahat}_{\Thetabar}[1/p]\rightarrow (R^{\ps, \Sigmahat}_{\Thetabar}[1/p]\otimes_{L[\PC^{\Sigmahat}_H]} L[\Rep^{\Sigmahat}_H])^{H^0}
\end{equation}
is an isomorphism.  \Cref{cor_iso_Qhat} then 
 implies that the map obtained by inverting
$p$ and taking $H^0$-invariants of the top row of \eqref{eq_pull}
\begin{equation} 
(R^{\ps, \Sigmahat}_{\Thetabar}[1/p]\otimes_{L[\PC^{\Sigmahat}_H]} L[\Rep^{\Sigmahat}_H])^{H^0}\rightarrow
(A^{\gen, \Sigmahat}_{\Thetabar}[1/p])^{H^0}
\end{equation}
is an isomorphism. Hence, \Cref{tensor_inv1} implies that the map obtained by inverting $p$ and taking $H^0$-invariants 
of the bottom row of \eqref{eq_pull} 
\begin{equation}
(R^{\ps, \Sigmahat}_{\Thetabar}[1/p]\otimes_{L[\PC^{\Sigmahat}_H]} L[\Rep^{\Gamma}_H])^{H^0}\rightarrow
(A^{\gen, \Gamma}_{\Thetabar}[1/p])^{H^0}
\end{equation}
is also an isomorphism. By \Cref{EM_dis}, the natural map $L[\PC^{\Gamma}_H] \to L[\Rep^{\Gamma}_H]^{H^0}$ is an isomorphism. Thus \Cref{tensor_inv2} implies that
\begin{equation}\label{nearly_there}
R^{\ps, \Sigmahat}_{\Thetabar}[1/p]\otimes_{L[\PC^{\Sigmahat}_H]} L[\PC^{\Gamma}_H] \rightarrow (A^{\gen, \Gamma}_{\Thetabar}[1/p])^{H^0}
\end{equation}
is an isomorphism. Since $\Gamma$ is a quotient of $\Sigmahat$, 
the map $R^{\ps, \Sigmahat}_{\Thetabar}\rightarrow R^{\ps, \Gamma}_{\Thetabar}$ is surjective, and thus induces a surjection 
$R^{\ps, \Sigmahat}_{\Thetabar} \otimes_{\OO[\PC^{\Sigmahat}_H]} \OO[\PC^{\Gamma}_H] \twoheadrightarrow R^{\ps, \Gamma}_{\Thetabar}$. 
Since \eqref{nearly_there} factors through 
$R^{\ps, \Gamma}_{\Thetabar}[1/p]\rightarrow (A^{\gen, \Gamma}_{\Thetabar}[1/p])^{H^0}$ we deduce that this map is an isomorphism.

   In \cite[Proposition 7.4]{defG}, we have shown that the map $R^{\ps, \Gamma}_{\Thetabar} \to (A^{\gen, \Gamma}_{\Thetabar})^{H^0}$ is finite and induces a universal homeomorphism on spectra.
    It follows by \cite[Definition 3.3.1]{alper}, that it is an adequate homeomorphism.
\end{proof}

\begin{cor}\label{normal}
    If $\Gamma=\Gamma_F$ is an absolute Galois group of a $p$-adic local field $F$ then $R^{\ps}_{\Thetabar}[1/p]$ is normal.
\end{cor}

\begin{proof}
    \cite[Corollary 15.23]{defG} says 
    that $A^{\gen}_{\Thetabar}[1/p]$ is normal. Since normality is preserved under taking invariants the assertion follows from \Cref{adequate}.
\end{proof}

\section{Absolutely irreducible locus and smoothness}\label{sec_abs_irr}
Let $\Gamma$ be a topologically finitely generated profinite group and 
let $H$ be a generalised reductive group scheme over $\OO$ as in \Cref{sec_pseudo_def}. We assume that 
$H^0$ is split, $H/H^0$ is a constant group scheme and $H(\OO)\rightarrow (H/H^0)(\OO)$ is surjective. This can always be achieved after extending scalars 
by \cite[Proposition 2.8]{defG}. We fix $\Thetabar\in \cPC^{\Gamma}_H(k)$ such that the homomorphism $\Gamma \to (H/H^0)(\kbar)$ defined by $\Thetabar$ is surjective and let $X^{\ps}:= X^{\ps, \Gamma}_{\Thetabar}$ and $X^{\gen}:= X^{\gen, \Gamma}_{\Thetabar}$.

Let $\kappa$ 
be an $\OO$-algebra, which is a field and let $\overline{\kappa}$ 
be its algebraic closure. A representation $\rho: \Gamma\rightarrow H(\kappa)$ is \textit{absolutely irreducible} if the composition $\Gamma \overset{\rho}{\longrightarrow} H(\kappa)\rightarrow (H/H^0)(\kappa)$ is surjective and  
 $\rho(\Gamma)$ is not contained
in $P(\overline{\kappa})$ for any proper 
parabolic\footnote{Since our $H$ is not
 necessarily connected there are some subtleties defining parabolic subgroups. 
 We refer the interested reader to \cite[Section 2.1, Example 9.14]{defG}.} subgroup of $H_{\kappabar}$. 
 
 Let $\ast$ be the closed point of $X^{\ps}$ and let 
 $Y$ be its preimage in $X^{\gen}$. The \textit{absolutely irreducible 
 locus} $(X^{\gen})^{\irr}$ in $X^{\gen}$ is defined in \cite[Definition 9.12]{defG}. It is
 an open subscheme of $X^{\gen}\setminus Y$ characterised by the property 
 that $x\in X^{\gen}\setminus Y$ lies in $(X^{\gen})^{\irr}$
 if and only if $\rho_x: \Gamma \rightarrow H(\kappa(x))$ is absolutely irreducible, where
  $\rho_x$ is  the specialisation of the universal representation 
over $X^{\gen}$ at $x$.
 The assumption that 
 $\Gamma \overset{\rhobar}{\longrightarrow} H(k) 
 \rightarrow (H/H^0)(k)$ is surjective and $H/H^0$ is constant implies that $\rho_x(\Gamma)$ surjects onto $(H/H^0)(\kappa(x))$. 
 
 The \textit{absolutely irreducible locus} $(X^{\ps})^{\irr}$ in $X^{\ps}$ is an open subscheme of $X^{\ps} \setminus \{\ast\}$ 
 characterised by the property that 
 $y\in X^{\ps} \setminus \{\ast\}$ lies in $(X^{\ps})^{\irr}$
 if and only if $\Theta^u_y$ is an
 $H$-pseudocharacter of an 
 (absolutely) irreducible representation $\rho: \Gamma \rightarrow H(\overline{\kappa(y)})$. The reconstruction theorem  \cite[Theorem 3.8]{quast} implies that all such $\rho$ are in the same $H^0$-conjugacy class and hence the  
 locus is well defined. 
 
 We remove $\ast$, even if $\rhobar$ is absolutely irreducible, because
 then the schemes $(X^{\ps})^{\irr}$ and $(X^{\gen})^{\irr}$ are Jacobson. 
 
 Let $\pi: X^{\gen} \rightarrow X^{\ps}$ be the morphism, which sends a representation to 
 its $H$-pseudocharacter. Then $(X^{\gen})^{\irr}= \pi^{-1}((X^{\ps})^{\irr})$.
 
 The main result of this section is \Cref{cover_nice_locus}, in which we assume that $\Gamma=\Gamma_F$, where $F$ is a $p$-adic local field. We establish the existence of an open subscheme $U$ of 
 $(X^{\ps})^{\irr}[1/p]$ and a covering 
 $\{\Spec A_i\rightarrow U\}_{i\in I}$ for the \'etale topology on $U$ together 
 with a family representations $\rho_i: \Gamma_F\rightarrow H(A_i)$ such that $\Theta_{\rho_i}= \Theta^u_{|A_i}$. We produce
 the representations $\rho_i$ by showing that 
 $\pi: \pi^{-1}(U) \rightarrow U$ is a smooth map 
 and then using the existence of sections \'etale 
 locally on the target. We also show that if $F\neq \Qp$ 
 then the complement of $U$ in $X^{\ps}[1/p]$ has codimension 
 at least $2$, and the same holds for the complement 
 of $\pi^{-1}(U)$ in $X^{\gen}[1/p]$. This is important 
 for the application of Riemann's Hebbarkeitssatz
 in \Cref{sec_constr}.
 
 If $H=\GL_n$ then stronger results have been proved by Chenevier \cite{che_durham}. He shows
 that the Cayley--Hamilton algebra is an Azumaya
 algebra $\mathcal A$ over $(X^{\ps}_{\Thetabar})^{\irr}$. 
 One may then 
 choose an \'etale covering $\{U_i\}_{i\in I}$ of 
 $(X^{\ps}_{\Thetabar})^{\irr}$,
 such that $\mathcal A|_{U_i}$ is a matrix algebra.
 By restricting the natural representation
 $\rho: \Gamma \rightarrow \mathcal A^*$ to
 $U_i$ one obtains the representations $\rho_i$.  
 The reason that the situation is much better in 
 the case $H=\GL_n$ is that it follows from Schur's
 lemma that the $H^0$-centraliser of an absolutely irreducible representation is $\Gm$. In particular, 
 it is smooth and connected. This need not be the case, when $H$ is generalised reductive. We avoid the smoothness problem by working in characteristic zero. However, we have to do some honest work to deal with the fact that the centralisers of the absolutely irreducible representations might not be connected. In this case the (unframed) deformation ring of $\rho_x$ need not 
 be isomorphic to the deformation ring of its pseudocharacter, see \Cref{main_luna}, and this would yield an obstruction to smoothness of $\pi$ at $x$. We deal with this problem by showing in \Cref{main_cor_luna} that 
 if $\Gamma=\Gamma_F$ and both $x$ and $\pi(x)$ are regular then the obstruction to smoothness of  
 $\pi$ at $x$ vanishes. This is done by 
 comparing the (unframed) deformation ring of $\rho_x$ with the completion of $\OO_{X^{\ps}, \pi(x)}$ in \Cref{main_luna}, the Chevalley--Shephard--Todd theorem and some arguments
 related to Luna's fundamental lemma. This section  benefited from a close reading of \cite[Theorem 4.1]{BHKT} and \cite[Theorem 53]{sikora}.

\subsection{Deformations of absolutely irreducible representations}

Let $\kappa$ be an $\OO$-algebra, which is either a finite or a local field. 
In this subsection we let $H$ be a generalised reductive group defined over $\kappa$ with constant
component group such that $H^0$ is split.
Let $\rho: \Gamma \rightarrow H(\kappa)$ 
be a continuous absolutely irreducible representation. We recall some results in deformation theory of $\rho$, following \cite{BHKT}, where the group $H$ is assumed to be connected. 

Let $\mathfrak A_{\kappa}$ be the category of local artinian $\kappa$-algebras with residue field $\kappa$.
If $A\in \mathfrak A_{\kappa}$ then $A$ is a finite dimensional $\kappa$-vector space and it inherits a topology from $\kappa$.
The functor 
$D^{\square}_{\rho}: 
\mathfrak A_{\kappa}
\rightarrow \Set$, which 
sends $(A, \mm_A)$ to the set 
of continuous representations 
$\rho_A: \Gamma\rightarrow H(A)$,
which are equal to $\rho$ modulo
$\mm_A$ is pro-represented by a
complete local noetherian
$\kappa$-algebra
$R^{\square}_{\rho}$.

If $G$ is an affine algebraic group scheme over $\kappa$ we let $G^{\wedge}: \mathfrak A_{\kappa}\rightarrow \Group$ be the functor which sends $A$ to 
$\Ker(G(A)\rightarrow G(\kappa))$. The
functor $G^{\wedge}$ is pro-represented by $\widehat{\OO}_{G,1}$ - the completion of the local ring at the identity $\OO_{G, 1}$ with respect to the maximal ideal.
We note that  $G^{\wedge}$ depends
only on the neutral component of $G$, so that 
$G^{\wedge}=G^{0,\wedge}$.

We let $S_{\rho}$ be the scheme 
theoretic centraliser of $\rho$ in $H^0$.
Since $\rho$ is absolutely irreducible  \cite[Proposition 9.7]{defG} implies that 
the underlying reduced subschemes of the neutral components of $S_{\rho}$ and 
the centre $Z(H)$ of $H$ coincide. In particular, if $S_{\rho}$ is smooth then $S_{\rho}^0=Z(H)^0$. Moreover, we have established in the proof of \cite[Proposition 2.9]{defG} that $Z(H)^0$ is diagonalisable. 

If $Z(H)^0$ is smooth  then the morphism  $H^0\rightarrow
H^0/Z(H)^0$ is smooth.
Hence, $\widehat{\OO}_{H^0, 1}$ is isomorphic to a formal power series ring over 
$\widehat{\OO}_{H^0/Z(H)^0, 1}$. 
This implies that the sequence 
\begin{equation}\label{mod_stab}
0\rightarrow Z(H)^{0,\wedge}(A)\rightarrow H^{0,\wedge}(A)\rightarrow (H^0/Z(H)^0)^{\wedge}(A)\rightarrow 0 \end{equation}
is exact for all $A\in \mathfrak A_{\kappa}$. The group $H^{0, \wedge}(A)$ acts on $D^{\square}_{\rho}(A)$ by conjugation. Since the induced action of $Z(H)^0(A)$ is trivial, it follows from \eqref{mod_stab} that the action factors through 
the quotient $(H^0/Z(H)^0)^{\wedge}(A)$.

\begin{lem}\label{KW} If $S_{\rho}$ is smooth then the action of $(H^0/Z(H)^0)^{\wedge}$ on $D^{\square}_{\rho}$ is free. 
\end{lem}
\begin{proof} The assumption that $S_{\rho}$ is smooth implies that $S^0_{\rho}=Z(H)^0$ and hence $Z(H)^0$ is smooth. We have to show that the map 
\begin{equation}\label{eq_KW}
(H^0/Z(H)^0)^{\wedge}\times D^{\square}_{\rho} \rightarrow D^{\square}_{\rho} \times D^{\square}_{\rho},\quad  (g, \rho)\mapsto (\rho, g \rho g^{-1})
\end{equation}is a closed immersion. This is equivalent to showing that the map of local 
rings $R^{\square}_{\rho}\wtimes
R^{\square}_{\rho}\rightarrow
\widehat{\OO}_{H^0/Z(H)^0,1}\wtimes R^{\square}_{\rho}$
is surjective. By Nakayama's lemma this is equivalent to
showing that \eqref{eq_KW} is injective, when evaluated
at the dual numbers. 

Let $g\in H^{0,
\wedge}(\kappa[\varepsilon])$ and let $\rho'\in 
D^{\square}_{\rho}(\kappa[\varepsilon])$ such that $g \rho'(\gamma) g^{-1}=\rho'(\gamma)$ for all $\gamma\in \Gamma$. After choosing a closed immersion of $H$ in some $\GL_d$ 
we may write $\rho'(\gamma)= \rho(\gamma)+ \varepsilon c(\gamma)$
and $g=1 + \varepsilon a$ 
for $c(\gamma), a\in M_d(\kappa)$. We have $g \varepsilon c(\gamma) g^{-1}= \varepsilon c(\gamma)$ as 
$\varepsilon^2=0$. Thus $g \rho'(\gamma) g^{-1} -\rho'(\gamma)= g \rho(\gamma) g^{-1} -\rho(\gamma)$ and hence $g$ fixes $\rho'$ if and only if $g\in S_{\rho}(\kappa[\varepsilon])$. 
Since $g\in H^{0, \wedge}(\kappa[\varepsilon])$ we deduce 
that $g\in S^{\wedge}_\rho(\kappa[\varepsilon])$. 
Since $S^0_{\rho}=Z(H)^0$ we obtain that 
$g\in Z(H)^{0, \wedge}(\kappa[\varepsilon])$. 
It follows from \eqref{mod_stab}  that the
$(H^0/Z(H)^0)^{\wedge}(\kappa[\varepsilon])$-stabiliser of $\rho'$ is trivial and hence  \eqref{eq_KW} is injective.
\end{proof}
 
It follows from \Cref{KW} and \cite[Proposition 2.5]{KW2} that if $S_{\rho}$ is smooth then the functor $D_{\rho}: \mathfrak A_{\kappa}\rightarrow \Set$, which 
sends $A$ to the set of $(H^0/Z(H)^0)^{\wedge}(A)$-orbits in $D^{\square}_{\rho}(A)$ is represented by a complete local noetherian 
$\kappa$-algebra $R_{\rho}$. The map $D^{\square}_{\rho}\rightarrow D_{\rho}$
which sends a framed deformation to its orbit induces a formally smooth 
homomorphism of local $\kappa$-algebras $R_{\rho}\rightarrow R_{\rho}^{\square}$. We choose a continuous representation $\rho^{\ver}: \Gamma\rightarrow H(R_{\rho})$ such that its orbit is the universal object 
in $D_{\rho}(R_{\rho})$. This induces a 
homomorphism of local $\kappa$-algebras
$R^{\square}_{\rho}\rightarrow R_{\rho}$ such that its composition with $R_{\rho}\rightarrow R^{\square}_{\rho}$ is the identity. The map 
\begin{equation}
 (H^0/Z(H)^0)^{\wedge}(A)\times h_{R_{\rho}}(A) \rightarrow D^{\square}_{\rho}(A), \quad (g,\varphi)
\mapsto  g (\rho^{\ver}\otimes_{R_{\rho}, \varphi} A ) g^{-1}
\end{equation} 
is bijective by \cite[Proposition 2.5]{KW2}. Hence, it induces an isomorphism of local $\kappa$-algebras 
\begin{equation}\label{Ohat}
R^{\square}_{\rho} \cong R_{\rho}\wtimes \Ohat_{H^0/Z(H)^0, 1}.
\end{equation}
\begin{remar} We will apply the results 
of this section only in characteristic zero, where $S_{\rho}$ is automatically smooth, \cite[Theorem 3.23]{milne_alg}. 
\end{remar}

\subsection{An application of Luna's fundamental lemma}\label{sec_luna}
In this subsection we will study, when the morphism $\pi: X^{\gen} \rightarrow X^{\ps}$ is smooth 
at a closed point $x$ of $X^{\gen}[1/p]$  such that the representation $\rho_x: \Gamma\rightarrow H(\kappa(x))$ is absolutely irreducible.
 The main result is \Cref{main_cor_luna}, which asserts that 
if $x$ and  $y:=\pi(x)$ are both regular points 
and $\Gamma$ is the absolute Galois group of 
a $p$-adic local field then $\pi$ is smooth at $x$. This result gets
applied in the next subsection in \Cref{nice_locus}.

Let $H_x$ be the base change of $H$ to 
$\kappa(x)$ and let $S_{\rho_x}$ be the scheme theoretic 
$H^0_x$-centraliser of $\rho_x$. Since $\kappa(x)$ has characteristic zero $S_{\rho_x}$ is smooth. Its 
 component group is a finite \'etale group scheme
by \cite[Proposition 2.37]{milne_alg}. The residue fields $\kappa(x)$ and $\kappa(y)$ are finite extensions of $L$. After replacing $L$ by a finite extension 
we may assume that $\kappa(x)=\kappa(y)=L$ and $S_{\rho_x}/S^0_{\rho_x}$ is a constant group
scheme corresponding to a finite group $\Delta$. 

The non-triviality of $\Delta$ can lead to an 
obstruction to smoothness 
of $\pi$ at $x$. This problem does not occur 
if $H=\GL_n$ as it follows from Schur's lemma that $S_{\rho_x}=\Gm$ and so $\Delta$ is trivial. 
Since $X^{\ps}[1/p]$ is isomorphic to the GIT-quotient of $X^{\gen}[1/p]$ by $H^0$ by \Cref{adequate}, we relate the completions
of local rings at $x$ and $y$ via arguments in the spirit of Luna's fundamental lemma, and then control the 
obstruction to smoothness caused by non-triviality of $\Delta$ via the Chevalley--Shephard--Todd theorem.

We will apply the results of the previous subsection 
with $\rho=\rho_x$, $\kappa=\kappa(x)$ and $H=H_x$. 
Proposition 5.16 in \cite{defG}
implies that $R^{\square}_{\rho_x}$ is naturally isomorphic to the completion of the local ring 
$\OO_{X^{\gen},x}$ with respect to its maximal ideal. Let $\Theta_y$ be the 
$H$-pseudocharacter of $\rho_x$ and 
let $R_{\Theta_y}$ be the completion\footnote{We have shown in \cite[Proposition 7.2]{MRL} that $R_{\Theta_y}$ is isomorphic
to the universal deformation ring of $\Theta_y$. However, we do not use this description below.} 
of $\OO_{X^{\ps},y}$. The map 
$\OO_{X^{\ps}, y}\rightarrow \OO_{X^{\gen}, x}\rightarrow R^{\square}_{\rho_x}$ induces 
a homomorphism of local $\kappa(y)$-algebras
$R_{\Theta_y}\rightarrow R^{\square}_{\rho_x}$. 
 Since conjugation does
not change the pseudocharacter the map 
$\OO_{X^{\ps}, y}\rightarrow R^{\square}_{\rho_x}$ 
factors through $\OO_{X^{\ps}, y}\rightarrow R_{\rho_x}$ and hence induces a map of local $\kappa(y)$-algebras  $R_{\Theta_y}\rightarrow R_{\rho_x}$. 

\begin{lem}\label{finite_ver} The map $R_{\Theta_y} \rightarrow R_{\rho_x}$ is finite.
\end{lem}
\begin{proof} Let $A$ be the 
fibre ring $\kappa(y)\otimes_{R_{\Theta_y}} R_{\rho_x}$. By Nakayama's lemma 
it is enough to show that $A$ is a finite dimensional 
$\kappa(y)$-vector space. Since $R_{\rho_x}$ is noetherian, it is enough to show that the Krull dimension of 
$A$ is zero. Since 
the map $R_{\Theta_y}\rightarrow R_{\rho_x}^{\square}$
factors through $R_{\Theta_y}\rightarrow R_{\rho_x}$ 
from \eqref{Ohat} we obtain 
\begin{equation}
B:= R^{\square}_{\rho_x}\otimes_{R_{\Theta_y}}\kappa(y) \cong A \wtimes \Ohat_{H_x^0/Z(H_x)^0,1}
\end{equation}
Thus it is enough to show that $\dim B \le \dim H_x -\dim Z(H_x)$. Since $R^{\square}_{\rho_x}$ is isomorphic to the completion of the local ring at $x$, 
$B$ is the 
completion of the local ring at $x$ of the 
fibre $X^{\gen} \times_{X^{\ps}} \Spec \kappa(y)$. Since $\rho_x$ is absolutely irreducible the 
fibre has dimension $\dim H_x -\dim Z(H_x)$ by 
\cite[Corollary 9.9]{defG}. Since the fibre 
is of finite type over $\kappa(y)$ 
the assertion follows. 
\end{proof}

If $A\in \mathfrak A_{\kappa(x)}$ then 
$S_{\rho_x}(A)$ acts on $D^{\square}_{\rho_x}(A)$ by conjugation. 
Since $S_{\rho_x}$ is smooth we have $S_{\rho_x}^0=Z(H_x)^0$.
Thus the 
action factors through the quotient
$(S_{\rho_x}/S^0_{\rho_x})(A)=\Delta$. 
Since $S_{\rho_x}(A)$ normalises $H^{0,\wedge}_x(A)$ this action induces 
an action of $\Delta$ on $D_{\rho_x}(A)$, 
and hence an action of $\Delta$ on $R_{\rho_x}$. Moreover, the natural map $D^{\square}_{\rho_x}(A)\rightarrow D_{\rho_x}(A)$ is 
$\Delta$-equivariant. Since $R^{\square}_{\rho_x}$ is formally 
smooth over $R_{\rho_x}$ by \eqref{Ohat} and the order of $\Delta$
is invertible in $\kappa(x)$ there is a $\Delta$-equivariant homomorphism $\sigma: R^{\square}_{\rho_x}\rightarrow R_{\rho_x}$ of local $\kappa(x)$-algebras such that the composition of $\sigma$ with 
the natural map $R_{\rho_x}\rightarrow R^{\square}_{\rho_x}$ is 
the identity map. Let  $\rho^{\ver}: \Gamma\rightarrow H(R_{\rho_x})$ be the representation corresponding to $\sigma\in D^{\square}_{\rho_x}(R_{\rho_x})$. The equivalence 
class of $\rho^{\ver}$ is the universal object in $D_{\rho_x}(R_{\rho_x})$ by construction. Moreover, it induces a morphism 
\begin{equation}\label{schaum1}
\Spec R_{\rho_x}\rightarrow X^{\gen}\times_{X^{\ps}} \Spec R_{\Theta_y}.
\end{equation}
The action of $H^0$ on $X^{\gen}$ induces a morphism 
\begin{equation}\label{schaum2}
  (H^0_x/Z(H_x)^0)\times( X^{\gen}\times_{X^{\ps}} \Spec R_{\Theta_y}) \rightarrow X^{\gen}\times_{X^{\ps}} \Spec R_{\Theta_y}.
\end{equation}
By composing \eqref{schaum1} and \eqref{schaum2} we obtain a morphism 
\begin{equation}\label{schaum3}
(H^0_x/Z(H_x)^0)\times \Spec R_{\rho_x}\rightarrow X^{\gen}\times_{X^{\ps}} \Spec R_{\Theta_y}.
\end{equation}
The morphism \eqref{schaum3} is $\Delta$-invariant 
for the action $g\cdot (h, \rho) = (hg^{-1}, g \rho g^{-1})$ on the source and the trivial action on the target.  Hence, \eqref{schaum3} factors through \begin{equation}\label{etale_mor}
(H^0_x/Z(H_x)^0)\times^\Delta \Spec R_{\rho_x}\rightarrow X^{\gen} \times_{X^{\ps}}\Spec R_{\Theta_y}.
\end{equation} 
\begin{prop} The morphism \eqref{etale_mor} 
is \'etale at the point $(1, x)$. 
\end{prop}
\begin{proof}
Since $\Delta$ acts freely on $(H^0_x/Z(H_x)^0)\times 
\Spec R_{\rho_x}$ the morphism  to its quotient by $\Delta$ is \'etale by \cite[Exp.\,V, Cor. 2.3]{SGA1}. Thus it is enough to establish that \eqref{schaum3} is \'etale at the 
point $(1,x)$. It follows from \eqref{Ohat} that \eqref{schaum3} induces an 
isomorphism between the completed local 
rings at $(1,x)$ and at its image. Since 
$X^{\gen} \rightarrow X^{\ps}$ is of finite type, \Cref{finite_ver} implies that \eqref{schaum3} is 
of finite type. Thus the map is \'etale at $(1, x)$ by \cite[Corollary 18.66]{GW2}.
\end{proof}

\begin{prop}\label{main_luna} The map $R_{\Theta_y}\rightarrow R_{\rho_x}$ induces an isomorphism $R_{\Theta_y}\cong 
R_{\rho_x}^{\Delta}$.
\end{prop}

\begin{proof}
Let $\mm$ be the maximal ideal of $R_{\Theta_y}$. For each $n\ge 1$, we base change   \eqref{etale_mor} to $\Spec R_{\Theta_y}/\mm^n$ to obtain 
 a morphism
\begin{equation}\label{etale_mor_n}
X_n:= H^0_x\times^{S_{\rho_x}} \Spec R_{\rho_x}/\mm^n R_{\rho_x}\rightarrow Y_n:=X^{\gen} \times_{X^{\ps}}\Spec R_{\Theta_y}/\mm^n,
\end{equation} 
which is \'etale at $(1,x)$. In the above description of $X_n$ we use that the order of $\Delta$ is invertible, so that taking $\Delta$-invariants is exact. We have 
\begin{equation}\label{etale_n}
X_n\sslash H^0_x\cong (\Spec R_{\rho_x}/\mm^n R_{\rho_x})\sslash S_{\rho}\cong \Spec (R_{\rho_x}/\mm^n R_{\rho_x})^{\Delta}
\end{equation}
Since 
reductive groups are linearly reductive in characteristic zero, \Cref{adequate} implies that
\begin{equation}\label{etale_m}
Y_n \sslash H^0_x\cong \Spec R_{\Theta_y}/\mm^n.
\end{equation}
We claim that \eqref{etale_mor_n} is an isomorphism. 
The claim,  \eqref{etale_n}
and \eqref{etale_m} imply that 
$R_{\Theta_y}/\mm^n\rightarrow (R_{\rho_x}/\mm^n R_{\rho_x})^{\Delta}$ is an isomorphism. 
By passing to the projective 
limit and using the exactness of $\Delta$-invariants we obtain 
the desired isomorphism.

To prove the claim, we note that if $\Omega$ is an algebraically closed field containing $L$ then  $X_n(\Omega)$ is 
just an $H^0(\Omega)$-orbit of $\rho_x$ in $X^{\gen}(\Omega)$, and it follows 
from the reconstruction theorem \cite[Theorem 3.7]{quast} that 
\eqref{etale_mor_n} induces a bijection between 
$X_n(\Omega)$ and $Y_n(\Omega)$.  
Hence, \eqref{etale_mor_n} is universally injective by \cite[\href{https://stacks.math.columbia.edu/tag/01S4}{Tag 01S4}]{stacks-project}. Since 
\eqref{etale_mor_n} is \'etale at 
$(1,x)$ and $H^0_x(\Omega)$ acts transitively on $X_n(\Omega)$, the map is \'etale. Thus \eqref{etale_mor_n} 
is an open immersion by \cite[\href{https://stacks.math.columbia.edu/tag/025G}{Tag 025G}]{stacks-project}. Since it induces a bijection on geometric points the map is an isomorphism. 
\end{proof}

Let $\mm$ be the maximal ideal of $R_{\rho_x}$. 
The action of $\Delta$ on $R_{\rho_x}$ induces an action on $\mm/\mm^2$, which is canonically 
isomorphic to $\kappa(x)$-dual of 
$H^1(\Gamma, \Lie H_x)$. Since the action 
of $\Delta$ is semi-simple there is 
a $\Delta$-invariant $\kappa(x)$-subspace
$V\subset \mm$, which maps isomorphically onto $\mm/\mm^2$. This induces a $\Delta$-equivariant map of $\kappa(x)$-algebras $S(V)\rightarrow R_{\rho_x}$ and a surjective $\Delta$-equivariant map of local $\kappa(x)$-algebras
\begin{equation}\label{PV_surj}
P(V)\twoheadrightarrow R_{\rho_x}
\end{equation}
with the notations of \Cref{section_todd}. Let $\overline{\Delta}$ be the image of $\Delta$ in $\GL(V)$.
It follows from \eqref{PV_surj} that the action of 
$\Delta$ on $R_{\rho_x}$ factors through $\overline{\Delta}$.

\begin{cor}\label{58} Suppose $x$ and $y$ are regular points of 
$X^{\gen}$ and $X^{\ps}$, respectively. Then
$\overline{\Delta}$ is generated by pseudoreflections.
\end{cor}
\begin{proof} The local rings at $x$ and $y$ are regular by assumption. Since we are in the equicharacteristic 
setting, their 
completions are isomorphic to formal power series rings over the 
residue field, which we have arranged to be $L$. 
Since the completion of the local 
ring at $x$ is isomorphic to $R^{\square}_{\rho_x}$ by \cite[Proposition 5.16]{defG}, it follows from \eqref{Ohat} that 
$R_{\rho_x}$ is also isomorphic to 
a formal power series ring over $L$.
This implies that \eqref{PV_surj} is an isomorphism.
Since the action of $\Delta$
on $R_{\rho_x}$ factors through 
$\overline{\Delta}$, 
the assertion follows from \Cref{main_luna} and \Cref{main_todd}.
\end{proof}

\begin{cor}\label{main_cor_luna} Assume that $\Gamma=\Gamma_F$, where $F$ is a finite extension of $\Qp$.  If $x$ and $y$ are regular points then $\pi: X^{\gen}\rightarrow X^{\ps}$ is smooth at $x$.
\end{cor}
\begin{proof} We claim that in this case $\overline{\Delta}$ does not contain a pseudoreflection. \Cref{58} implies that $\overline{\Delta}$ is trivial and \Cref{main_luna} implies that $R_{\Theta_y}\rightarrow R_{\rho_x}$ is an isomorphism. 
It follows from \eqref{Ohat} that 
$\widehat{\OO}_{X^{\gen},x}$ is isomorphic to the ring of formal 
power series over $\widehat{\OO}_{X^{\ps}, y}$. 
Since the map is of finite type, the assertion follows from
\cite[Theorem 18.63 (iii)]{GW2}.

To prove the claim let $g\in \overline{\Delta}$ be non-trivial. We may canonically decompose 
\begin{equation}\label{decompose}
\Lie H_x= V'\oplus W,
\end{equation}
where $g-1$ is zero on $V'$ and invertible on $W$.
We claim that $\dim W\ge 2$. To verify this we extend the scalars to the algebraic closure of $L$.
We may choose 
an element $h\in S_{\rho_x}(\overline{L})$, such that $h$ has finite order and it is mapped to $g$. Since we
are in characteristic zero $h$ is semisimple and thus it  is  contained in 
a maximal torus $T$ of $H^0_{\overline{L}}$.
Since the action of $g$ on $\Lie H_x$ is non-trivial, 
\cite[Theorem 13.18 (1)]{borel} implies that 
 there is a root $\alpha: T\rightarrow \Gm$ such that $\alpha(h)\neq 1$.  Then $(-\alpha)(h)\neq 1$ and hence $\dim W\ge 2$.

Since $h$ centralises $\rho_x(\Gamma_F)$ the decomposition \eqref{decompose} is $\Gamma_F$-invariant. It follows from the Euler--Poincar\'e characteristic formula that $\dim H^1(\Gamma_F, W)\ge [F:\Qp]\dim W$.
Since $H^1(\Gamma_F, \Lie H_x)\cong V^*$ and $g-1$ is invertible on a direct summand $H^1(\Gamma_F, W)$,  we deduce that $g$ cannot be a pseudoreflection
and hence the claim is proved.
\end{proof} 

\begin{remar} We note that if  
the assertion of \Cref{main_cor_luna} 
holds after replacing $L$ by a finite 
field extension then it already holds over $L$. Thus the extension of scalars made at the beginning of the section 
to trivialise the component group of $S_{\rho_x}$ and to arrange $\kappa(x)=\kappa(y)=L$ does not affect this statement.
\end{remar}

\subsection{Arithmetic setting} In this subsection we assume that $\Gamma$ is the absolute Galois group of a $p$-adic local field $F$. 

The morphism $\pi: X^{\gen} \rightarrow X^{\ps}$ induces an isomorphism $X^{\gen}[1/p]\sslash H^0\cong X^{\ps}[1/p]$ by \Cref{adequate}. Thus $\pi(X^{\gen}[1/p])=X^{\ps}[1/p]$ by  \cite[Theorem 3 (ii)]{seshadri}. 

The schemes $X^{\gen}[1/p]$ 
and $X^{\ps}[1/p]$ are Jacobson, and the 
residue fields of closed points are finite 
extensions of $L$ by \cite[Lemmas 3.17, 3.18]{BIP_new}. Thus open loci are uniquely determined by their $\Qpbar$-points.

Let $V^{\irr}$ be the absolutely irreducible locus in $X^{\gen}[1/p]$ defined in \cite[Definition 9.12]{defG}. It is 
an open subscheme of  $X^{\gen}[1/p]$, 
and its $\Qpbar$-points correspond to 
 $x\in X^{\gen}(\Qpbar)$ such that 
 $\rho_x(\Gamma_F)$ is not contained in
 any proper parabolic  subgroup of $H(\Qpbar)$. Let $Z^{\red}$ be the complement of $V^{\irr}$ in $X^{\gen}[1/p]$. Then 
 $Z^{\red}$ is a closed $H^0$-invariant subspace of
 $X^{\gen}[1/p]$ and hence its image $\pi(Z^{\red})$ is 
 a closed subset of $X^{\ps}[1/p]$ by \cite[Theorem 3]{seshadri} and \Cref{adequate}. We let $U^{\irr}$ be its complement. We have $V^{\irr}=\pi^{-1}(U^{\irr})$.

\begin{lem}\label{topo} Let $Z$ be a closed $H^0$-invariant 
subscheme of $X^{\gen}[1/p]$. Then the following hold: 
\begin{enumerate}
\item $\pi(Z)$ is closed in $X^{\ps}[1/p]$;
\item $V^{\irr}\cap Z = \pi^{-1}(U^{\irr}\cap \pi(Z))$ as sets;
\item if $V^{\irr}\cap Z$ is dense in $Z$ 
then $U^{\irr}\cap \pi(Z)$ is dense in $\pi(Z)$ and 
$$\dim X^{\gen}[1/p] - \dim Z = \dim X^{\ps}[1/p] - \dim \pi(Z);$$
\item if $Z$ contains $Z^{\red}$ then $Z=\pi^{-1}(\pi(Z))$ and hence the map 
$$X^{\gen}[1/p]\setminus Z \rightarrow X^{\ps}[1/p]\setminus \pi(Z)$$
is well defined and surjective.
\end{enumerate}
\end{lem}
\begin{proof}
Part (1) follows from \cite[Theorem 3 (iii)]{seshadri}. If $y$ is a geometric point of  $U^{\irr}\cap \pi(Z)$ then all $x\in \pi^{-1}(y)$ 
are in the same $H^0$-orbit and at least one $x$ lies in $Z$. Since $Z$
is $H^0$-invariant, we get an inclusion $\pi^{-1}(U^{\irr}\cap \pi(Z))\subset V^{\irr}\cap Z$. Since $U^{\irr}=\pi(V^{\irr})$ the other inclusion is trivial. Hence we obtain part (2). 

For part (3) we let $W$ be the closure of $U^{\irr}\cap \pi(Z)$ in $X^{\ps}[1/p]$. Since $\pi(Z)$ is closed by part (1) we have $W\subset \pi(Z)$. Since 
$\pi^{-1}(W)$ is closed and contains $V^{\irr}\cap Z$ we deduce that $Z\subset \pi^{-1}(W)$ and hence $\pi(Z)\subset W$. The schemes 
$X^{\gen}[1/p]$ and $X^{\ps}[1/p]$ 
are noetherian, catenary and Jacobson. We
have 
\begin{equation}\label{alpha}
\dim Z= \dim (V^{\irr}\cap Z), \quad \dim \pi(Z)=\dim (U^{\irr}\cap \pi(Z))
\end{equation}
by \cite[\href{https://stacks.math.columbia.edu/tag/0DRT}{Tag 0DRT}]{stacks-project}. It follows from \cite[Corollary 9.10]{defG} and part (2) that 
\begin{equation}\label{beta}
\dim (V^{\irr}\cap Z) - \dim (U^{\irr}\cap \pi(Z)) = 
\dim H_x - \dim Z(H_x).
\end{equation}
Moreover, Corollaries 13.28, 13.31 in \cite{defG} imply that 
\begin{equation}\label{gamma}
\dim X^{\gen}[1/p]-\dim X^{\ps}[1/p]= \dim H_x - \dim Z(H_x).
\end{equation}
Putting \eqref{alpha}, \eqref{beta} and \eqref{gamma} together we obtain part (3). 

For part (4) we may write 
$\pi(Z)= (U^{\irr}\cap \pi(Z)) \cup \pi(Z^{\red})$. Since $\pi^{-1}(\pi(Z^{\red}))=Z^{\red}$
and $Z= (V^{\irr}\cap Z) \cup Z^{\red}$ the assertion follows from part (2).
\end{proof}

\begin{prop}\label{nice_locus} Assume that $F\neq \Qp$. Then there is an open  subscheme $U$ of $X^{\ps}[1/p]$
and an open subscheme $V$ of $X^{\gen}[1/p]$ such that the  following hold:
\begin{enumerate}
\item $U\subset U^{\irr}$ and $V\subset V^{\irr}$;
\item $U$ and $V$ are both regular;
\item $\pi|_V:V\rightarrow U$ is smooth and surjective;
\item $\dim X^{\ps}[1/p] - \dim (X^{\ps}[1/p]\setminus U) \ge 2$;
\item 
$\dim X^{\gen}[1/p] - \dim (X^{\gen}[1/p]\setminus V)\ge 2$.
\end{enumerate} 
 \end{prop} 
\begin{proof}  We let 
\begin{equation}
Z= Z^{\red}\cup X^{\gen}[1/p]^{\sing} \cup \pi^{-1}(X^{\ps}[1/p]^{\sing}),
\end{equation}
where the superscript $\sing$ denotes the singular locus. Since the schemes $X^{\gen}[1/p]$ and 
$X^{\ps}[1/p]$ are excellent, this locus is
closed. Moreover, $Z$ is $H^0$-invariant.
We let $V=X^{\gen}[1/p]\setminus Z$ and
$U=X^{\ps}[1/p]\setminus \pi(Z)$. 
Parts (1) and (4) of \Cref{topo} 
imply that $U$ is open in $X^{\ps}[1/p]$
and $\pi(V)=U$. Parts (1) and (2) of the
proposition hold by construction. 
The locus in $V$, where $\pi|_V$ is not smooth is closed. If this locus is non-empty then it would contain a closed point of $V$. It follows from (2) and \Cref{main_cor_luna} that this is not possible. Hence $\pi|_V$ is smooth and part (3) of the proposition holds. 

We will now verify the estimates on the codimensions of $Z$ in $X^{\gen}[1/p]$ and $\pi(Z)$ in $X^{\ps}[1/p]$. We note that the schemes 
$X^{\gen}[1/p]$ and $X^{\ps}[1/p]$ are equidimensional by \cite[Corollaries 13.28, 13.31]{defG} and \Cref{useful_codim_dim} applies here with $R=R^{\ps}_{\Thetabar}$.
We let 
\begin{equation}
V'= V^{\irr}\cap X^{\gen}[1/p]^{\sing}, \quad V''= V^{\irr}\cap \pi^{-1}(X^{\ps}[1/p]^{\sing})
\end{equation}
and denote by $Z'$ and $Z''$ their closures
in $X^{\gen}[1/p]$. Then $Z=Z'\cup Z'' \cup Z^{\red}$ and it is enough to show that 
$Z', Z'', Z^{\red}$ have codimension at least 
$2$ in $X^{\gen}[1/p]$ and $\pi(Z')$, $\pi(Z'')$ and $\pi(Z^{\red})$ have codimension  at least $2$ in $X^{\ps}[1/p]$. 

Since $X^{\gen}[1/p]$ is normal by \cite[Corollary 15.23]{defG} and $V^{\irr}$ is an open subscheme, $V^{\irr}$ is also normal.
Since $V'$ is the singular locus in $V^{\irr}$, we get that its codimension in $V^{\irr}$ is at least $2$. 
It follows from \eqref{alpha} that $\dim Z'=\dim V'$. 
Since 
$V^{\irr}$ is dense in $X^{\gen}[1/p]$ we have 
$\dim V^{\irr}=\dim X^{\gen}[1/p]$ and hence
$Z'$ has codimension at least $2$ in $X^{\gen}[1/p]$. Part (3) of \Cref{topo} implies that $\pi(Z')$ has codimension at least $2$ in $X^{\ps}[1/p]$.

Since $X^{\ps}[1/p]$ is normal by \Cref{normal} and $U^{\irr}$ is an open 
subscheme, $U^{\irr}$ is also normal. 
We claim that $U^{\irr}\cap \pi(Z'')$ is equal to the singular locus in $U^{\irr}$. 
Since $V^{\irr}=\pi^{-1}(U^{\irr})$, 
it follows
from the definition of $V''$ that 
it contains 
the preimage of the singular locus in $U^{\irr}$, which implies that the singular 
locus is contained in $\pi(V'')$, which is equal to $U^{\irr}\cap \pi(Z'')$ by part (2) \Cref{topo}. The other 
inclusion is given by 
\begin{equation}
U^{\irr}\cap \pi(Z'')\subset U^{\irr} \cap (\pi(Z^{\red})\cup X^{\ps}[1/p]^{\sing})= U^{\irr}\cap X^{\ps}[1/p]^{\sing}.
\end{equation}
The claim and normality of $U^{\irr}$ imply that $U^{\irr}\cap \pi(Z'')$ has codimension at least $2$. It follows from 
\eqref{alpha} that $\dim (U^{\irr}\cap \pi(Z''))= \dim \pi(Z'')$. Since $U^{\irr}$ is dense in $X^{\ps}[1/p]$ we have 
$\dim U^{\irr}=\dim X^{\ps}[1/p]$. 
Hence $\pi(Z'')$ has codimension  at least $2$ in $X^{\ps}[1/p]$. Part (3) 
of \Cref{topo} implies that $Z''$ has 
codimension  at least $2$ in
$X^{\gen}[1/p]$.

 It follows from  \cite[Proposition 13.14, 13.17]{defG} that 
 \begin{equation}\label{FQp1}
 \dim X^{\gen}[1/p] - \dim Z^{\red}\ge [F:\Qp].
 \end{equation}
 We note that the hypothesis (DIM) in the statement 
 of \cite[Proposition 13.17]{defG} is verified in 
 \cite[Theorem 13.26]{defG}. 
 If $\rho_x(\Gamma_F)$ is contained in a proper parabolic subgroup $P$ then its $H$-semi\-simp\-li\-fi\-ca\-tion takes values in a Levi subgroup $L$ of $P$. Thus $\pi(x)$ lies in the image of the natural map $X^{\ps}_{\overline{\Psi}}[1/p]\rightarrow X^{\ps}_{\Thetabar}[1/p]$, for some $\Psibar\in \cPC^{\Gamma}_L(k)$.
 Moreover, this map is finite by \cite[Corollary 6.5]{defG} and $\pi(Z^{\red})$ is the union of images taken over all proper Levi subgroups $L$ of $H$ containing a given split maximal 
 torus and all $\Psibar\in \cPC^{\Gamma}_L(k)$, which map to 
 $\Thetabar$, by construction of the absolutely irreducible locus in \cite[Section 9.2]{defG}. Since 
 \begin{equation}
 \dim X^{\ps}[1/p]= \dim H_k [F:\Qp] +\dim Z(H_k)
 \end{equation}
 by \cite[Corollary 13.31]{defG}, it follows from \cite[Proposition 2.16]{defG} that 
 \begin{equation}\label{FQp2}
 \dim X^{\ps}[1/p] - \dim \pi(Z^{\red}) \ge [F:\Qp].
 \end{equation}
 The assumption $F\neq \Qp$ implies that $[F:\Qp]\ge 2$ and this finishes the proof of parts (4) and (5).
\end{proof}
\begin{remar} If $F=\Qp$ then  \Cref{nice_locus} holds if there is no representation $\rhobar_1: \Gamma_F\rightarrow L(\kbar)$, 
where $L$ is a 
Levi subgroup of $H$ 
such that $\dim H -\dim L=2$ and 
the $H$-semisimplification of $\rhobar_1$ is equal to $\rhobar^{\mathrm{ss}}$.
However, it does not 
hold in general. For example,
if $F=\Qp$ and $G=\GL_2$ and $\rhobar$ is reducible then by examining the arguments of \cite[Lemma 3.29]{BIP_new} one sees that the
reducible locus has codimension $1$
and thus part (5) cannot hold. Nevertheless, the Proposition 
is true if we replace the bound $2$ by $1$ in parts (4) and (5) as the assumption $F\neq \Qp$ is only invoked
when using the bounds \eqref{FQp1} and \eqref{FQp2}.
\end{remar}

\begin{remar}
\label{much_ado} At the beginning of the section we have 
made an assumption that the homomorphism $\alpha:\Gamma_F\rightarrow (H/H^0)(\kbar)$ associated to $\Thetabar$ is surjective. If $\alpha$ is not surjective
then \Cref{nice_locus} still holds after tweaking the definitions of the irreducible locus in $X^{\gen}$ and $X^{\ps}$. Namely, let $H_1$ be a closed subgroup scheme of $H$ such that their neutral components are the same and $(H_1/H^0)(\kbar)$ is equal to the image of $\alpha$. Then we may regard $\Thetabar$ as an $H_1$-pseudocharacter, and \cite[Lemma 9.2]{defG} implies that $X^{\gen}_{\Thetabar, H_1} = X^{\gen}_{\Thetabar, H}$ and $X^{\ps}_{\Thetabar, H_1}= X^{\ps}_{\Thetabar, H}$. Since \Cref{nice_locus} applies to the pair $(H_1, \Thetabar)$ it will also hold for $(H, \Thetabar)$ if we let $V^{\irr}_H:= V^{\irr}_{H_1}$ and $U^{\irr}_H:= U^{\irr}_{H_1}$. With these definitions \Cref{intro_nice_locus} stated 
in the introduction holds for all pairs 
$(H, \Thetabar)$, the only caveat being that representations corresponding to points in $V^{\irr}_H$ will not be $H$-irreducible, unless $H=H_1$, as we require\footnote{We explain our reasons at the end of \cite[\S 9]{defG}.} the surjectivity on the component group as part of the definition of irreducibility.
\end{remar} 

\begin{cor}\label{cover_nice_locus} Let $U\subset X^{\ps}[1/p]$ be 
an open subscheme as in \Cref{nice_locus}. There
is an \'etale covering $\{U_i\rightarrow U\}_{i\in I}$ by affine schemes 
$U_i=\Spec A_i$ and a family of representations $\{\rho_i: \Gamma_F \rightarrow H(A_i)\}_{i\in I}$ such that 
$\Theta_{\rho_i}= \Theta^u_{|A_i}$. 
\end{cor}
\begin{proof} Since $\pi|_V: V\rightarrow U$
is smooth and surjective, it has sections \'etale locally on $U$ by \cite[Corollaire 17.16.3 (ii)]{EGA4_4}. The sections 
$s_i: \Spec A_i \rightarrow V \rightarrow X^{\gen}$ give rise to $\rho_i\in X^{\gen}(A_i)$.
\end{proof}

\section{Relative analytification}\label{sec_rel_an}
We will introduce an analytification functor for schemes 
locally of finite type over $\Spec R[1/p]$, where $R$ is a complete local noetherian $\OO$-algebra with finite residue field, prove its universal 
property in \Cref{universal},  and will relate it 
to the analytification functor of K\"opf for schemes  
locally of finite type over affinoids. 

The main result of this section is \Cref{etale_cover}, which shows that if $Y$ is a scheme locally of finite type over $\Spec R[1/p]$ then the functor
$U\mapsto \OO_{U^{\ana}}(U^{\ana})$ is a sheaf 
of rings for the \'etale topology on $Y$. We prove this result by reducing to the case, when
$Y$ is a scheme locally of finite type over
$\Spec A$, where $A$ is an affinoid algebra,
which is handled in \Cref{etale_cover_aff} by 
using Berkovich spaces. For the application 
it is important that the functor $Y\mapsto Y^{\ana}$ 
preserves normality. We establish this in \Cref{normal_rig} 
by showing that the completions of the local rings $\OO_{Y^{\ana},y}$ 
and $\OO_{Y, \pi(y)}$
are isomorphic. 
\subsection{Schemes over affinoids}
Let $A$ be an $L$-affinoid algebra in the sense of \cite[Definition 6.1.1/1]{BGR}, so that $A$ is 
a quotient of a Tate algebra $T_n$ over $L$ for 
some $n\ge 0$. When working with 
Berkovich spaces we will add the adjective 
\textit{strictly} to distinguish them from more general affinoid algebras in Berkovich's theory. Let $\varphi: Y \rightarrow \Spec A$ 
be a scheme locally of finite type over $A$. We will recall some
results of K\"opf \cite{Koepf} on the relative analytification of $Y$. 
A modern reference for this is \cite[Section 2]{conrad_ample}. 
By \cite[Satz 1.4]{Koepf}, reproved in \cite[Example 2.2.11]{conrad_ample}, there 
exists a rigid analytic space $\varphi^{\ana}: Y^{\ana} \rightarrow \Sp(A)$ together with a morphism $\pi: Y^{\ana}\rightarrow Y$ of 
locally $G$-ringed  spaces over $A$, such that for every 
rigid analytic space $T\rightarrow \Sp(A)$ we have a natural 
bijection 
\begin{equation}\label{universal_koepf}
\Hom_{\Sp(A)}(T, Y^{\ana})\cong \Hom_{\Spec A}(T, Y), \quad f \mapsto \pi\circ f,
\end{equation}
where the left-hand-side are morphisms of rigid analytic spaces over 
$\Sp(A)$ and the right-hand-side are morphisms of locally $G$-ringed  spaces over $A$. We refer the reader to \cite[Definition 2.2.2]{conrad_ample} for the definition of the category of locally $G$-ringed  spaces over $A$. This category contains the category of 
locally ringed spaces (and hence the category of schemes) over $\Spec A$ and the category of rigid analytic varieties over $\Sp(A)$ as full subcategories. 

A similar analytification 
construction is carried out in \cite[Section 2.6]{berkovich_ihes}
in the setting of Berkovich analytic spaces, where
the universal property can be formulated by working in 
the category of locally ringed spaces. 
We will denote the resulting analytification 
by $\varphi^{\Ber}: Y^{\Ber}\rightarrow \MM(A)$.

\begin{prop}\label{points} Let $Y$ be a scheme locally of finite type over $\Spec A$, where $A$ is an affinoid algebra. The map $\pi: Y^{\ana}\rightarrow Y$ induces a bijection 
between $Y^{\ana}$ and the set of points in $Y$ with the residue field 
a finite extension of $L$. Moreover, if $y\in Y^{\ana}$ then the 
natural map $\widehat{\OO}_{Y, \pi(y)} \rightarrow \widehat{\OO}_{Y^{\ana}, y}$ is an isomorphism of complete local rings.
\end{prop}
\begin{proof} \cite[Satz 2.1]{Koepf}, \cite[Example 2.2.11]{conrad_ample}. 
\end{proof} 
Using the excellence of the local rings of rigid analytic spaces one
may transfer information between the local rings $\OO_{Y^{\ana}, y}$
and $\OO_{Y, \pi(y)}$ and their completions, obtaining the following
corollaries. 
 \begin{cor} $\dim Y^{\ana}=\dim Y$.
\end{cor}
\begin{proof} \cite[Folgerung 2.5]{Koepf}.
\end{proof}

\begin{cor} $Y^{\ana}$ is normal (resp. reduced) if and only if $Y$ is 
normal (resp. reduced).
\end{cor}
\begin{proof} \cite[Folgerung 2.5, 2.6]{Koepf}. 
\end{proof}

\begin{lem}\label{ana_Ber} Let $A$ be a strict $L$-affinoid algebra
and let $Y \rightarrow \Spec A$ be a finite type morphism of schemes.
Let $Y^{\ana}$ be the analytification of $Y$ as a rigid analytic
variety and let $Y^{\Ber}$ be the analytification of $Y$ as
a Berkovich space. Then $\OO(Y^{\ana})=\OO(Y^{\Ber})$.
\end{lem}
\begin{proof} Let $\mathbb A^1_A=\Spec A[t]$. It follows 
from the universal property of analytification, \cite[Theorem 2.2.5 (1)]{conrad_ample} and \cite[Proposition 2.6.1]{berkovich_ihes}, 
that 
\begin{equation}\label{ana1}
\OO(Y^{\ana}) \cong \Hom_{\Spec A}(Y^{\ana}, \mathbb A^1_A)\cong \Hom_{\Sp(A)}(Y^{\ana}, (\mathbb A^1_A)^{\ana})
\end{equation}

\begin{equation}\label{ana2}
\OO(Y^{\Ber}) \cong \Hom_{\Spec A}(Y^{\Ber}, \mathbb A^1_A)\cong \Hom_{\MM(A)}(Y^{\Ber}, (\mathbb A^1_A)^{\Ber})
\end{equation}
where $\Hom_{\Spec A}$ stands for morphisms of locally ($G$-)ringed  spaces over $\Spec A$.

It follows from the construction of the 
analytification in the references cited above that $Y^{\ana}$ and $Y^{\Ber}$  (resp. 
$(\mathbb A^1_A)^{\ana}$ and $(\mathbb A^1_A)^{\Ber}$) correspond to each other under 
\cite[Theorem 1.6.1]{berkovich_ihes} when $Y$ is separated, for then $Y^{\Ber}$ is Hausdorff and paracompact. In particular,
we have a natural isomorphism 
\begin{equation}\label{ana3}
\Hom_{\MM(A)}(Y^{\Ber}, (\mathbb A^1_A)^{\Ber})\cong 
\Hom_{\Sp(A)}(Y^{\ana}, (\mathbb A^1_A)^{\ana}).
\end{equation}
This also holds for general $Y$ by a simple descent argument.
Putting \eqref{ana1}, \eqref{ana2} and \eqref{ana3} together we obtain the claim. 
\end{proof}

\begin{prop}
\label{etale_cover_aff} Let $A$ be a strict 
$L$-affinoid algebra and let $\{\varphi_i: U_i\rightarrow \Spec A\}_{i\in I}$ be a covering for the \'etale topology on $\Spec A$. Then the  following equaliser diagram 
\begin{equation}\label{equaliser}
 A \rightarrow \prod_{i\in I} \OO(U^{\ana}_i) \rightrightarrows \prod_{(i,j)\in I \times I} \OO(U^{\ana}_{ij})
\end{equation}
is exact, where $U_{ij}= U_i \times_{\Spec A} U_j$. 
\end{prop}
\begin{proof} We claim that  $\{\varphi^{\Ber}_i: U_i^{\Ber}\rightarrow \MM(A)\}_{i\in I}$ 
is an \'etale covering of $\MM(A)$. Since  
\begin{equation}
U_{ij}^{\Ber}\cong U_i^{\Ber}\times_{\MM(A)} U_j^{\Ber}.
\end{equation}
by \cite[Proposition 2.6.1]{berkovich_ihes}
and $\OO_X$ is a sheaf for the \'etale topology on 
Berkovich spaces, \cite[Example 4.1.6]{berkovich_ihes}, 
the claim implies that the equaliser diagram 
\begin{equation}\label{equaliser2}
 A \rightarrow \prod_{i\in I} \OO(U^{\Ber}_i) \rightrightarrows \prod_{(i,j)\in I \times I} \OO(U^{\Ber}_{ij})
\end{equation}
is exact and \Cref{ana_Ber}
implies that \eqref{equaliser} is exact. 

The claim is proved in the beginning of \cite[Section 4.1]{berkovich_ihes}. We will spell out the details. We note that $\varphi_i^{\Ber}$ are
\'etale by \cite[Proposition 3.3.11]{berkovich_ihes}.
Hence, it is enough to show that 
\begin{equation}\label{union}
\MM(A)= \bigcup_{i\in I} \varphi_i^{\Ber}(U^{\Ber}_i).
\end{equation}
Let $U$ be the disjoint union of the $U_i$ and let $\varphi: U\rightarrow \Spec A$ be 
a morphism such that $\varphi|_{U_i}= \varphi_i$.  
 Since $\varphi$ is surjective 
\cite[Proposition 2.6.8]{berkovich_ihes} 
implies that $\varphi^{\Ber}$ is surjective. Since
$U^{\Ber}$ is a disjoint union of $U_i^{\Ber}$ we 
obtain \eqref{union}.
\end{proof}

\subsection{\texorpdfstring{Schemes over $\Spec R[1/p]$}{Schemes over Spec R[1/p]}}
We will define a variant of the analytification functor considered above. Let $R$ be a complete local noetherian $\OO$-algebra with residue field $k$ and maximal ideal $\mm$ and let $Y$ be a scheme locally of finite type over $X:=\Spec R[1/p]$.
We will construct  a relative analytification $Y^{\ana}\rightarrow \mathfrak X$ of $Y$, where $\mathfrak X:=(\Spf R)^{\rig}$ is the Berthelot's rigid generic
fibre, \cite[Section 7]{deJong}. In particular, $\mathfrak X$ is a rigid analytic space over $L$.

We proceed as follows. 
We may write $\mathfrak X$ as an increasing 
union of  open affinoid
subdomains $\mathfrak X= \bigcup_{i\in I} \Sp(A_i)$.
For each $i$, let $Y_i$ be the fibre product 
of $\Spec A_i$ and $Y$ over $\Spec R$. Then $Y_i$ is a scheme 
locally of finite type over $\Spec A_i$ and we obtain 
a rigid analytic space $Y_i^{\ana}\rightarrow \Sp(A_i)$ satisfying the 
universal property \eqref{universal_koepf} as well as a morphism 
of locally G-ringed spaces $\pi_i: Y_i^{\ana}\rightarrow Y_i$. 
If $j<i$ then $Y_j = \Spec A_j \times_{\Spec A_i} Y_i$ and the universal property implies 
$Y_j^{\ana}$ is canonically isomorphic to 
$\Sp(A_j)\times_{\Sp(A_i)} Y_{i}^{\ana}$. 
We glue $\{Y_i^{\ana}\}_{i\in I}$ 
using these isomorphisms to obtain
$Y^{\ana}$. 

It follows from the construction that $\Sp(A_i)\times_{\mathfrak X} Y^{\ana}\cong Y_i^{\ana}$. 
Since the covering of $\mathfrak X$ is admissible, for any morphism 
$\Sp(B)\rightarrow \mathfrak X$ in $\Rig_L$ there exists $i\in I$ such that it factors 
through $\Sp(A_i)\rightarrow \mathfrak X$. We thus have 
\begin{equation}\label{is_ok}
\Sp(B)\times_{\mathfrak X} Y^{\ana}\cong \Sp(B)\times_{\Sp A_i} Y_i^{\ana}\cong (\Spec B \times_{X} Y)^{\ana},
\end{equation}
where the last isomorphism follows from the universal property of the 
analytification of schemes locally of finite type over $\Spec B$. We deduce from \eqref{is_ok} that the construction 
of $Y^{\ana}$ is independent of the covering. 

\begin{remar}
In practice, the schemes we deal with in this paper are affine and we think of the analytification as follows.  If $Y=\Spec S[1/p]$, where $S=R[y_1,\ldots, y_n]$ 
and $R=\OO\br{x_1,\ldots, x_d}$  then $Y^{\ana}$ is a product $\mathbb{D}^d \times_{\Sp(L)} (\mathbb A^n_L)^{\ana}$, where $\mathbb{D}^d$ is an open $d$-dimensional unit disk and $(\mathbb A^n_L)^{\ana}$ 
is the analytification of an $n$-dimensional affine space over $L$, 
\cite[Example 9.3.4/1]{BGR}. In general, we may present 
$$R=\OO\br{x_1, \ldots, x_d}/(f_1, \ldots, f_s), \quad S=R[y_1,\ldots, y_n]/(g_1,\ldots, g_m)$$ and then $Y^{\ana}$ 
is the closed locus inside $\mathbb{D}^d \times_{\Sp(L)} (\mathbb A^n)^{\ana}$ cut out by these equations. 
\end{remar}

\begin{lem}\label{clofib} The functor $Y\mapsto Y^{\ana}$ from the category of  schemes locally of finite type 
over $X$ to the category of rigid analytic spaces over $\mathfrak X$ maps open (resp. closed) immersions to open (resp. closed) immersions and commutes with fibre products.
\end{lem}
\begin{proof} It is enough to verify these assertions after base changing to $\Spec A_i$. It follows from the isomorphism \eqref{is_ok} that we are then in the setting considered by K\"opf. The assertion about open immersions is \cite[Folgerung 1.3 2)]{Koepf}, the assertion 
about closed immersions is \cite[Hilfsatz 2.10]{Koepf} and the assertion about the
fibre products is \cite[Satz 1.8]{Koepf}.
\end{proof}

\begin{prop}\label{algebraic} Let $A$ be 
an affinoid $L$-algebra. There is a natural isomorphism 
$$\Hom_{\Rig_L}(\Sp(A), \mathfrak X)\cong \Hom_{\OO\text{-}\alg}(R, A).$$
Moreover, every $\varphi\in \Hom_{\OO\text{-}\alg}(R, A)$ is continuous
for the $\mm$-adic topology on $R$ and the Banach space topology on $A$.
\end{prop}

\begin{proof}
    First, assume $R=\OO\br{x_1,\ldots,x_n}$, so that $\mathfrak X=\mathbb D^n$.
    Let $\varphi:R\to A$ be an $\OO$-algebra homomorphism.
    For every maximal ideal $\mathfrak q$ of $A$, the residue field $A/\mathfrak q$ is a finite extension of $L$.
    By \cite[Lemmas 3.17 and 3.18(4)]{BIP_new}, the image of each $x_i$ in this field has absolute value $<1$.
    The maximum principle \cite[Theorem 3.1/15]{bosch} implies that the supremum norm $|\varphi(x_i)|_{\sup}<1$. 
    Hence each $\varphi(x_i)$ is topologically nilpotent by \cite[Corollary 3.1/18]{bosch}.
    So substitution defines an $\mm$-adically continuous $\OO$-algebra homomorphism $\psi : R \to A$ and we must show that $\psi = \varphi$.
    
    Fix a maximal ideal $\mathfrak q$ of $A$ and an integer $r \geq 1$.
    Let $I = \ker(R \xrightarrow{\varphi} A \to A/\mathfrak q)$.
    Then $R/I^r$ is a finite $\OO$-algebra.
    The composite $R\xrightarrow{\varphi}A\to A/\mathfrak q^r$ annihilates $I^r$, hence it factors over $R/I^r$. Hence this composite is continuous.
    It agrees with the composite $R\xrightarrow{\psi}A\to A/\mathfrak q^r$, as it agrees on the dense subalgebra $\OO[x_1, \ldots, x_n]$ of $R$.
    
    It follows that for all maximal ideals $\mathfrak q$ of $A$ and every $a \in R$, we have $\varphi(a) - \psi(a) \in \bigcap_{r \geq 1} \mathfrak q^r$.
    By Krull's intersection theorem $\varphi(a) - \psi(a)$ vanishes in $A_{\mathfrak q}$.
    We conclude that $\psi = \varphi$, hence $\varphi$ is continuous.
    It follows that $\varphi$ arises from a morphism $\Sp(A) \to \mathfrak X$, which is uniquely determined by the coordinate functions $\varphi(x_i)$.
    This proves the proposition when $R=\OO\br{x_1,\ldots,x_n}$.
    
    In the general case we choose a surjection 
    $\pi:\OO\br{x_1,\ldots, x_n}\twoheadrightarrow R$ of $\OO$-algebras. This yields a closed immersion 
    $\mathfrak X \subset \mathbb D^n$. Then 
    $\varphi: R \rightarrow  A$ induces a morphism 
    $\Sp A \rightarrow \mathbb D^n$ by the previous part, which factors as $\Sp(A)\rightarrow \mathfrak X\subset \mathbb D^n$. The $\mm$-adic topology on $R$ coincides with the quotient topology defined by $\pi$ and the continuity of $\varphi$ follows.
\end{proof}

If $T$ is a rigid analytic space over $L$ then the ring 
of global functions $\OO(T)$ carries a natural topology, 
such that if $T= \Sp A$ an affinoid then the topology is just
the Banach space topology on $A$ and if $T$ is arbitrary and 
$\{U_i\rightarrow T\}_{i\in I}$ is a covering of $T$ by affinoids for the analytic Grothendieck topology on rigid analytic
spaces then the natural topology on $\OO(T)$ coincides
with the subspace topology induced by $\OO(T)\rightarrow \prod_{i\in I} \OO_{U_i}(U_i)$ for the Banach space topology in $\OO_{U_i}(U_i)$ and the product topology on the target, 
\cite[Section 4.1]{DPS}. 

\begin{cor}\label{bsch} There is a natural isomorphism
$$\Hom_{\Rig_L}(T, \mathfrak X)\cong \Hom_{\OO\text{-}\alg}(R, \OO(T)).$$
Moreover, every $\varphi\in \Hom_{\OO\text{-}\alg}(R, \OO(T))$ is continuous.
\end{cor} 
\begin{proof} If $T$ is affinoid then the assertion follows from 
\Cref{algebraic}. The general case follows by considering an admissible 
covering of $T$ by affinoids. 
\end{proof}

\begin{cor}\label{bsch_aff} Let $Y=\Spec S[1/p]$, where  $S$ is a finite type $R$-algebra. Then there is a natural isomorphism $\Hom_{\Rig_L}(T, Y^{\ana})\cong \Hom_{\OO\text{-}\alg}(S, \OO(T))$.
\end{cor}
\begin{proof} It is enough to deal with the case, when $S=R[x_1, \ldots, x_n]$.
Then $Y^{\ana}\cong \mathfrak X\times_{\Sp(L)} (\mathbb A^n_L)^{\ana}$, where $(\mathbb A^n_L)^{\ana}$ is the analytification of the $n$-dimensional affine space over $L$. We have
\begin{equation}
\begin{split}
\Hom_{\Rig_L}(T, Y^{\ana})&\cong \Hom_{\Rig_L}(T, \mathfrak X)\times \Hom_{\Rig_L}(T, (\mathbb A^n_L)^{\ana})\\
&\cong \Hom_{\Rig_L}(T, \mathfrak X)\times \OO(T)^n\\
&\cong\Hom_{\OO\text{-}\alg}(R, \OO(T))\times \OO(T)^n\\
& \cong\Hom_{\OO\text{-}\alg}(S, \OO(T)),
\end{split}
\end{equation}
where the first isomorphism follows from the universal property of fibre products, the second follows from the properties of analytification functor over $L$, \cite[Lemma 5.4/2]{bosch}, the 
third isomorphism is given by \Cref{bsch} and the 
last one follows from the universal property of polynomial rings. 
\end{proof} 

\begin{thm}\label{universal} Let $(Y, \OO_Y)$ be a scheme locally of finite type over $X=\Spec R[1/p]$. Then there is a 
morphism of locally $G$-ringed  
$L$-spaces $(\pi, \pi^*): (Y^{\ana}, \OO_{Y^{\ana}})\rightarrow (Y, \OO_Y)$ satisfying the
following universal property: 

Given $(T, \OO_T)\in \Rig_L$ and a morphism of locally $G$-ringed $L$-spaces $(T, \OO_T)\rightarrow
(Y, \OO_Y)$, the latter factors through 
$(\pi, \pi^*)$ via the unique morphism of $L$-rigid spaces $(T, \OO_T)\rightarrow (Y^{\ana}, \OO_{Y^{\ana}})$.
\end{thm}
\begin{proof} If $R=\OO$ then $R[1/p]=L$ and the
statement is proved in \cite[Proposition 5.4/4]{bosch}.
\Cref{algebraic} enables us to carry out the same 
argument for more general $R$. In particular, if 
$Y$ is affine then the assertion follows from \Cref{bsch_aff}, which is the analog of \cite[Lemma 5.4/2]{bosch}. We now proceed as in the proof 
of \cite[Proposition 5.4/4]{bosch}: we cover $Y$ by 
affine open subschemes $U_i$ and then glue $U_i^{\ana}$ along $(U_i\cap U_j)^{\ana}$ to get 
a rigid analytic space $Z$ with the required 
universal property. To verify that $Z$ is canonically isomorphic to
$Y^{\ana}$ one may  observe that if $\Sp(B)\rightarrow \mathfrak X$ is 
a morphism in $\Rig_L$ then $\Sp(B)\times_{\mathfrak X} Z$ satisfies the
universal property of the analytification of $Y\times_X \Spec B$ over
$\Sp(B)$, \cite[Satz 1.4]{Koepf}, and then the isomorphism follows
from \eqref{is_ok}.
\end{proof} 
\begin{cor}\label{rel_an_fun} The relative analytification $Y\mapsto Y^{\ana}$ 
defines a functor from the category of schemes locally of finite type over $X$
to the category of rigid analytic spaces over $\mathfrak X$.
\end{cor}
\begin{proof} Let $\varphi: Y_1\rightarrow Y_2$ be a morphism of schemes
locally of finite type over $X$. Then the composition $Y_1^{\ana}\rightarrow 
Y_1 \rightarrow Y_2$ is a morphism of  locally ringed $G$-spaces over $X$. 
It induces a morphism $\varphi^{\ana}: Y_1^{\ana} \rightarrow Y_2^{\ana}$ by \Cref{universal}. It is clear from the construction that the analytification 
of morphisms is compatible with composition.
\end{proof}

\begin{prop}\label{deJong_like} The morphism $\pi: Y^{\ana}\rightarrow Y$ 
induces a bijection between $Y^{\ana}$ and the set of closed points of $Y$. Moreover, if $y\in Y^{\ana}$ maps to $\pi(y)\in Y$ then 
the completions of the 
local rings $\OO_{Y^{\ana}, y}$ and $\OO_{Y, \pi(y)}$ 
 are naturally isomorphic. 
\end{prop}
\begin{proof} If $Y=X=\Spec R[1/p]$ then the assertion is 
proved in \cite[Lemma 7.1.9]{deJong}. One may deduce the proposition
by combining this result with \Cref{points}. However, we will give 
a different argument  by following the proof of \cite[Theorem 2.2.5 4), Example 2.2.11]{conrad_ample}.

It follows from \cite[Lemma 3.17, 3.18(4)]{BIP_new} that the set of  closed points of $Y$ coincides with  the set of points of $Y$ such that the residue field
is a finite extension of $L$. The assertion about bijection follows 
by applying \Cref{universal} with  $T=\Sp(L')$ for all finite extensions $L'$ of $L$. The assertion about completions follows from applying 
\Cref{universal} with $T=\Sp(B)$ for all finite-dimensional artinian $L$-algebras $B$.
\end{proof}

\Cref{deJong_like} together with the excellence
of local rings of $Y^{\ana}$ and $Y$ allows us to argue 
as in \cite[Section 2.1]{Koepf}, \cite[Section 1.2]{conrad} and
implies the following corollaries. 

\begin{cor}\label{dim_rig} $\dim Y^{\ana}=\dim Y$.
\end{cor} 

\begin{cor}\label{normal_rig}  $Y^{\ana}$ is normal (resp. reduced) if and only if 
$Y$ is normal (resp. reduced).
\end{cor} 
\begin{prop}\label{etale_cover} Let $Y$ be a scheme locally of finite type over $X$. Let $\{U_i\rightarrow Y\}_{i\in I}$ be an \'etale covering of $Y$ by $X$-schemes. Then the
equaliser diagram 
\begin{equation}\label{eq0}
\OO(Y^{\ana})\rightarrow \prod_{i\in I} \OO(U_i^{\ana})\rightrightarrows
\prod_{(i,j)\in I\times I} \OO(U_{ij}^{\ana})
\end{equation}
is exact, where $U_{ij}= U_i\times_Y U_j$.
\end{prop}

\begin{proof} 
Let $\{\Sp(A_m)\rightarrow Y^{\ana}\}_{m\in M}$ be an admissible covering by affinoid subdomains. For each 
pair $(m,n)\in M\times M$ let 
 $\{\Sp(A_{mnt})\}_{t\in M_{mn}}$
be an admissible covering of $\Sp(A_m)\times_{Y^{\ana}} \Sp(A_n)$
by affinoid subdomains. Then the equaliser diagram 
\begin{equation}\label{eq1}
\OO(Y^{\ana})\rightarrow \prod_{m\in M} A_m \rightrightarrows \prod_{(m,n)\in M\times M, t\in M_{mn}} A_{mnt}
\end{equation}
is exact. For each $(i,m)\in I\times M$ let $U_{im}= \Spec A_m\times_{Y} U_i$ and let $U_{im}^{\ana}$ be the relative 
analytification of $U_{im}$ over $\Sp(A_m)$. Then 
 $\{U_{im}\rightarrow \Spec A_m \}_{i\in I}$ is an \'etale covering of $\Spec A_m$ and \Cref{etale_cover_aff} gives an exact equaliser diagram: 
\begin{equation}\label{eq2}
A_{m}\rightarrow \prod_{i\in I} \OO(U_{im}^{\ana})\rightrightarrows \prod_{(i,j)\in I\times I} \OO(U_{imjm}^{\ana}),
\end{equation}
where $U_{imjm}$ is the fibre product of $U_{im}$ and $U_{jm}$ over $\Spec A_{m}$. (We will employ the same notation scheme with more complicated index sets below.) Similarly, for $(m, n)\in M\times M$ and $t\in M_{mn}$ and $i\in I$ we define $U_{imnt}=\Spec A_{mnt}\times_Y U_i$ and obtain an exact equaliser diagram: 
\begin{equation}\label{eq3}
A_{mnt}\rightarrow \prod_{i\in I} \OO(U_{imnt}^{\ana})\rightrightarrows \prod_{(i,j)\in I\times I} \OO(U_{imntjmnt}^{\ana}).
\end{equation}
It follows from the construction of $U_i^{\ana}$ that 
$U_{im}^{\ana}\cong \Sp(A_m)\times_{Y^{\ana}} U_i^{\ana}$. 
Since $\{\Sp(A_m)\}_{m\in M}$ is an admissible covering of $Y^{\ana}$, 
$\{U_{im}^{\ana}\}_{m\in M}$ is an admissible covering of $U_i^{\ana}$. Similarly $\{U_{imnt}^{\ana}\}_{t\in M_{mn}}$ is an admissible covering 
for $U_{im}^{\ana}\times_{U_i^{\ana}}U^{\ana}_{in}$. We thus obtain an exact equaliser diagram: 
\begin{equation}\label{eq4} 
\OO(U_i^{\ana})\rightarrow \prod_{m\in M} \OO(U_{im}^{\ana})\rightrightarrows
\prod_{(m,n)\in M\times M, t\in M_{mn}}\OO(U_{imnt}^{\ana}).
\end{equation}
The assertion follows by considering 
the diagram 
\begin{equation} 
\eqref{eq0}\rightarrow \prod_{m\in M} \eqref{eq2}\rightrightarrows \prod_{(m,n)\in M\times M, t\in M_{mn}} \eqref{eq3}
\end{equation}
obtained by restriction maps:
\begin{equation}
\begin{tikzcd}
\OO(Y^{\ana})\arrow[r]\arrow[d] & \prod_{i} \OO(U_i^{\ana})\arrow[r, shift left]
        \arrow[r, shift right]\ar[d] & 
\prod_{i,j} \OO(U_{ij}^{\ana})\arrow[d] \\
\prod_m A_{m}\arrow[r]
\arrow[d, shift left]
        \arrow[d, shift right]& \prod_{i, m} \OO(U_{im}^{\ana})\arrow[r, shift left]
        \arrow[r, shift right]\arrow[d, shift left]
        \arrow[d, shift right]& \prod_{i,j,m} \OO(U_{imjm}^{\ana}) \arrow[d, shift left]
        \arrow[d, shift right]\\
        \prod_{m,n,t}A_{mnt}\arrow[r] &\prod_{i,m,n,t} \OO(U_{imnt}^{\ana})\arrow[r, shift left]
        \arrow[r, shift right]& \prod_{i,j,m,n,t} \OO(U_{imntjmnt}^{\ana})
  \end{tikzcd}
\end{equation}
so that \eqref{eq0} is the first row, $\prod_m \eqref{eq2}$ is the second row, $\prod_{m,n,t} \eqref{eq3}$ is the third row, \eqref{eq1} is the first column and $\prod_{i}\eqref{eq4}$ is the second column. The exactness of \eqref{eq0} follows from the exactness of the first and second columns and the second and third rows in the diagram. 
\end{proof}

Let $\mathfrak a$ be an ideal of $R$. Let $Z= \Spec (R/\mathfrak a)[1/p]$, $\mathfrak Z:=(\Spf R/\mathfrak a)^{\rig}$, $U=X\setminus Z$  and $U^{\rig}=\mathfrak X\setminus \mathfrak Z$. Then $U$ is an 
open subscheme of $X$ and $U^{\rig}$ is a Zariski open rigid subvariety 
of $\mathfrak X$. 
\begin{prop}\label{complement} We have 
\begin{equation}
(U\times_X Y)^{\ana}\cong U^{\rig}\times_{\mathfrak X} Y^{\ana}, \quad 
(Z\times_X Y)^{\ana}\cong \mathfrak Z \times_{\mathfrak X} Y^{\ana}.
\end{equation}
Moreover, $(U\times_X Y)^{\ana}$ is the complement of $(Z\times_X Y)^{\ana}$ in $Y^{\ana}$.
\end{prop}
\begin{proof} The complement of $\mathfrak Z \times_{\mathfrak X} Y^{\ana}$ in $Y^{\ana}$ is equal to $U^{\rig}\times_{\mathfrak X} Y^{\ana}$, as
both are open rigid subvarieties of $Y^{\ana}$ with the same underlying topological space. The same argument 
shows that $U^{\rig}=U^{\ana}$. 
If $\{\Sp(A_i)\}_{i\in I}$ is an admissible covering of $\mathfrak X$ 
then $\Spec A_i\times_X Z\cong \Spec A_i/\mathfrak a A_i$. Thus 
$(\Spec A_i\times_X Z)^{\ana}= \Sp(A_i/\mathfrak a A_i)$, which implies that $Z^{\ana}=\mathfrak Z$. The assertion follows from \Cref{clofib}.
\end{proof}
\section{Construction of the infinitesimal character}\label{sec_constr}

Let $G$ be a connected reductive group defined 
over $F$ and split over a finite Galois extension $E$ of $F$. Its $L$-group is a generalised reductive group scheme $\LG$ defined over $\Spec \ZZ$ of the
form $\Ghat\rtimes \underline{\Gal(E/F)}$, where 
$\Ghat$ is a split connected reductive group scheme over $\ZZ$ 
with its based root datum dual to that of $G_E$, $\underline{\Gal(E/F)}$ is a constant group scheme associated to the finite group $\Gal(E/F)$. We refer the reader to \cite[Section 2]{DPS} for the construction of the action of $\underline{\Gal(E/F)}$
on $\Ghat$. Our definition of the 
$L$-group depends on the choice of 
$E$, and we set things up this way, because we want $\LG$ to be of finite type over $\Spec \ZZ$. This choice is completely harmless, as one can easily check that replacing $E$ by a finite extension does not affect the statements.  

If $A$ is a ring then 
a  representation $\rho: \Gamma_F \rightarrow \LG(A)$ is called \textit{admissible} if the 
composition of $\rho$ with the natural projection $\LG(A)\rightarrow \underline{\Gal(E/F)}(A)$ 
maps $g\in \Gamma_F$ to the constant function 
$\Spec A\rightarrow \Gal(E/F)$ with value $q(g)$, 
where $q: \Gamma_F \rightarrow \Gal(E/F)$ is the quotient map. 

Let $X$ be a rigid analytic space over $L$.
As explained before \Cref{bsch}, the ring 
of global functions $\OO_X(X)$ carries a natural topology, 
such that if $X= \Sp A$ an affinoid then the topology is just
the Banach space topology on $A$.

Let $\rho: \Gamma_F \rightarrow \LG(\OO_X(X))$ be a continuous 
admissible representation. Following  \cite[Definition 4.1]{DPS} we attach to $\rho$ a $\Qp$-algebra homomorphism 
$$\zeta_{\rho}: Z(\Res_{F/\Qp} \mathfrak g)\rightarrow \OO_X(X),$$
where $\mathfrak g$ is the Lie algebra of $G$, $\Res_{F/\Qp} \mathfrak g$ is the Lie algebra
of the restriction of scalars $\Res_{F/\Qp} G$, 
and $Z(\Res_{F/\Qp} \mathfrak g)$ is the centre of the 
universal enveloping algebra of 
$\Res_{F/\Qp} \mathfrak g$.

Let us sketch the construction of $\zeta_{\rho}$. Let $\{U_i\rightarrow X\}_{i\in I}$
be an admissible covering of $X$ by affinoids $U_i=\Sp(A_i)$. 
The map $\OO_X(X)\rightarrow A_i$ induces a continuous admissible 
representation $\rho_i: \Gamma_F \rightarrow \LG(A_i)$. 
 The first step 
is to attach to $\rho_i$ the Sen operator $\Theta_{\Sen, \rho_i}\in (\Cp \wtimes A_i)\otimes_{\Qp} \ghat$, where $\ghat$ 
is the Lie algebra of $\Ghat_{\Qp}$, \cite[Lemma 4.6]{DPS}.  This induces a homomorphism of $\Qp$-algebras $S(\ghat^*)\rightarrow \Cp \wtimes A_i$, 
where $S(\ghat^*)$ is the symmetric algebra of the 
$\Qp$-linear dual of $\ghat$. By restricting this homomorphism to the subring $S(\ghat^*)^{\Ghat}\subset S(\ghat^*)$ we obtain 
a homomorphism of $\Qp$-algebras 
$\theta_{\rho_i}: S(\ghat^*)^{\Ghat} \rightarrow \Cp\wtimes A_i$.
It is shown in \cite[Lemma 4.10]{DPS} that  $\theta_{\rho_i}$ takes values in $E\otimes_{\Qp} A_i$, where $E$ is any finite 
Galois extension of $F$ such that $G$ splits over $E$. 
The construction of the Sen operator is functorial in $U_i$ and 
hence the homomorphisms  $\theta_{\rho_i}$ glue to a homomorphism 
of $\Qp$-algebras $\theta_{\rho}: S(\ghat^*)^{\Ghat} \rightarrow E \otimes_{\Qp} \OO_X(X)$. Let us assume that $L$ is large enough 
so that there are $[E:\Qp]$ embeddings of $E$ into $L$ and let
$\sigma: F\hookrightarrow L$ be a $\Qp$-algebra homomorphism 
and $\tau: E \hookrightarrow L$ be any  $\Qp$-algebra homomorphism 
extending $\sigma$. There is an $L$-algebra homomorphism 
$\zeta_{\rho, \sigma}: Z(\mathfrak g)\otimes_{F, \sigma} L \rightarrow \OO_X(X)$ defined as the composition
\begin{equation}\label{one}
Z(\mathfrak g)\otimes_{F, \sigma} L \overset{\cong}{\longrightarrow} S(\that^*)^W_L\overset{\cong}{\longrightarrow}S(\ghat^*)^{\Ghat}_L
               \overset{\theta_{\rho}}{\longrightarrow} E\otimes \OO_X(X)
\overset{m_{\tau}}{\longrightarrow} \OO_X(X),
\end{equation}
where the first arrow comes from the Harish-Chandra isomorphism, the 
second is Chevalley's restriction theorem and the last arrow is the map $x\otimes a\mapsto \tau(x)a$. Although the first and the last 
arrow depend on the choice of $\tau$ it is checked in \cite[Lemma 4.18]{DPS} that the composition depends only on $\sigma$. 
Finally, since 
$$Z(\Res_{F/\Qp}\mathfrak g)\otimes_{\Qp} L \cong \bigotimes_{\sigma: F \hookrightarrow L} Z(\mathfrak g)\otimes_{F, \sigma} L$$
we may define $\zeta_{\rho}:= \otimes_{\sigma} \zeta_{\rho, \sigma}$ and then use 
adjunction to obtain a 
homomorphism of $\Qp$-algebras.

\begin{lem}[\cite{DPS}]\label{conj_same} Let $g\in \Ghat(\OO_X(X))$ and 
let $\leftidx{^g}{\rho}{}:\Gamma_F \rightarrow \LG(\OO_X(X))$
be the representation $\leftidx{^g}{\rho}{}(\gamma):= g \rho(\gamma)g^{-1}$. Then $\zeta_{\rho}= \zeta_{\leftidx{^g}{\rho}{}}.$
\end{lem}
\begin{proof} By the last part of \cite[Lemma 4.10]{DPS} 
we have $\theta_{\rho}= \theta_{\grho}$. The main  
point in the proof is that the Sen operators are conjugate by $g$, 
and hence the restriction of the corresponding homomorphisms 
$S(\ghat^*)\rightarrow \Cp\wtimes A_i$
to the subring of $\Ghat$-invariant functions yields 
the same homomorphism. Since the rest of the terms 
in \eqref{one} do not depend on $\rho$ we obtain the lemma.  
\end{proof}
\subsection{An application of the  Hebbarkeitssatz}
Let $\rhobar: \Gamma_F \rightarrow \LG(k)$
be a continuous admissible representation,
let $\Thetabar$ be its  $\LG$-pseudocharacter, let $R^{\ps}$ be the
universal deformation ring of $\Thetabar$. The scheme 
$X^{\gen}:=X^{\gen}_{\Thetabar}$ is of finite type over $\Spec R^{\ps}$ and we let $A^{\gen}:=A^{\gen}_{\Thetabar}$ be the ring of its global sections. To ease the notation we let $R^{\ana}$ be the ring of global sections of the 
rigid space $(\Spf R^{\ps})^{\rig}$ and let $A^{\ana}$ 
be the ring of global sections of the analytification 
of $X^{\gen}$ in the sense of
\Cref{universal}. Let $\rho^{\ana}: \Gamma_F \rightarrow 
\LG(A^{\ana})$ be the representation obtained from the
universal representation 
$\rho^{u}: \Gamma_F \rightarrow 
\LG(A^{\gen})$
by extending scalars along $ A^{\gen}\rightarrow A^{\ana}$.

\begin{lem}\label{cont_adm} Given $(T, \OO_T)\in \Rig_L$ and a morphism of locally $G$-ringed $L$-spaces $(T, \OO_T)\rightarrow
(X^{\gen}, \OO_{X^{\gen}})$, the representation 
$\rho: \Gamma_F\overset{\rho^{\ana}}{\rightarrow} \LG(A^{\ana})\rightarrow \LG(\OO_T(T))$ is continuous and admissible. In particular, $\rho^{\ana}$ is continuous and admissible. 
\end{lem} 
\begin{proof} The universal property of analytification \Cref{universal} yields a morphism of rigid analytic spaces 
from $T$ to $(X^{\gen})^{\ana}$ and hence a continuous 
ring homomorphism $A^{\ana}\rightarrow \OO_T(T)$. 
The maps $R^{\ps}\rightarrow R^{\ana}\rightarrow A^{\ana}$ are also continuous and continuity of $\rho$ follows from 
\cite[Lemma 5.10 (2), Proposition 8.3]{defG}.

The admissibility follows from the functoriality of 
$X^{\gen}$. Namely, the quotient map $f: \LG\rightarrow \LG/\Ghat$ induces a morphism of schemes $X^{\gen}_{\Thetabar} \rightarrow X^{\gen}_{f\circ \Thetabar}$ by \cite[Proposition 8.5]{defG}. Since
the neutral component of $\LG/\Ghat$
is trivial $X^{\gen}_{f\circ \Thetabar}= X^{\ps}_{f\circ \Thetabar}$ by \Cref{lit_defT}. Since $\LG/\Ghat$ is a constant group
scheme $X^{\ps}_{f\circ \Thetabar}=\Spec \OO$.
Thus
for all $R^{\ps}$-algebras $B$ such that $\Spec B$ 
is connected and for all $\rho\in X^{\gen}_{\Thetabar}(B)$ the homomorphisms $\Gamma_F\overset{\rho}{\longrightarrow}
\LG(B)\rightarrow (\LG/\Ghat)(B)=\Gal(E/F)$ are the same. Since $\rhobar\in X^{\gen}_{\Thetabar}(k)$ is admissible we deduce that $\rho^u$ is admissible, which implies that $\rho^{\ana}$ is admissible.

The assertion concerning $\rho^{\ana}$ follows by choosing $T=(X^{\gen})^{\ana}$. 
\end{proof} 

\begin{prop}\label{work} If $F\neq \Qp$ then 
$\theta_{\rho^{\ana}}$ takes values in $E\otimes_{\Qp} R^{\ana}$.
\end{prop}
\begin{proof} Let 
$\{U_i\rightarrow
U\}_{i\in I}$ 
be the covering and let 
$\{\rho_i: \Gamma_F \rightarrow \LG(A_i)\}_{i\in I}$ be the family of representations constructed in \Cref{cover_nice_locus}. Let $A_i^{\ana}$ be the 
ring of global sections of $U_i^{\ana}$ and let $\rho_i^{\ana}:\Gamma_F \rightarrow \LG(A_i^{\ana})$ be the 
representation obtained from $\rho_i$ by extending scalars along $A_i \rightarrow A_i^{\ana}$. 
It follows from \Cref{cont_adm} that $\rho_i^{\ana}$ are admissible and continuous. 
We claim that the homomorphisms 
$\theta_{\rho_i^{\ana}}: S(\ghat^*)^{\Ghat}\rightarrow 
E\otimes_{\Qp} A_i^{\ana}$ agree on the intersections 
$U_{ij}^{\ana}$, where $U_{ij} =U_i\times_{U} U_j$.

It follows from \Cref{normal} that $
R^{\ps}[1/p]$ is reduced, which implies that $U^{\ana}$ and hence $U_{ij}^{\ana}$ are reduced by \Cref{normal_rig}.
Thus to check that $\theta_{\rho^{\ana}_i}$ and 
$\theta_{\rho^{\ana}_j}$ 
are equal after extension of scalars to $\OO(U_{ij}^{\ana})$ it is enough to check that their specialisations 
at closed points of $U_{ij}^{\ana}$ 
are equal. Let $x\in U_{ij}^{\ana}(\Qpbar)$ 
and let $y$ be its image in $U^{\ana}$. Then
$\rho_{i,x}^{\ana}$ and $\rho_{j, x}^{\ana}$ have the same 
$\LG$-pseudocharacter equal to 
the specialisation of $\Theta^u$ at $y$. Since 
$y$ by construction lies in the irreducible locus  we deduce that $\rho_{i,x}^{\ana}$ and $\rho_{j,x}^{\ana}$ are
$\Ghat(\Qpbar)$-conjugate. It follows 
from \Cref{conj_same} that $\theta_{\rho_{i, x}^{\ana}}=
\theta_{\rho_{j,x}^{\ana}}$. Since $\theta_{\rho_{i, x}^{\ana}}$
is the specialisation of $\theta_{\rho_i^{\ana}}$ at $x$ 
we obtain the claim. 

The claim  and \Cref{etale_cover}   imply that $\theta_{\rho_i^{\ana}}$ glue 
to a homomorphism of $\Qp$-algebras 
$\theta':S(\ghat^*)^{\Ghat}\rightarrow E\otimes_{\Qp} \OO(U^{\ana})$. \Cref{conj_same} implies that 
for every $y\in U^{\ana}(\Qpbar)$ the specialisation 
$\theta'_y$ is equal to $\theta_{\rho}$, where
$\rho: \Gamma_F \rightarrow \LG(\Qpbar)$ is any 
continuous admissible representation
with  pseudocharacter equal to $\Theta^u_y$. 

Let $V=\pi^{-1}(U)$, let $Z$ be the complement of $V$ in $X^{\gen}[1/p]$. \Cref{complement} implies that the analytification $V^{\ana}$ is equal to  the complement of $Z^{\ana}$ in $(X^{\gen})^{\ana}$. 
\Cref{deJong_like} implies that $V(\Qpbar)=V^{\ana}(\Qpbar)$. Thus,
if $x\in V^{\ana}(\Qpbar)$ and $y$ is the image of $x$ in $(\Spf R^{\ps})^{\rig}(\Qpbar)$ then the 
$\LG$-pseudocharacter of $\rho^{\ana}_x$ is equal to $\Theta^u_y$. Hence, we obtain 
\begin{equation}\label{special}
\theta_{\rho^{\ana}_x} = \theta'_y, \quad \forall x\in V^{\ana}(\Qpbar).
\end{equation}

Since $(\Spf R^{\ps})^{\rig}$ is normal, and the complement of $U^{\ana}$  has codimension  at 
least $2$, $\OO(U^{\ana})=R^{\ana}$ by the 
Riemann's Hebbarkeitssatz, \cite[Satz 10]{bartenwerfer}, 
\cite[Proposition 2.13]{lutkebohmert}. Hence, $\theta'$ takes 
values in $E\otimes_{\Qp} R^{\ana}$.

Let $\theta'_{A^{\ana}}$ be the homomorphism obtained by composing 
$\theta'$ with the map 
$E\otimes_{\Qp} R^{\ana}\rightarrow 
E\otimes_{\Qp} A^{\ana}$. We claim 
that $\theta'_{A^{\ana}}= \theta_{\rho^{\ana}}$. 
The claim implies the proposition. 
To prove the claim we apply the Hebbarkeitssatz again. 
It follows from \Cref{dim_rig} and \Cref{nice_locus} that 
\begin{equation}
\dim (X^{\gen})^{\ana} -\dim Z^{\ana} = \dim X^{\gen}[1/p] - \dim Z \ge 2.
\end{equation} 
Since $X^{\gen}[1/p]$ is normal by \cite[Corollary 15.23]{defG}, $(X^{\gen})^{\ana}$ is also normal by \Cref{normal_rig}. Thus 
$\OO(V^{\ana})= A^{\ana}$. Hence, two functions in $A^{\ana}$ are equal if and only if their specialisations
at all $x\in V^{\ana}(\Qpbar)$ are equal. The specialisation of $\theta_{\rho^{\ana}}$ at $x\in V^{\ana}(\Qpbar)$ is equal to $\theta_{\rho^{\ana}_x}$ and hence the claim follows from \eqref{special}.
\end{proof}

\begin{thm}\label{main} 
The infinitesimal character
$\zeta_{\rho^{\ana}}: Z(\Res_{F/\Qp}\mathfrak g) \rightarrow A^{\ana}$ takes values in $R^{\ana}$ and 
$\theta_{\rho^{\ana}}: S(\ghat^*)^{\Ghat}\rightarrow E\otimes A^{\ana}$ takes values in $E\otimes R^{\ana}$.
\end{thm}
\begin{proof} If $F\neq \Qp$ then the assertion 
follows directly from \Cref{work} and \eqref{one}. We will prove the theorem by reducing to this case. 

As before let $E$ be an extension of $F$ such that 
$G$ splits over $E$. We assume that $E\neq \Qp$, 
even if $G$ splits over $\Qp$. 
Let $\Thetabar_E\in \cPC^{\Gamma_E}_{\Ghat}(k)$ be the 
restriction of $\Thetabar$ to $\Gamma_E$. 
Let $X^{\gen}_E:= X^{\gen, \Gamma_E}_{\Thetabar_E}$, let $A^{\gen}_E$ be the ring of its global sections, $R^{\ps}_E$ the deformation ring 
of $\Thetabar_E$. 
Let $A^{\ana}_E$ be the ring of global sections 
of the analytification of $X^{\gen}_E$ and 
let $R^{\ana}_E$ be the ring of global sections
of $(\Spf R^{\ps}_E)^{\rig}$. Similarly, we add the subscript $F$, for example
 $X^{\gen}_F$, $A^{\ana}_F$ and $R^{\ana}_F$, 
 to indicate that we are working with $\LG$-valued $\Gamma_F$-representations 
 and pseudocharacters. 
 
 Restricting pseudocharacters and representations to $\Gamma_E$ induces a finite homomorphism of 
 local $R^{\ps}_E$-algebras $R^{\ps}_E\rightarrow R^{\ps}_F$ and a homomorphism of $R^{\ps}_E$-algebras
 $A^{\gen}_E\rightarrow A^{\gen}_F$  such that the map
 $R^{\ps}_E\rightarrow A^{\gen}_F$ factors 
 as $R^{\ps}_E\rightarrow R^{\ps}_F\rightarrow A^{\gen}_F$ by \cite[Section 8.4]{defG}. Since relative analytification is functorial by \Cref{rel_an_fun},
  the map induces
 a continuous homomorphism of $R^{\ana}_E$-algebras
 $A^{\ana}_E \rightarrow A^{\ana}_F$, such that 
 the image of $R^{\ana}_E$ is contained in $R^{\ana}_F$. Since the construction of $\theta$ is  functorial we obtain that 
 \begin{equation}\label{restiction} 
 \theta_{\rho^{\ana}_F |_{\Gamma_E}} = \theta_{\rho^{\ana}_E}\otimes_{A^{\ana}_E} A^{\ana}_F.
 \end{equation} 
 \Cref{work} implies that $\theta_{\rho^{\ana}_F |_{\Gamma_E}}$ takes values in the image 
 of $E\otimes_{\Qp} R^{\ana}_E$ in $E\otimes_{\Qp} A^{\ana}_F$, which is contained in $E\otimes_{\Qp}R^{\ana}_F$. 
 
 It follows from the construction of the Sen operator 
  that 
  $\Theta_{\Sen,\rho_F^{\ana}}=\Theta_{\Sen,\rho_F^{\ana}|_{\Gamma_E}}$, see the end of the proof of \cite[Theorem 3.1]{DPS} and 
  the Tannakian argument in the proof of 
  \cite[Lemma 4.6]{DPS}. Thus 
  \begin{equation}\label{restiction2} 
 \theta_{\rho^{\ana}_F |_{\Gamma_E}} = \theta_{\rho^{\ana}_F}
 \end{equation} 
 and the theorem follows from \eqref{one}. \end{proof}
 
 \begin{cor}\label{special_again} Let $\zeta_{\rho^{\ana},y}: Z(\Res_{F/\Qp} \mathfrak g)\rightarrow \Qpbar$  be the specialisation of $\zeta_{\rho^{\ana}}$ at $y\in (\Spf R^{\ps})^{\rig}(\Qpbar)$. Then $\zeta_{\rho^{\ana},y} =\zeta_{\rho}$ and $\theta_{\rho^{\ana},y}=\theta_{\rho}$, where 
 $\rho: \Gamma_F \rightarrow \LG(\Qpbar)$ is any continuous 
 admissible representation with $\Theta_{\rho}=\Theta^u_y$. 
 \end{cor}
 \begin{proof} The representation $\rho$ defines 
 a point $x\in X^{\gen}(\Qpbar)$ and hence a $\Qpbar$-point in the relative analytification of $A^{\gen}$ by \Cref{deJong_like}. Using the functoriality of the infinitesimal  character we have that $\zeta_{\rho}$ 
 is equal to the composition 
 $$Z(\Res_{F/\Qp}\mathfrak g)\overset{\zeta_{\rho^{\ana}}}{\longrightarrow} A^{\ana} \overset{x}{\longrightarrow} \Qpbar.$$
  The assertion follows from \Cref{main}. The argument 
  for $\theta$ is the same using functoriality of the Sen operator. 
 \end{proof}

 \begin{lem}\label{descent} Let $K$ be a $\Qp$-subalgebra of $L$ and let  $\OO_K$ be its ring of integers. 
 Let $S$ be a complete local 
 noetherian $\OO_K$-subalgebra of $R^{\ps}$.
  If the following hold 
 \begin{enumerate}
 \item the natural map $\OO\otimes_{\OO_K} S \rightarrow R^{\ps}$ is an isomorphism;
 \item $\Theta^u$ takes values in $S$,
 \end{enumerate}
 then $\theta_{\rho^{\ana}}$ takes values in $E\otimes S^{\ana}$ and $\zeta_{\rho^{\ana}}$ takes values 
 in $S^{\ana}$ - the ring of global 
 sections of $(\Spf S)^{\rig}$. 
 \end{lem}
 \begin{proof}  After enlarging $L$ we may assume that $L/K$ is a Galois extension. It follows 
 from (1) that $R^{\ana}\cong L\otimes_K S^{\ana}$. The action of $\Gal(L/K)$ on $L$ 
 induces an action on $R^{\ana}$ by ring automorphisms $g\mapsto \phi_g$. Let 
 $\zeta_{\rho^\ana}^g$ be the composition of 
 $\zeta_{\rho^{\ana}}$ with $\phi_g$. 
 We claim that $\zeta_{\rho^{\ana}}^g= \zeta_{\rho^{\ana}}$ for all $g\in \Gal(L/K)$. 
 The claim implies that $\zeta_{\rho^{\ana}}$ 
 takes values in $(R^{\ana})^{\Gal(L/K)}=S^{\ana}$. 
 
 Since $R^{\ana}$ is reduced it is enough to verify 
 that the specialisations of $\zeta_{\rho^{\ana}}$
 and $\zeta^g_{\rho^{\ana}}$ at $y\in (\Spf R^{\ps})^{\rig}(\Qpbar)$ 
 are the same. Now 
 $\zeta^g_{\rho^{\ana},y}= \zeta_{\rho^{\ana}, y\circ \phi_g}$.
 It follows from (2) that $\Theta^u_y= \Theta^u_{y\circ \phi_g}$ and hence \Cref{special_again} implies that 
 $\zeta_{\rho^{\ana}, y\circ \phi_g}=\zeta_{\rho^{\ana}, y}$. 
 
 The same argument 
 shows that $\theta_{\rho^{\ana}}$ takes values in 
 $\Gal(L/K)$-invariants of 
 $(E\otimes_{\Qp} L\otimes_K S^{\ana})$, which are equal to  $E\otimes_{\Qp} S^{\ana}$.
 \end{proof}

\subsection{Rigid analytic space of admissible pseudocharacters}\label{adm_pseudo}
The morphism $f:\LG \rightarrow \LG/\Ghat$ induces a morphism 
$\PC^{\Gamma_F}_{\LG} \rightarrow \PC^{\Gamma_F}_{\LG/\Ghat}$, $\Theta\mapsto f\circ \Theta$.
Since $\LG/\Ghat=\underline{\Gal(E/F)}$ is a constant group scheme, its neutral component is trivial, 
and \Cref{lit_defT} implies 
that the 
morphism 
$\Rep^{\Gamma_F}_{\LG/\Ghat}\rightarrow  \PC^{\Gamma_F}_{\LG/\Ghat}$, $\rho \mapsto \Theta_{\rho}$  is an isomorphism. We say that $\Theta\in \PC^{\Gamma_F}_{\LG}(A)$ is \emph{admissible} if 
$f\circ \Theta$ corresponds to 
a homomorphism $\rho: \Gamma_F \rightarrow \underline{\Gal(E/F)}(A)$, which maps $g\in \Gamma_F$ to the constant function $\Spec A \rightarrow \Gal(E/F)$ with value $q(g)$, where $q: \Gamma_F\rightarrow \Gal(E/F)$ is the quotient map. If $\rho\in \Rep^{\Gamma_F}_{\LG}(A)$ then $\rho$ is admissible 
if and only if $\Theta_{\rho}$ is admissible.

We will show that the functor 
$\tilde{X}^{\adm}_{\LG}: \Rig_{\Qp}^{\op}\rightarrow \Set$, which 
maps a rigid analytic space $Y$ to the set of continuous admissible $\LG$-pseudocharacters $\Theta: \Gamma_F\rightarrow \LG(\OO(Y))$ is representable by a quasi-Stein rigid 
analytic space $X^{\adm}_{\LG}$ using the results of \cite{quast}. We will define a homomorphism 
of $\Qp$-algebras $\zeta^{\univ}: Z(\Res_{F/\Qp} \mathfrak g)\rightarrow \OO(X^{\adm}_{\LG})$ using the results 
of the previous section and the description of $X^{\adm}_{\LG}$ obtained from \cite{quast}. If $Y\in \Rig_{\Qp}$ then an admissible continuous $\LG$-pseudocharacter 
$\Theta: \Gamma_F \rightarrow \OO(Y)$ defines a morphism 
of rigid spaces $Y\rightarrow X^{\adm}_{\LG}$ and 
we may then define $\zeta_{Y,\Theta}$ as the composition 
of $\zeta^{\univ}$ and $\OO(X^{\adm}_{\LG})\rightarrow \OO(Y)$. 

We will first recall the results of \cite[Section 6]{quast}, 
which generalise the results of Chenevier \cite{che_durham} in the case of $\GL_n$. Let $H$ be a generalised reductive group 
scheme defined over $\Zp$. The functor $\tilde{X}_H: \Rig_{\Qp}^{\op}\rightarrow \Set$, which sends $Y$ to the 
set of continuous $H$-pseudocharacters $\Theta: \Gamma_F \rightarrow \OO(Y)$ is represented by a quasi-Stein rigid 
analytic space 
\begin{equation}\label{XH}
X_H= \coprod_{z\in |\PC^{\Gamma_F}_H|} (\Spf R^{\ps}_{\Theta_z})^{\rig},
\end{equation}
by \cite[Theorem 6.19]{quast}, where 
$|\PC^{\Gamma_F}_H|\subset \PC^{\Gamma_F}_H$ is the subset of closed points $z$ with finite residue field $k_z$, such that the canonical $H$-pseudocharacter $\Theta_z\in \PC^{\Gamma_F}_H(k_z)$ attached to $z$ is continuous for the discrete topology on $k_z$ and $R^{\ps}_{\Theta_z}$ 
represents the deformation problem $\Def_{\Theta_z}: \mathfrak A_{W(k_z)}\rightarrow \Set$ of $\Theta_z$, 
where $W(k_z)$ is the ring of Witt vectors of $k_z$. 
This description of $R^{\ps}_{\Theta_z}$ is implicit in the proof of \cite[Theorem 6.19]{quast}.

We apply this with $H=\LG$. 
\begin{lem} The functor 
$\tilde{X}^{\adm}_{\LG}$ is represented by  a quasi-Stein rigid analytic space
\begin{equation}\label{admLG}
X^{\adm}_{\LG}= \coprod_{z\in |\PC^{\Gamma_F,\adm}_{\LG}|} (\Spf R^{\ps}_{\Theta_z})^{\rig},
\end{equation}
where $|\PC^{\Gamma_F,\adm}_{\LG}|=\{z\in |\PC^{\Gamma_F}_{\LG}|:  \Theta_z \text{ is  admissible}\}$. \end{lem}
\begin{proof} The quotient map $q:\Gamma_F \rightarrow \Gal(E/F)$ 
gives us a point in $X_{\LG/\Ghat}(\Qp)$. It follows 
from the definition of admissible pseudocharacters  that $\tilde{X}^{\adm}_{\LG}$ is represented by $X_{\LG}\times_{X_{\LG/\Ghat}} \Sp \Qp$, where the morphism 
$X_{\LG}\rightarrow X_{\LG/\Ghat}$ is induced by $f$ and $\Sp \Qp \rightarrow X_{\LG/\Ghat}$ is induced by $q$. The assertion follows from \eqref{XH}.
\end{proof}

\begin{thm}\label{main_L} There exists a unique  family of $\Qp$-algebra 
homomorphisms $$\zeta_{Y,\Theta}: Z(\Res_{F/\Qp} \mathfrak g)\rightarrow \OO(Y)$$ indexed by pairs $Y\in \Rig_{\Qp}$ and 
$\Theta\in X^{\adm}_{\LG}(Y)$ with the following properties:
\begin{enumerate} 
\item (functoriality) for every morphism $\varphi: Y'\rightarrow Y$ in $\Rig_{\Qp}$ we have $$\varphi^*(\zeta_{Y, \Theta})= \zeta_{Y', \varphi^* \Theta};$$
\item (specialisation) for every $Y\in \Rig_{\Qp}$ and every 
$y\in Y(\Qpbar)$ the specialisation 
$$\zeta_{y, \Theta}:  Z(\Res_{F/\Qp} \mathfrak g)
\overset{\zeta_{Y,\Theta}}{\longrightarrow} \OO(Y)\overset{y}{\longrightarrow} \Qpbar$$
is equal to $\zeta_{\rho}$, where $\rho: \Gamma_F \rightarrow 
\LG(\Qpbar)$ is any continuous representation with 
$\Theta_{\rho}= \Theta\otimes_{\OO(Y), y} \Qpbar$. 
\end{enumerate}
\end{thm}

\begin{proof} We first fix $z\in |\PC^{\Gamma_F,\adm}_{\LG}|$
and let $\Theta^{u,z}: \Gamma_F \rightarrow R^{\ps}_{\Theta_z}$ 
be the universal deformation of $\Theta_z$. 
Let 
$L$ be any extension of $\Qp$ with the ring of integers $\OO$ and residue field $k$, which contains $k_{z}$ as a subfield. 
Then $\OO$ is naturally a $W(k_z)$-algebra and if we let $\Thetabar= \Theta_z\otimes_{k_z} k$ then the deformation functor $\Def_{\Thetabar}: \mathfrak A_{\OO}\rightarrow \Set$
of $\Thetabar$ is represented by $\OO\otimes_{W(k_z)} R^{\ps}_{\Theta_z}$ and $\OO\otimes_{W(k_z)} \Theta^{u,z}$
is the universal deformation of $\Thetabar$.
It follows from the reconstruction theorem \cite[Theorem 3.8]{quast} that if $k$ is large enough 
then there exists a continuous representation 
$\rhobar: \Gamma_F\rightarrow \LG(k)$ such that its 
pseudocharacter is equal to $\Thetabar$. In this case we are 
in the situation of the previous subsection with $R^{\ps}=\OO\otimes_{W(k_z)} R^{\ps}_{\Theta_z}$. 
We let $\zeta^z:= \zeta_{\rho^{\ana}}$. \Cref{descent} implies 
that $\zeta^z$ takes values in $(R^{\ps}_{\Theta_z})^{\ana}$. 
It follows from \Cref{special_again} that the specialisation property holds for $\zeta^z$. We define
$$\zeta^{\univ}: Z(\Res_{F/\Qp} \mathfrak g)\rightarrow \OO(X^{\adm}_{\LG})=\prod_{z\in |\PC^{\Gamma_F,\adm}_{\LG}|} (R^{\ps}_{\Theta_z})^{\ana}$$
as the product $\zeta^{\univ}:=\prod_z \zeta^z$. Since the specialisation 
property holds for $\zeta^z$ it also holds for $\zeta^{\univ}$. If $Y\in \Rig_{\Qp}$ and $\Theta\in X^{\adm}_{\LG}(Y)$ then $\Theta$ defines a morphism 
$\varphi_{\Theta}: Y\rightarrow X^{\adm}_{\LG}$ in $\Rig_{\Qp}$ and 
we define $$\zeta_{Y, \Theta}:= \varphi_{\Theta}^* (\zeta^{\univ}).$$ By construction $\zeta_{Y, \Theta}$
satisfies both functoriality and specialisation properties. 
Moreover, if $Y= X^{\adm}_{\LG}$ and $\Theta=\Theta^{\univ}$ the universal pseudocharacter over $X^{\adm}_{\LG}$ then 
$\zeta_{Y, \Theta}=\zeta^{\univ}$.

Let us prove the uniqueness of the family. If $Y$ is reduced
then $\OO(Y)$ injects into the product of the residue fields
at $y\in Y(\Qpbar)$ and hence $\zeta_{Y, \Theta}$ is 
uniquely determined by the specialisation property. Thus $\zeta_{X^{\adm}_{\LG}, \Theta^{\univ}}$ is
uniquely determined. If $Y\in \Rig_{\Qp}$ and 
$\Theta\in X^{\adm}_{\LG}(Y)$ then the functoriality property implies that $\zeta_{Y, \Theta}=\varphi_{\Theta}^*(\zeta_{X^{\adm}_{\LG}, \Theta^{\univ}})$ and hence the whole family is uniquely determined. 
\end{proof}

\begin{prop}\label{compatibility_L}
Let $\rho: \Gamma_F \rightarrow \LG(\OO(Y))$ 
be a continuous admissible representation, where $Y\in \Rig_{\Qp}$ and let $\Theta_{\rho} \in X^{\adm}_{\LG}(Y)$ be  its $\LG$-pseudocharacter. Then $\zeta_{\rho}$ associated to $\rho$ in \cite{DPS} coincides with $\zeta_{Y, \Theta_{\rho}}$.
\end{prop}
\begin{proof} If $Y$ is a point then the 
assertion follows from the specialisation property in \Cref{main_L}. This immediately implies the assertion for reduced $Y$. 

In general, it is enough to assume $Y=\Sp A$ is an affinoid by
passing to a covering. One may further assume that $A$ is 
a local artinian $\Qp$-algebra with residue field a finite extension of $\Qp$ as one can embed $A$ into a product of its completions at well 
chosen closed points (see the proof of \cite[Proposition 8.1]{MRL}). 

Assume the setting above and let $\mm_A$ be the maximal ideal of $A$, $\kappa$ its residue field, $\OO_{\kappa}$ the ring of integers of $\kappa$,
$k':=\OO_{\kappa}/\varpi_{\kappa}$, where $\varpi_{\kappa}$ is a uniformiser.  Let $\rho':\Gamma_F \rightarrow \LG(\kappa)$ be the 
reduction of $\rho$ modulo $\mm_A$. Let 
$\Theta_{\rho}\in \cPC^{\Gamma_F}_{\LG}(A)$ and
$\Theta_{\rho'}\in \cPC^{\Gamma_F}_{\LG}(\kappa)$
be the $\LG$-pseudocharacters of $\rho$ and $\rho'$,
respectively. Then $\Theta_{\rho}$ is a deformation 
of $\Theta_{\rho'}$ to $A$. It is shown in \cite[Lemma 7.1]{MRL} that $\Theta_{\rho'}$ takes values in $\OO_{\kappa}$. Let 
$\overline{\Theta}
\in \cPC^{\Gamma_F}_{\LG}(k')$ be its
reduction modulo $\varpi_{\kappa}$.
Since $\Theta_{\rho}$ is a deformation of $\Theta_{\rho'}$ to $A$, \cite[Proposition 7.2]{MRL}
implies that there is an $\OO_{\kappa}$-algebra homomorphism $\varphi: R^{\ps}_{\Thetabar} \rightarrow A$ such that $\Theta_{\rho}= \Theta^u \otimes_{R^{\ps}_{\Thetabar}, \varphi} A$, where $\Theta^u\in \cPC^{\Gamma_F}_{\LG}(R^{\ps}_{\Thetabar})$ is the universal deformation of $\Thetabar$. Since $A$ is an 
artinian $\kappa$-algebra with residue field $\kappa$ 
we deduce that there is an open subalgebra $A^0$ of $A$, which is a finitely generated $R^{\ps}_{\Thetabar}$-module. It follows from 
\cite[Lemma 4.14]{defG} that $\rho$ is $R^{\ps}_{\Thetabar}$-condensed. 
Hence, we obtain  a morphism $\Sp(A)\rightarrow (X^{\gen})^{\ana}$ such that $\rho$ is the specialization of the universal representation $\rho^{\ana}$ along this morphism. Thus  $\zeta_{\rho}$ is the specialisation of $\zeta_{\rho^{\ana}}$, but this is by construction $\zeta_{Y, \Theta_{\rho}}$.
\end{proof}

\subsection{\texorpdfstring{Modification for $C$-groups}{Modification for C-groups}}\label{sec_mod_C}
A conjecture of Buzzard--Gee \cite[Conjecture 5.3.4]{BG} predicts that Galois representations
attached to $C$-algebraic automorphic forms should take values in a $C$-group, which is a variant of an $L$-group. Motivated by this 
conjecture we will prove an analog of \Cref{main_L} for pseudocharacters 
valued in $C$-groups. To paraphrase \cite{BG}: the two theorems
``differ by half the sum of positive roots''.

Let $G$ be a connected reductive group over 
a $p$-adic local field $F$, which splits over 
a finite Galois extension $E$ of $F$. The $C$-group of $G$ is denoted by $\CG$ and is defined as the $L$-group $\LG^T$ of a certain central extension 
\begin{equation} 
0 \rightarrow \Gm \rightarrow G^T \rightarrow G \rightarrow 0
\end{equation}
defined in \cite[Proposition 5.3.1]{BG}. In 
\cite[Section 1]{zhu} Xinwen Zhu provides a very nice direct description of $\CG$ and shows that 
\begin{equation}\label{CG}
\CG\cong \Ghat \rtimes ( \Gm \times \underline{\Gal(E/F)}) 
\cong  \Ghat^T \rtimes \underline{\Gal(E/F)}\cong \LG \rtimes \Gm. 
\end{equation}
Let $d:\CG \rightarrow \Gm$ be the projection induced by the last isomorphism in \eqref{CG}. If $A$ is a $\Zp$-algebra then we say that a representation $\rho: \Gamma_F \rightarrow \CG(A)$ is \textit{$C$-admissible} if it is admissible as 
a representation into the $L$-group $\LG^T(A)$ in the sense of the previous subsection and 
$d\circ \rho$ is equal to the $p$-adic cyclotomic character $\chi_{\cyc}$.

We will now recall a construction of \cite[Section 4.7]{DPS}, which to a continuous  $C$-admissible representation $\rho: \Gamma_F \rightarrow \CG(\OO(X))$ attaches 
a homomorphism of $\Qp$-algebras 
$\zeta^C_\rho: Z(\Res_{F/\Qp}\mathfrak g)\rightarrow \OO(X)$, where $X\in \Rig_L$. By our previous 
construction, viewing $\CG$ as 
an $L$-group $\LG^T$, we have 
already constructed a homomorphism of $\Qp$-algebras $\theta_{\rho}: S((\ghat^T)^*)^{\Ghat^T} \rightarrow E\otimes \OO(X)$. In \cite[Section 4.7]{DPS}, using \eqref{CG}, the authors construct a homomorphism of  $\Qp$-algebras 
$\alpha: S(\ghat^*)^{\Ghat}\rightarrow S((\ghat^T)^*)^{\Ghat^T}$. This is where the shift by half sum of positive roots comes in. Then 
$\theta'_{\rho}: S(\ghat^*)^{\Ghat}\rightarrow E\otimes\OO(X)$ is defined as $\theta_{\rho}\circ \alpha$.
The homomorphism $\zeta^C_{\rho}$ is defined by \eqref{one} using $\theta'_{\rho}$ instead of $\theta_{\rho}$.

We will define \textit{$C$-admissible} pseudocharacters 
valued in $\CG$. Since $\Gm$ is commutative 
the map $\Rep^{\Gamma}_{\Gm}(A)\rightarrow 
\PC^{\Gamma}_{\Gm}(A)$ is bijective for any group $\Gamma$. The morphism $d: \CG\rightarrow \Gm$ induces a map $\PC^{\Gamma}_{\CG}(A) \rightarrow \PC^{\Gamma}_{\Gm}(A)$, $\Theta\mapsto d\circ \Theta$.
Let $A$ be a $\Zp$-algebra. We say that $\Theta\in \PC^{\Gamma_F}_{\CG{}}(A)$ is $C$-admissible, if $\Theta$ is admissible as 
the $\LG^T$-valued pseudocharacter of $\Gamma_F$ and $d\circ \Theta=\chi_{\cyc}\otimes_{\Zp} A $. It follows from the definition if $\rho\in \Rep^{\Gamma_F}_{\LG^T}(A)$ then $\rho$ is $C$-admissible if and only if 
$\Theta_{\rho}$ is $C$-admissible. 

Let $\omega: \Gamma_F \rightarrow
\Fp^{\times}$ be the reduction of $\chi_{\cyc}$ modulo $p$. Let $R_{\omega}$ be the $\Zp$-algebra 
representing the deformation functor $D_{\omega}: \mathfrak A_{\Zp} \rightarrow \Set$. Let $\varphi_{\cyc}: R_{\omega}\rightarrow \Zp$ be the 
homomorphism of $\Zp$-algebras corresponding 
to $\chi_{\cyc}\in D_{\omega}(\Zp)$. 

Let $k$ be a finite extension of $\Fp$ and let 
$\Thetabar\in \cPC^{\Gamma_F}_{\CG{}}(k)$ such that 
$\Thetabar$ is $C$-admissible. Let $\Def^{C}_{\Thetabar}: \mathfrak A_{W(k)} \rightarrow \Set$ be the subfunctor of $\Def_{\Thetabar}$ parameterising $C$-admissible deformations of $\Thetabar$. The map $\Theta\mapsto d\circ \Theta$ induces a homomorphism 
of local $\Zp$-algebras $R_{\omega}\rightarrow R^{\ps}_{\Thetabar}$. The functor $\Def^{C}_{\Thetabar}$ is represented by 
$R^{C}_{\Thetabar}= R^{\ps}_{\Thetabar}\otimes_{R_{\omega}, \varphi_{\cyc}} \Zp$.

Let $\tilde{X}^C_{\CG{}}: \Rig_{\Qp}^{\op} \rightarrow \Set$ be the functor, which sends
$Y\in \Rig_{\Qp}$ to the set of continuous $C$-admissible pseudocharacters $\Theta: \Gamma_F \rightarrow \CG(\OO(Y))$.

\begin{lem}\label{quasi_Stein_C} The functor $\tilde{X}^{C}_{\CG{}}$ is represented by a quasi-Stein rigid analytic space 
$$ X^{C}_{\CG{}}= \coprod_{z\in |\PC^{\Gamma_F,C}_{\CG{}}|} (\Spf R^{C}_{\Theta_z})^{\rig},$$
where
$|\PC^{\Gamma_F,C}_{\CG{}}|:=\{ z\in|\PC^{\Gamma_F,\adm}_{\LG^T}|: d\circ\Theta_z=\chi_{\cyc}\otimes_{\Zp} k_z\}$. 
\end{lem}
\begin{proof} It follows from the definition of $C$-admissible pseudocharacters that $ \tilde{X}^{C}_{\CG}$ is represented by 
$X^{\adm}_{\LG^T}\times_{X_{\Gm}} \Sp \Qp$, where the morphism $X^{\adm}_{\LG^T}\rightarrow X_{\Gm}$ is induced by $d$ and $\Sp \Qp\rightarrow X_{\Gm}$ by $\chi_{\cyc}\in X_{\Gm}(\Qp)$. It follows from \eqref{admLG} that 
$$ X^C_{\CG}\cong \coprod_{z\in |\PC^{\Gamma_F,C}_{\CG{}}|} (\Spf R^{\ps}_{\Theta_z})^{\rig} \times_{(\Spf R_{\omega})^{\rig}} (\Spf \Zp)^{\rig}.$$ 
Since $R^C_{\Theta_z}= R^{\ps}_{\Theta_z}\otimes_{R_{\omega}} \Zp$, the assertion follows. 
\end{proof}

Let $\rhobar: \Gamma_F \rightarrow \CG(k)$ be continuous and $C$-admissible. Let $\Thetabar$ be its $\CG$-pseudocharacter and let $X^{C}_{\Thetabar}:=\Spec R^C_{\Thetabar}$. Let $X^{\gen, C}_{\Thetabar}:= 
X^{\gen}_{\Thetabar}\times_{X^{\ps}_{\Thetabar}} X^{C}_{\Thetabar}$, let $A^C$ be its ring of global functions  and let $\rho^C: \Gamma_F \rightarrow \CG(A^C)$ be the universal representation over $X^{\gen, C}_{\Thetabar}$. Then $X^{\gen, C}_{\Thetabar} = X^{\gen}_{\Thetabar}\times_{X^{\ps}_{\omega}} \Spec \Zp$, where the map $\Spec \Zp\rightarrow X^{\ps}_{\omega}$ corresponds to $\chi_{\cyc}\in X^{\ps}_{\omega}(\Zp)$. Thus $\rho^C$ is $C$-admissible. Let $A^L$ be the ring of global sections of $X^{\gen}_{\Thetabar}$ and let $\rho^L: \Gamma_F \rightarrow \LG^T(A^L)$ be the universal representation.

\begin{lem}\label{normal_C} We have $(A^C[1/p])^{\Ghat^T}\cong R^C_{\Thetabar}[1/p]$. In particular, $R^C_{\Thetabar}[1/p]$ is normal. 
\end{lem}
\begin{proof} Since $A^C\cong A^L\otimes_{R_{\omega}}\Zp$, \Cref{tensor_inv2} and \Cref{adequate} imply that 
$$(A^C[1/p])^{\Ghat^T}\cong (A^L[1/p])^{\Ghat^T}\otimes_{R_{\omega}[1/p]} \Qp\cong 
R^{\ps}_{\Thetabar}[1/p]\otimes_{R_{\omega}[1/p]} \Qp
\cong R^C_{\Thetabar}[1/p].$$
The normality of $A^C[1/p]$ follows from 
\cite[Theorem 15.19]{defG},
as explained in the proof of \cite[Theorem 16.5]{defG}.
Since taking invariants preserves normality, we deduce that $R^{C}_{\Thetabar}[1/p]$ is normal. 
\end{proof}

Let $(A^C)^{\ana}$ be the ring of global functions of the analytification of $X^{\gen, C}_{\Thetabar}$, let $\rho^{C,\ana}= \rho^C\otimes_{A^C} (A^C)^{\ana}$ and let $(R^C_{\Thetabar})^{\ana}$ be the ring of global sections of $(\Spf R^C_{\Thetabar})^{\rig}$. 
Let  $(A^L)^{\ana}$ be the ring of global sections of the analytification of $X^{\gen}_{\Thetabar}$, and let
$\rho^{L,\ana}=\rho^L\otimes_{A^L} (A^L)^{\ana}$.

\begin{thm}\label{main_C} There exists a unique  family of $\Qp$-algebra 
homomorphisms $$\zeta_{Y,\Theta}^C: Z(\Res_{F/\Qp} \mathfrak g)\rightarrow \OO(Y)$$ indexed by pairs $Y\in \Rig_{\Qp}$ and 
$\Theta\in X^{C}_{\CG}(Y)$ with the following properties:
\begin{enumerate} 
\item (functoriality) for every morphism $\varphi: Y'\rightarrow Y$ in $\Rig_{\Qp}$ we have $$\varphi^*(\zeta^C_{Y, \Theta})= \zeta^C_{Y', \varphi^* \Theta};$$
\item (specialisation) for every $Y\in \Rig_{\Qp}$ and every 
$y\in Y(\Qpbar)$ the specialisation 
$$\zeta^C_{y, \Theta}:  Z(\Res_{F/\Qp} \mathfrak g)
\overset{\zeta^C_{Y,\Theta}}{\longrightarrow} \OO(Y)\overset{y}{\longrightarrow} \Qpbar$$
is equal to $\zeta^C_{\rho}$, where $\rho: \Gamma_F \rightarrow 
\CG(\Qpbar)$ is any continuous representation with 
$\Theta_{\rho}= \Theta\otimes_{\OO(Y), y} \Qpbar$. 
\end{enumerate}
\end{thm}
\begin{proof} The proof is analogous to the proof 
of \Cref{main_L}. It is enough to define $\zeta^C_{Y, \Theta}$, when $Y=X^C_{\CG}$ and $\Theta$ is the universal pseudocharacter over it. 
\Cref{quasi_Stein_C} implies that $\OO(X^{C}_{\CG})$ 
is a product of the rings $(R^C_{\Theta_z})^{\ana}$. For each 
$z\in |\PC^{\Gamma_F,C}_{\CG}|$ there 
is a finite extension $k$ of $k_z$ and
a continuous representation 
$\rhobar: \Gamma_F \rightarrow \CG(k)$ such that $\Theta_{\rhobar}= \Theta_z\otimes_{k_z} k$. Since 
$\Theta_z$ is $C$-admissible, so is the representation $\rhobar$. The homomorphism 
$\theta_{\rho^{C,\ana}}$ is equal to the composition 
of $\theta_{\rho^{L,\ana}}$ with the quotient map 
$E\otimes (A^L)^{\ana}\rightarrow E\otimes (A^C)^{\ana}$. 
Since $A^C= A^L\otimes_{R_{\omega}, \varphi_{\cyc}} \Zp$, the map 
$R^{\ps}_{\Theta_z}\rightarrow A^C$ factors through $R^{C}_{\Theta_z}\rightarrow A^C$.
Since $\theta_{\rho^{L,\ana}}$ takes values in $E\otimes (R^{\ps}_{\Theta_z})^{\ana}$ 
by \Cref{descent}, we deduce that 
$\theta_{\rho^{C,\ana}}$ takes values 
in $E\otimes (R^C_{\Theta_z})^{\ana}$. 
Thus $\theta'_{\rho^{C, \ana}}$ takes 
values in $E\otimes(R^C_{\Theta_z})^{\ana}$ and hence $\zeta^{C,z}:=\zeta^C_{\rho^{C, \ana}}$ takes values in $(R^C_{\Theta_z})^{\ana}$ and we may 
define $\zeta^{C, \univ}:= \prod_{z} \zeta^{C,z}$, where the product is taken over $z\in |\PC^{\Gamma_F,C}_{\CG}|$. \Cref{normal_C} implies 
that $(\Spf R^C_{\Theta_z})^{\rig}$ is normal and hence reduced. It follows from \Cref{quasi_Stein_C} that $X^{C}_{\CG}$ is reduced. 
The rest of the argument is the same as in \Cref{main_L}.
\end{proof}

\begin{prop}\label{compatibility_C}
Let $\rho: \Gamma_F \rightarrow \CG(\OO(Y))$ 
be a continuous $C$-admissible representation, where $Y\in \Rig_{\Qp}$ and let $\Theta_{\rho} \in X^{C}_{\CG}(Y)$ be  its $\CG$-pseudocharacter. Then $\zeta^C_{\rho}$ associated to $\rho$ in \cite{DPS} coincides with $\zeta^C_{Y, \Theta_{\rho}}$.
\end{prop}
\begin{proof} The proof is the same as the proof of \Cref{compatibility_L}.
\end{proof} 

\section{Application to Hecke eigenspaces}\label{sec_Hecke}

We use our construction to generalise the results of \cite[Section 9.9]{DPS} concerning infinitesimal characters of the subspace 
of locally analytic vectors in the Hecke eigenspaces in the setting of definite unitary groups. Our construction 
of the infinitesimal character enables us to handle 
Hecke eigenspaces corresponding to Galois representations,
which are residually reducible. Apart from this new ingredient the argument is the same as in \cite{DPS}. We refer the reader to \Cref{explain} for more details. 

 We put ourselves in the setting of the proof of  \cite[Theorem 3.3.3]{FS}. In particular,  $F$ is a totally real field, $E$ is a totally imaginary quadratic extension of $F$,  
 $S$ is a finite set of places of $F$ containing all the places above $p$ and $\infty$, $S_E$ is the set of places of $E$ above $S$,  $E_S$ is the maximal extension of $E$ unramified outside $S_E$,
   $G$ is a unitary group over $F$ which is an outer form of $\GL_n$ with respect to the quadratic extension $E/F$ such that $G$ is quasi-split at all finite places and anisotropic at all infinite places. We assume that all the places of $F$ above $p$ split in $E$. This implies that $G(F\otimes_{\QQ} \Qp)$ 
is isomorphic to a product of $\GL_n(F_v)$ for $v\mid p$. Let $K_p$ be 
the subgroup of $G(F\otimes_{\QQ} \Qp)$ mapped to $\prod_{v\mid p} \GL_n(\OO_{F_v})$ under this isomorphism. 
Let $U^p=\prod_{v\nmid p\infty} U_v$ be a compact  open subgroup of $G(\mathbb A^{\infty, p}_F)$. If $U_p$ is a compact open subgroup  of $G(F\otimes_{\QQ} \Qp)$ then the double coset 
  $$Y(U^p U_p):= G(F)\backslash G(\mathbb A_F)/ U^p U_p G(F \otimes_{\QQ} \RR)^{\circ}$$
is a finite set. The  completed cohomology with $\Zp$-coefficients and prime-to-$p$ level $U^p$ is defined as
$$ \Htilde(U^p):= \varprojlim_{m} \varinjlim_{U_p} H^0(Y(U^p U_p), \ZZ/p^m \ZZ),$$
where the inner limit is taken over all open compact subgroups $U_p$. We may identify  $\Htilde(U^p)$ (resp. $\Htilde(U^p)_{\Qp}$)  with the space of continuous $\Zp$-valued (resp. $\Qp$-valued) functions on the profinite set $$Y(U^p):=G(F)\backslash G(\mathbb A_F)/ U^p  G(F \otimes_{\QQ} \RR)^{\circ}\cong\varprojlim_{U_p} Y(U^p U_p).$$
 The action of $G(F\otimes_{\QQ} \Qp)$ on $Y(U^p)$ makes $\Htilde(U^p)_{\Qp}$ into an admissible unitary 
$\Qp$-Banach space representation of $G(F\otimes_{\QQ} \Qp)$ with unit ball equal to $\Htilde(U^p)$. The group $K_p$ acts on $Y(U^p)$ with finitely many orbits. We may replace $U^p$ by an open subgroup such that the action of $K_p$ on $Y(U^p)$ is free, \cite[Remark 9.3]{DPS}. Then the restriction of $\Htilde(U^p)_{\Qp}$
to $K_p$ is isomorphic to a finite direct sum of copies of the space of 
continuous $\Qp$-valued functions on $K_p$ with the action of $K_p$ by right translations.

We will now introduce a variant of the big Hecke algebra, which acts on $\Htilde(U^p)$ by continuous, $G(F\otimes_{\QQ} \Qp)$-invariant  endomorphisms. Let $\mathbb T^S$ be the double coset algebra $\ZZ[U^S\backslash G(\mathbb A^S_F)/ U^S]$. Let $\Spl^S_F$ be the set of places of $F$ outside $S$ that split in the extension $E/F$, let $\Spl^S_E$
be the set of places of $E$ above $\Spl^S_F$. For each 
$w\in \Spl^S_E$ above a place $v$ of $F$ we fix an isomorphism $i_w: G(F_v)\cong \GL_n(E_w)$, which maps $U_v$ to $\GL_n(\OO_{E_w})$. Let $T^{(i)}_w\in \mathbb T^S$
be the Hecke operator corresponding to the double coset 
$\GL_n(\OO_{E_w}) \bigl(\begin{smallmatrix} \varpi_w I_i & 0 \\ 0 & I_{n-i}\end{smallmatrix}\bigr)  \GL_n(\OO_{E_w})$ via $i_w$, where $\varpi_w$ is a uniformiser of $E_w$.  Let $\mathbb T^S_{\Spl}$ be the subalgebra of $\mathbb T^S$ generated by $T_w^{(i)}$ for all $w\in \Spl^S_E$ and $1\le i \le n$.
The algebra $\mathbb T^S$ 
acts on $H^0(Y(U^p U_p), \ZZ/p^m \ZZ)$ and  we let $\mathbb T^S_{\mathrm{Spl}}(U^p U_p, \ZZ/p^m \ZZ)$ be the image of 
$\mathbb T^S_{\Spl}$ in $\End_{\ZZ}( H^0(Y(U^p U_p), \ZZ/p^m \ZZ))$. We define 
\begin{equation}\label{defnT}
\mathbb T^S_{\mathrm{Spl}}(U^p):= \varprojlim_{m, U_p} \mathbb T^S_{\mathrm{Spl}}(U^p U_p, \ZZ/p^m \ZZ).
\end{equation}
Since $\mathbb T^S_{\mathrm{Spl}}(U^p U_p, \ZZ/p^m \ZZ)$ are finite $\ZZ/p^m \ZZ$-modules, 
the algebra $\mathbb T^S_{\mathrm{Spl}}(U^p)$ is a profinite $\Zp$-algebra. It has 
only finitely many open maximal ideals by \cite[Lemma C.3]{FS}. If $\mm$ is an open maximal ideal of $\mathbb T^S_{\mathrm{Spl}}(U^p)$ we let  $\mathbb T^S_{\mathrm{Spl}}(U^p)_{\mm}$ 
be the $\mm$-adic completion of $\mathbb T^S_{\mathrm{Spl}}(U^p)$.
We have 
\begin{equation}\label{Ch_rem}
\Htilde(U^p)\cong \prod_{\mm} \Htilde(U^p)_{\mm}, \quad \mathbb T^S_{\mathrm{Spl}}(U^p)\cong \prod_{\mm} \mathbb T^S_{\mathrm{Spl}}(U^p)_{\mm},
\end{equation}
where the (finite) product is taken over all open maximal ideals of $\mathbb T^S_{\mathrm{Spl}}(U^p)$ and $\Htilde(U^p)_{\mm}:=\varprojlim_{m} \varinjlim_{U_p} H^0(Y(U^p U_p), \ZZ/p^m \ZZ)_{\mm}$.

\begin{prop} The algebra $\mathbb T^S_{\mathrm{Spl}}(U^p)$ is noetherian. Moreover, 
there is a continuous $n$-dimensional determinant law $D: \Zp[\Gal(E_S/E)]\rightarrow \mathbb T^S_{\mathrm{Spl}}(U^p)$ such that 
\begin{equation}\label{DFrob}
 D(1 +t \Frob_w)= \sum_{i=0}^n t^i N(w)^{i(i-1)/2} T_w^{(i)}, \quad \forall w\in \Spl^S_E,
\end{equation} 
where $N(w)$ is the absolute norm of the ideal $w$ and $\Frob_w\in \Gal(E_S/E)$ is 
a geometric Frobenius at $w$ and $t$ is an indeterminate.
\end{prop}
\begin{proof} \cite[Theorem C.10]{FS}.
\end{proof} 

Let $\Theta$ be a $\GL_n$-pseudocharacter attached to $D$ via \cite{emerson2023comparison}. The $L$-group of $\GL_n$ is $\GL_n$ itself.
We would like to have a pseudocharacter valued in the $C$-group of $\GL_n$. 
We do this via a twisting construction described in \cite[Section 4.7]{DPS}. 
Let $\tilde{\delta}: \Gm \rightarrow \GL_n$ be the cocharacter such that  $\tilde{\delta}(t)$
is a diagonal matrix with the diagonal entries given by $(1, t^{-1}, \ldots, t^{1-n})$. Recall that the $C$-group ${}^C\GL_n\cong \GL_n \rtimes \Gm$. The map 
\begin{equation}\label{twisting}
\tw_{\tilde{\delta}}: {}^C\GL_n \rightarrow \GL_n \times \Gm, \quad (g,t)\mapsto (g \tilde{\delta}(t), t)
\end{equation}
is an isomorphism of group schemes, \cite[Example 2 (3)]{zhu}. We consider $\Theta \boxtimes \chi_{\cyc}$ as
a $\GL_n \times \Gm$-pseudocharacter of $\Gal(E_S/E)$ valued in $\mathbb T^S_{\mathrm{Spl}}(U^p)$ and define $$\Theta^C:= \tw_{\tilde{\delta}}^{-1}\circ (\Theta\boxtimes \chi_\cyc).$$
Then $\Theta^C$ is a $C$-admissible ${}^C\GL_n$-pseudocharacter valued in $\mathbb T^S_{\mathrm{Spl}}(U^p)$. If $w$ is a place of $E$ above $p$ we denote by $\Theta^C_w$ the restriction of $\Theta^C$ to the decomposition group at $w$. 

Finally, we come to the application of the construction carried out in the previous sections. Let $\mathbb T^S_{\mathrm{Spl}}(U^p)^{\ana}$ be the ring of global sections
of $(\Spf \mathbb T^S_{\mathrm{Spl}}(U^p))^{\rig}$. For each $w\mid p$ \Cref{main_C} gives us 
a $\Qp$-algebra homomorphism 
\begin{equation}\label{zetaCw}
\zeta^C_{w}: Z(\Res_{E_w/\Qp} (\gl_{n, E_w}))\rightarrow \mathbb T^S_{\mathrm{Spl}}(U^p)^{\ana}
\end{equation}
corresponding to $Y=(\Spf \mathbb T^S_{\mathrm{Spl}}(U^p))^{\rig}$ and $\Theta=\Theta^C_w$.
For each place $v$ of $F$ above $p$ we fix a place $\tilde{v}$ of $E$ above $v$. 
Since, by assumption, all places 
$v$ of $F$ above $p$ split in $E$ we have an isomorphism
of groups $(\Res_{F/\QQ} G)_{\Qp} \cong \prod_{v\mid p}
\Res_{E_{\tilde{v}}/\Qp} \GL_{n, E_{\tilde{v}}}$. This induces an isomorphism of $\Qp$-algebras
\begin{equation}\label{places}
Z(\Res_{F/\QQ} \mathfrak g)_{\Qp} \cong \bigotimes_{v\mid p} Z(\Res_{E_{\tilde{v}}/\Qp} (\gl_{n, E_{\tilde{v}}})),
\end{equation}
where $\mathfrak g$ is the Lie algebra of $G$. We define
$\zeta^C: Z(\Res_{F/\QQ} \mathfrak g)_{\Qp} \rightarrow 
\mathbb T^S_{\mathrm{Spl}}(U^p)^{\ana}$ as the tensor product of $\zeta^C_w$ via \eqref{places}.

Let $x: \mathbb T^S_{\mathrm{Spl}}(U^p)\rightarrow \Qpbar$ be a $\Zp$-algebra homomorphism, let $\mm_x$ be its kernel and let $\kappa(x)$ be the residue field. 
Since $\mathbb T^S_{\mathrm{Spl}}(U^p)$ is a product of 
complete local noetherian $\Zp$-algebras, $\kappa(x)$ is a finite extension of $\Qp$. 
Let $\Htilde(U^p)_{\Qp}[\mm_x]$ be the subspace of 
$\Htilde(U^p)_{\Qp}$ of elements annihilated by all the 
elements in $\mm_x$ - the so-called Hecke eigenspace. 
Since  $\Htilde(U^p)_{\Qp}[\mm_x]$ is closed in $\Htilde(U^p)_{\Qp}$ and is preserved by the action 
of $G(F\otimes_{\QQ} \Qp)$, it is an admissible unitary
$\kappa(x)$-Banach space representation of  $G(F\otimes_{\QQ} \Qp)$. Let $\Htilde(U^p)_{\Qp}[\mm_x]^{\la}$ be the subspace of locally analytic vectors for the action of the $p$-adic 
Lie group $G(F\otimes_{\QQ} \Qp)$. By differentiating 
the action one obtains an action of the Lie algebra and
an action of $Z(\Res_{F/\QQ} \mathfrak g)_{\Qp}$ on 
$\Htilde(U^p)_{\Qp}[\mm_x]^{\la}$.

\Cref{deJong_like} implies that 
$x$ factors through 
$\mathbb T^S_{\mathrm{Spl}}(U^p)\rightarrow 
\mathbb T^S_{\mathrm{Spl}}(U^p)^{\ana}
\rightarrow \kappa(x)\rightarrow \Qpbar$.
We let $\zeta^C_x: Z(\Res_{F/\QQ} \mathfrak g)_{\Qp}\rightarrow \kappa(x)$
denote the composition of $\zeta^C$ with the induced map $\mathbb T^S_{\mathrm{Spl}}(U^p)^{\ana}
\rightarrow \kappa(x)$.

\begin{thm}\label{Hecke} The action of $Z(\Res_{F/\QQ} \mathfrak g)_{\Qp}$ on $\Htilde(U^p)_{\Qp}[\mm_x]^{\la}$
is given by $\zeta^C_x$.
\end{thm} 
\begin{proof} The proof is the same as the proof of 
\cite[Theorem  9.20]{DPS}, which is proved in a more general setting and specialised to the setting of definite unitary groups in 
\cite[Section 9.9]{DPS}, so we only give a sketch. 

The action of 
$\mathbb T^S_{\mathrm{Spl}}(U^p)$ on $\Htilde(U^p)_{\Qp}$ induces an action of $\mathbb T^S_{\mathrm{Spl}}(U^p)^{\ana}$ on the locally analytic
vectors $\Htilde(U^p)_{\Qp}^{\la}$. We then 
have two actions of $Z(\Res_{F/\QQ} \mathfrak g)_{\Qp}$
on $\Htilde(U^p)_{\Qp}^{\la}$: one via $\zeta^C$ and the 
other coming from the action of $G(F\otimes_{\QQ} \Qp)$.
It is enough to show that these actions coincide. 
This is verified by first showing that the two actions 
coincide on a subspace contained in locally algebraic 
vectors in $\Htilde(U^p)_{\Qp}^{\la}$, by relating 
this subspace to classical automorphic forms and using 
the local-global compatibility at $p$ of Galois representations associated to these automorphic forms. 

More precisely, let $L$ be a sufficiently large finite extension of $\Qp$, and let $\tau_p$ be a smooth irreducible representation of $K_p$ on an $L$-vector space, which is a supercuspidal type. Let $\Irr$ be the
set of isomorphism classes of irreducible algebraic 
representations of $(\Res_{F/\QQ} G)_L$. 
If $[V]\in \Irr$ then 
we may view 
$V$ as a representation of $K_p$ 
on an $L$-vector space via 
$K_p\subset (\Res_{F/\QQ} G)(\Qp)\subset (\Res_{F/\QQ} G)(L)$. Since $\Htilde(U^p)_{\Qp}$ is admissible, 
$\Hom_{K_p}(\tau_p\otimes_L V, \Htilde(U^p)_{L})$
is a finite dimensional $L$-vector space on which $\mathbb T^S_{\Spl}(U^p)$ acts. This action is semi-simple by 
\cite[Lemma 9.18]{DPS}.

Let $y: \mathbb T^S_{\Spl}(U^p)\rightarrow \Qpbar$ be a homomorphism of $\Zp$-algebras, such that 
$$\Hom_{K_p}(\tau_p\otimes_L V, \Htilde(U^p)_{L})[\mm_y]\neq 0.$$ 
We claim that $\zeta^C_y$ is equal to the 
infinitesimal character of the algebraic representation $V$. The claim implies that 
the two actions of $Z(\Res_{F/\QQ} \mathfrak g)_{\Qp}$ coincide on the image of the 
evaluation map 
\begin{equation}
\bigoplus_{[V]\in \Irr} 
\Hom_{K_p}(\tau_p\otimes_L V, \Htilde(U^p)_{L})\otimes_L (\tau_p\otimes_L V) \rightarrow \Htilde(U^p)_{L}^{\la}.
\end{equation}
Since this subspace is dense in $\Htilde(U^p)_{L}^{\la}$ by \cite[Corollary 7.7]{DPS}, the two actions coincide.  

To prove the claim we fix an isomorphism $\iota: \Qpbar\cong \mathbb C$. It follows from \cite[Lemma 9.18]{DPS} that there exists a cuspidal 
automorphic representation $\pi_y = \otimes'_v \pi_{y, v}$ of $G(\mathbb A_F)$ 
such that $ \mathbb T^S_{\Spl}(U^p)$ acts on 
$\pi_y^{U^S}$ via $\iota\circ y$. Let 
$\rho_y: \Gal(E_S/E)\rightarrow \GL_n(\Qpbar)$ 
be the Galois representation associated to $\pi_y$.
The local-global compatibility
between $\pi_{y}$ and $\rho_y$ is summarised in \cite[Theorem 7.2.1]{EGH}, 
see the references in its proof for proper attributions. In particular, it follows from 
\eqref{DFrob} and \cite[Theorem 7.2.1 (iii)]{EGH}
$$\det(1 + t \rho_y(\Frob_w))= D_y(1+ t \Frob_w), \quad \forall w\in \Spl_E^S.$$ 
Since the union of conjugacy classes of $\Frob_w$ for $w\in \Spl_E^S$ is dense in 
$\Gal(E_S/E)$ we deduce that $D_y$ is the 
determinant law associated to $\rho_y$. 
The representation $\rho_y$ is de Rham 
at all places $w\mid p$ and the Hodge--Tate 
weights can be expressed in terms of the 
highest weight of $V$ by \cite[Theorem 7.2.1 (ii)]{EGH}. Let $\rho^C_y: \Gal(E_S/E)\rightarrow {}^C\GL_n(\Qpbar)$ be
the representation $\tw_{\tilde{\delta}}^{-1}\circ(\rho_y\boxtimes \chi_{\cyc})$ obtained via the twisting construction \eqref{twisting}. 
Then $\Theta^C_y$ is ${}^C\GL_n$-pseudocharacter of $\rho^C_y$.
It follows from \cite[Proposition 5.5]{DPS} that $\zeta^C_{\rho^C_y}$ is equal to the 
infinitesimal character of $V$. It follows from the specialisation property in \Cref{main_C} that $\zeta^C_y=
\zeta^C_{\rho^C_y}$, which proves the claim.
\end{proof}

\begin{remar}\label{explain} Let us emphasise that the definition  
of $\zeta^C_{w}$ in \eqref{zetaCw} is the only new thing in this section. 
In \cite{DPS} an analog of 
\Cref{Hecke}  is 
proved for $\Htilde(U^p)_{\mm}$ instead 
of $\Htilde(U^p)$, under the assumption that 
the specialisation of $\Theta$ at $\mm$ is absolutely irreducible. 
This assumption implies the existence of Galois representation
$\rho: \Gal(E_S/E)\rightarrow \GL_n(\mathbb T^S_{\Spl}(U^p)_{\mm})$
such that its pseudocharacter is equal to $\Theta$. The authors 
in \cite{DPS} then use the twisting construction to 
obtain a representation $\rho^C: \Gal(E_S/E)\rightarrow {}^C\GL_n(\mathbb T^S_{\Spl}(U^p)_{\mm})$, and 
work with the infinitesimal character $\zeta^C_{\rho^C}$ instead of 
$\zeta^C$. Our construction of the infinitesimal character $\zeta^C$ enables us to carry 
out the same argument on the 
whole of $\Htilde(U^{p})_{\Qp}$, not just the part, where the residual Galois representation is irreducible.
\end{remar}

\bibliographystyle{plain}
\bibliography{Ref}

\begin{thebibliography}{10}

\bibitem{alper}
Jarod Alper.
\newblock Adequate moduli spaces and geometrically reductive group schemes.
\newblock {\em Algebr. Geom.}, 1(4):489--531, 2014.

\bibitem{bartenwerfer}
Wolfgang Bartenwerfer.
\newblock Der erste {R}iemannsche {H}ebbarkeitssatz im nichtarchimedischen
  {F}all.
\newblock {\em J. Reine Angew. Math.}, 286/287:144--163, 1976.

\bibitem{berkovich_ihes}
Vladimir~G. Berkovich.
\newblock {\'E}tale cohomology for non-{Archimedean} analytic spaces.
\newblock {\em Publ. Math., Inst. Hautes {\'E}tud. Sci.}, 78:5--161, 1993.

\bibitem{BHKT}
Gebhard B\"{o}ckle, Michael Harris, Chandrashekhar Khare, and Jack~A. Thorne.
\newblock {$\hat G$}-local systems on smooth projective curves are potentially
  automorphic.
\newblock {\em Acta Math.}, 223(1):1--111, 2019.

\bibitem{BIP_new}
Gebhard B\"{o}ckle, Ashwin Iyengar, and Vytautas Pa\v{s}k\={u}nas.
\newblock On local {G}alois deformation rings.
\newblock {\em Forum Math. Pi}, 11:Paper No. e30, 54, 2023.

\bibitem{borel}
Armand Borel.
\newblock {\em Linear algebraic groups}, volume 126 of {\em Graduate Texts in
  Mathematics}.
\newblock Springer-Verlag, New York, second edition, 1991.

\bibitem{BGR}
S.~Bosch, U.~G\"untzer, and R.~Remmert.
\newblock {\em Non-{A}rchimedean analysis}, volume 261 of {\em Grundlehren der
  mathematischen Wissenschaften [Fundamental Principles of Mathematical
  Sciences]}.
\newblock Springer-Verlag, Berlin, 1984.
\newblock A systematic approach to rigid analytic geometry.

\bibitem{bosch}
Siegfried Bosch.
\newblock {\em Lectures on formal and rigid geometry}, volume 2105 of {\em
  Lecture Notes in Mathematics}.
\newblock Springer, Cham, 2014.

\bibitem{Bourbaki_Lie46}
Nicolas Bourbaki.
\newblock {\em Elements of mathematics. {Lie} groups and {Lie} algebras.
  {Chapters} 4--6. {Transl}. from the {French} by {Andrew} {Pressley}}.
\newblock Berlin: Springer, paperback reprint of the hardback edition 2002
  edition, 2008.

\bibitem{BH}
Winfried Bruns and J\"{u}rgen Herzog.
\newblock {\em Cohen--{M}acaulay rings}, volume~39 of {\em Cambridge Studies in
  Advanced Mathematics}.
\newblock Cambridge University Press, Cambridge, 1993.

\bibitem{BG}
Kevin Buzzard and Toby Gee.
\newblock The conjectural connections between automorphic representations and
  {G}alois representations.
\newblock In {\em Automorphic forms and {G}alois representations. {V}ol. 1},
  volume 414 of {\em London Math. Soc. Lecture Note Ser.}, pages 135--187.
  Cambridge Univ. Press, Cambridge, 2014.

\bibitem{che_durham}
Ga\"{e}tan Chenevier.
\newblock The {$p$}-adic analytic space of pseudocharacters of a profinite
  group and pseudorepresentations over arbitrary rings.
\newblock In {\em Automorphic forms and {G}alois representations. {V}ol. 1},
  volume 414 of {\em London Math. Soc. Lecture Note Ser.}, pages 221--285.
  Cambridge Univ. Press, Cambridge, 2014.

\bibitem{conrad}
Brian Conrad.
\newblock Irreducible components of rigid spaces.
\newblock {\em Ann. Inst. Fourier (Grenoble)}, 49(2):473--541, 1999.

\bibitem{conrad_ample}
Brian Conrad.
\newblock Relative ampleness in rigid geometry.
\newblock {\em Ann. Inst. Fourier (Grenoble)}, 56(4):1049--1126, 2006.

\bibitem{deJong}
A.~J. de~Jong.
\newblock Crystalline {Dieudonn{\'e}} module theory via formal and rigid
  geometry.
\newblock {\em Publ. Math., Inst. Hautes {\'E}tud. Sci.}, 82:5--96, 1995.

\bibitem{DPS}
Gabriel Dospinescu, Vytautas Pa{\v{s}}k{\=u}nas, and Benjamin Schraen.
\newblock Infinitesimal characters in arithmetic families.
\newblock {\em Sel. Math., New Ser.}, 31(4):77, 2025.
\newblock Id/No 76.

\bibitem{emerson2023comparison}
Kathleen Emerson and Sophie Morel.
\newblock Comparison of different definitions of pseudocharacters.
\newblock \url{https://arxiv.org/abs/2310.03869}, 2023.

\bibitem{padic_LL}
Matthew Emerton, Toby Gee, and Eugen Hellmann.
\newblock An introduction to the categorical $p$-adic {L}anglands program.
\newblock {\em arXiv preprint}, 2023.
\newblock \url{https://arxiv.org/abs/2210.01404}.

\bibitem{EGH}
Matthew Emerton, Toby Gee, and Florian Herzig.
\newblock Weight cycling and {S}erre-type conjectures for unitary groups.
\newblock {\em Duke Math. J.}, 162(9):1649--1722, 2013.

\bibitem{FS}
Jessica Fintzen and Sug~Woo Shin.
\newblock Congruences of algebraic automorphic forms and supercuspidal
  representations (with appendices by {Rapha{\"e}l} {Beuzart}-{Plessis} and
  {Vytautas} {Pa{\v{s}}k{\=u}nas}).
\newblock {\em Camb. J. Math.}, 9(2):351--429, 2021.

\bibitem{GW1}
Ulrich G\"ortz and Torsten Wedhorn.
\newblock {\em Algebraic geometry {I}. {S}chemes---with examples and
  exercises}.
\newblock Springer Studium Mathematik---Master. Springer Spektrum, Wiesbaden,
  second edition, [2020] \copyright 2020.

\bibitem{GW2}
Ulrich G{\"o}rtz and Torsten Wedhorn.
\newblock {\em Algebraic geometry {II}: cohomology of schemes. {With} examples
  and exercises}.
\newblock Springer Stud. Math. -- Master. Wiesbaden: Springer Spektrum, 2023.

\bibitem{EGA4_4}
A.~Grothendieck.
\newblock {\'E}l{\'e}ments de g{\'e}om{\'e}trie alg{\'e}brique. {IV}: {\'E}tude
  locale des sch{\'e}mas et des morphismes de sch{\'e}mas ({Quatri{\`e}me}
  partie). {R{\'e}dig{\'e}} avec la colloboration de {J}. {Dieudonn{\'e}}.
\newblock {\em Publ. Math., Inst. Hautes {\'E}tud. Sci.}, 32:1--361, 1967.

\bibitem{SGA1}
A.~Grothendieck, editor.
\newblock {\em S{\'e}minaire de g{\'e}om{\'e}trie alg{\'e}brique du {Bois}
  {Marie} 1960-61. {Rev{\^e}tements} {\'e}tales et groupe fondamental ({SGA}
  1). {Un} s{\'e}minaire dirig{\'e} par {Alexander} {Grothendieck}.
  {Augment{\'e}} de deux expos{\'e}s de {M}. {Raynaud}.}, volume~3 of {\em Doc.
  Math. (SMF)}.
\newblock Paris: Soci{\'e}t{\'e} Math{\'e}matique de France, {\'e}dition
  recompos{\'e}e et annot{\'e}e du original publi{\'e} en 1971 par {Springer}
  edition, 2003.

\bibitem{heinrich}
Katharina Heinrich.
\newblock Some remarks on biequidimensionality of topological spaces and
  {N}oetherian schemes.
\newblock {\em J. Commut. Algebra}, 9(1):49--63, 2017.

\bibitem{KW2}
Chandrashekhar Khare and Jean-Pierre Wintenberger.
\newblock Serre's modularity conjecture. {II}.
\newblock {\em Invent. Math.}, 178(3):505--586, 2009.

\bibitem{Koepf}
Ursula K\"opf.
\newblock {\"U}ber eigentliche {F}amilien algebraischer {V}ariet\"aten \"uber
  affinoiden {R}\"aumen.
\newblock {\em Schr. Math. Inst. Univ. M\"unster (2)}, pages iv+72, 1974.

\bibitem{lutkebohmert}
Werner L\"utkebohmert.
\newblock On extension of rigid analytic objects.
\newblock {\em M\"unster J. Math.}, 15(1):83--166, 2022.

\bibitem{milne_alg}
J.~S. Milne.
\newblock {\em Algebraic groups. {The} theory of group schemes of finite type
  over a field}, volume 170 of {\em Camb. Stud. Adv. Math.}
\newblock Cambridge: Cambridge University Press, 2017.

\bibitem{MRL}
Vytautas Pa{\v{s}}k{\=u}nas and Julian Quast.
\newblock Deformations of pseudocharacters and {M}azur's finiteness condition.
\newblock \url{https://arxiv.org/abs/2506.10901}, 2026.

\bibitem{defG}
Vytautas Pa{\v{s}}k{\=u}nas and Julian Quast.
\newblock On local {Galois} deformation rings: generalised reductive groups.
\newblock {\em Forum Math. Pi}, 14:Paper No. e15, 2026.

\bibitem{defT}
Vytautas Pa\v{s}k\={u}nas and Julian Quast.
\newblock On local {G}alois deformation rings: generalised tori.
\newblock {\em Forum Math. Sigma}, 13:Paper No. e45, 2025.

\bibitem{quast}
Julian Quast.
\newblock Deformations of {$G$}-valued pseudocharacters.
\newblock {\em Peking Mathematical Journal}, 2026.

\bibitem{seshadri}
C.~S. Seshadri.
\newblock Geometric reductivity over arbitrary base.
\newblock {\em Adv.\,Math.}, 26(3):225--274, 1977.

\bibitem{sikora}
Adam~S. Sikora.
\newblock Character varieties.
\newblock {\em Trans. Amer. Math. Soc.}, 364(10):5173--5208, 2012.

\bibitem{stacks-project}
The {Stacks Project Authors}.
\newblock \textit{Stacks Project}.
\newblock \url{https://stacks.math.columbia.edu}, 2022.

\bibitem{WE_alg}
Carl Wang-Erickson.
\newblock Algebraic families of {G}alois representations and potentially
  semi-stable pseudodeformation rings.
\newblock {\em Math. Ann.}, 371(3-4):1615--1681, 2018.

\bibitem{zhu}
Xinwen Zhu.
\newblock A note on integral {S}atake isomorphisms.
\newblock In {\em Arithmetic geometry}, volume~41 of {\em Tata Inst. Fundam.
  Res. Stud. Math.}, pages 469--489. Tata Inst. Fund. Res., Mumbai, [2024]
  \copyright 2024.

\end{thebibliography}
\end{document}